\documentclass{amsart}

\usepackage{amsmath,amssymb,amscd,amsxtra,soul,mathtools}

\usepackage{hyperref}

\usepackage{esint}

\usepackage[shortlabels]{enumitem}
\usepackage[capitalise]{cleveref}
\crefrangeformat{enumi}{items #3#1#4--#5#2#6}

\usepackage{graphics} 
\usepackage{epsfig} 
\usepackage{graphicx}  
\usepackage{epstopdf}

\usepackage[toc,page]{appendix}

\hypersetup{colorlinks=false}

\usepackage{orcidlink}

\numberwithin{equation}{section}

\newtheorem{thm}{Theorem}[section]
\newtheorem{lem}[thm]{Lemma}
\newtheorem{cor}[thm]{Corollary}
\newtheorem{prop}[thm]{Proposition}

\theoremstyle{definition}

\theoremstyle{remark}

\newtheorem{rem}[thm]{Remark}
\newtheorem{notation}[thm]{Notation}
\newtheorem{claim}{Claim}

\crefname{thm}{Theorem}{Theorems}
\crefname{lem}{Lemma}{Lemmas}
\crefname{cor}{Corollary}{Corollaries}
\crefname{prop}{Proposition}{Propositions}
\crefname{mainthm}{Theorem}{Theorems}
\crefname{maincor}{Corollary}{Corollaries}

\crefname{defn}{Definition}{Definitions}
\crefname{conj}{Conjecture}{Conjectures}
\crefname{ex}{Example}{Examples}
\crefname{exs}{Examples}{Examples}
\crefname{prob}{Problem}{Problems}
\crefname{quest}{Question}{Questions}

\crefname{rem}{Remark}{Remarks}
\crefname{claim}{Claim}{Claims}
\crefname{case}{case}{cases}
\crefname{hyp}{Hypothesis}{Hypotheses}
\crefname{notation}{Notation}{Notations}

\newcommand{\C}{\mathbb{C}}
\newcommand{\K}{\mathbb{K}}
\newcommand{\N}{\mathbb{N}}

\newcommand{\R}{\mathbb{R}}
\renewcommand{\S}{\mathbb{S}}
\newcommand{\Z}{\mathbb{Z}}

\newcommand{\HH}{\mathcal{H}}

\newcommand{\LL}{\mathcal{L}}
\newcommand{\MM}{\mathcal{M}}
\newcommand{\OO}{\mathcal{O}}

\renewcommand{\SS}{\mathcal{S}}
\newcommand{\TT}{\mathcal{T}}

\newcommand{\XX}{\mathcal{X}}
\newcommand{\YY}{\mathcal{Y}}

\newcommand{\fa}{\mathfrak{a}}
\newcommand{\fe}{\mathfrak{e}}

\newcommand{\fX}{\mathfrak{X}}

\newcommand{\bfd}{\mathbf{d}}
\newcommand{\bfe}{\mathbf{e}}
\newcommand{\bff}{\mathbf{f}}
\newcommand{\bfg}{\mathbf{g}}

\newcommand{\bfz}{\mathbf{z}}

\newcommand{\bfC}{\mathbf{C}}
\newcommand{\bfD}{\mathbf{D}}

\newcommand{\bfdelta}{\boldsymbol{\delta}}
\newcommand{\bfDelta}{\boldsymbol{\Delta}}

\newcommand{\bfPi}{\boldsymbol{\Pi}}

\newcommand{\sw}{\mathsf{w}}

\newcommand{\sH}{\mathsf{H}}
\newcommand{\sJ}{\mathsf{J}}

\newcommand{\sN}{\mathsf{N}}

\newcommand{\supp}{\operatorname{supp}}

\newcommand{\im}{\operatorname{im}}

\newcommand{\Cl}{\operatorname{Cl}}

\newcommand{\Tr}{\operatorname{Tr}}
\newcommand{\tr}{\operatorname{tr}}

\newcommand{\rank}{\operatorname{rank}}

\newcommand{\Hess}{\operatorname{Hess}}
\newcommand{\spec}{\operatorname{spec}}
\newcommand{\ind}{\operatorname{ind}}
\newcommand{\Crit}{\operatorname{Crit}}
\newcommand{\Zero}{\operatorname{Zero}}
\newcommand{\grad}{\operatorname{grad}}

\newcommand{\Str}{\operatorname{Str}}
\newcommand{\str}{\operatorname{str}}

\newcommand{\Cinftyc}{C^\infty_{\text{\rm c}}}

\newcommand{\betaNo}{\beta_{\text{\rm No}}}

\newcommand{\zetasm}{\zeta_{\text{\rm sm}}}
\newcommand{\zetala}{\zeta_{\text{\rm la}}}
\newcommand{\thetasm}{\theta_{\text{\rm sm}}}
\newcommand{\thetala}{\theta_{\text{\rm la}}}
\newcommand{\dvol}{\operatorname{dvol}}

\usepackage{color}
\definecolor{darkgreen}{cmyk}{1,0,1,.2}
\definecolor{m}{rgb}{1,0.1,1}

\newdimen\theight
\def\TeXref#1{%
             \leavevmode\vadjust{\setbox0=\hbox{{\tt
                     \quad\quad  {\small \textrm #1}}}%
             \theight=\ht0
             \advance\theight by \lineskip
             \kern -\theight \vbox to
             \theight{\rightline{\rlap{\box0}}%
             \vss}%
             }}%

\title{Zeta invariants of Morse forms}

\author[J.A. \'Alvarez L\'opez]{Jes\'us A. \'Alvarez L\'opez\,\orcidlink{0000-0001-6056-2847}}
\email{jesus.alvarez@usc.es}
\address{Department of Mathematics and CITMAga\\
         University of Santiago de Compostela\\
         15782 Santiago de Compostela\\ Spain}

\author[Y.A. Kordyukov]{Yuri A. Kordyukov\,\orcidlink{0000-0003-2957-2873}}
\email{yurikor@matem.anrb.ru}
\address{Institute of Mathematics\\ Ufa Federal Research Center\\
Russian Academy of Sciences\\
112 Chernyshevsky street\\ 450008 Ufa\\ Russia}

\author[E. Leichtnam]{Eric Leichtnam\,\orcidlink{0000-0002-5058-5508}}
\email{eric.leichtnam@imj-prg.fr}
\address{Institut de Math\'ematiques de Jussieu-PRG\\ CNRS\\ Batiment Sophie Germain (bureau 740)\\ Case~7012\\ 75205 Paris Cedex 13, France}

\thanks{The authors are partially supported by the grants MTM2017-89686-P and PID2020-114474GB-I00 (AEI/FEDER, UE) and ED431C 2019/10 (Xunta de Galicia, FEDER)}

\date{\today}



\begin{document}

\begin{abstract}
Let $\eta$ be a closed real 1-form on a closed Riemannian $n$-manifold $(M,g)$. Let $d_z$, $\delta_z$ and $\Delta_z$ be the induced Witten's type perturbations of the de~Rham derivative and coderivative and the Laplacian, parametrized by $z=\mu+i\nu\in\mathbb C$ ($\mu,\nu\in\mathbb{R}$, $i=\sqrt{-1}$). Let $\zeta(s,z)$ be the zeta function of $s\in\mathbb{C}$, defined as the meromorphic extension of the function $\zeta(s,z)=\operatorname{Str}({\eta\wedge}\,\delta_z\Delta_z^{-s})$ for $\Re s\gg0$. We prove that $\zeta(s,z)$ is smooth at $s=1$ and establish a formula for $\zeta(1,z)$ in terms of the associated heat semigroup. For a class of Morse forms, $\zeta(1,z)$ converges to some $\mathbf{z}\in\mathbb{R}$ as $\mu\to+\infty$, uniformly on $\nu$. We describe $\mathbf{z}$ in terms of the instantons of an auxiliary Smale gradient-like vector field $X$ and the Mathai-Quillen current on $TM$ defined by $g$. Any real 1-cohomology class has a representative $\eta$ satisfying the hypothesis. If $n$ is even, we can prescribe any real value for $\mathbf{z}$ by perturbing $g$, $\eta$ and $X$, and achieve the same limit as $\mu\to-\infty$. This is used to define and describe certain tempered distributions induced by $g$ and $\eta$. These distributions appear in another publication as contributions from the preserved leaves in a trace formula for simple foliated flows, giving a solution to a problem stated by C.~Deninger.
\end{abstract}

\maketitle

\tableofcontents

\section{Introduction}\label{s: intro}

\subsection{Witten's perturbed operators}\label{ss: intro-Witten}

Let $M$ be  a closed $n$-manifold. For any smooth function $h$ on $M$, Witten \cite{Witten1982} introduced a perturbed de~Rham differential operator $d_\mu=d+\mu\,{dh\wedge}$, depending on a parameter $\mu\in\R$. Endowing $M$ with a Riemannian metric $g$, we have a corresponding perturbed codifferential operator $\delta_\mu=\delta-\mu\,{dh\lrcorner}$, and a perturbed Laplacian $\Delta_\mu=d_\mu\delta_\mu+\delta_\mu d_\mu$. Since $d_\mu=e^{-\mu h}\,d\,e^{\mu h}$, it defines the same Betti numbers as $d$. However $\Delta_\mu$ and the usual Laplacian $\Delta$ have different spectrum in general. In fact, if $h$ is a Morse function and $g$ is Euclidean with respect to Morse coordinates around the critical points, then the spectrum of $\Delta_\mu$ develops a long gap as $\mu\to+\infty$, giving rise to the small and large spectrum. The eigenforms of the small/large eigenvalues generate the small/large subcomplex, $(E_{\mu,\text{\rm sm/la}},d_\mu)$. When $h$ is a Morse function, Witten gave a beautiful analytic proof of the Morse inequalities by analyzing the small spectrum. This was refined by subsequent work of Helffer and Sj\"ostrand \cite{HelfferSjostrand1985} and Bismut and Zhang \cite{BismutZhang1992,BismutZhang1994}, showing that, if moreover $X:=-\grad h$ is a Smale vector field, then the Morse complex $(\bfC^\bullet,\bfd)$ of $X$ can be considered as the limit of $(E_{\mu,\text{\rm sm}},d_\mu)$. More precisely, for certain perturbed Morse complex $(\bfC^\bullet,\bfd_\mu)$, isomorphic to $(\bfC^\bullet,\bfd)$, there is a quasi-isomorphism $\Phi_\mu:(E_{z,\text{\rm sm}},d_\mu)\to(\bfC^\bullet,\bfd_\mu)$, defined by integration on the unstable cells of the zero points of $X$, which becomes an isomorphism for $\mu\gg0$ and almost isometric as $\mu\to+\infty$ (after rescaling at every degree).

We can replace $dh$ with any closed real 1-form $\eta$, obtaining a generalization of the Witten's perturbations, $d_\mu$, $\delta_\mu$ and $\Delta_\mu$. Now $d_\mu$ need not be gauge equivalent to $d$, obtaining new twisted Betti numbers $\beta^k_\mu$. However the numbers $\beta^k_\mu$ have well defined ground values $\betaNo^k$, called the Novikov numbers, which depend upon the de~Rham cohomology class $[\eta]\in H^1(M,\R)$. Assume that:
\begin{enumerate}[{\rm(a)}]
\item\label{i-a:  eta Morse} $\eta$ is a Morse form (it has Morse-type zeros), and $g$ is Euclidean with respect to Morse coordinates around the zero points of $\eta$.
\end{enumerate} 
(Some concepts used in this section are recalled in \Cref{ss: Morse forms,ss: Smale}.) Then $\Delta_\mu$ also develops a long gap separating a small spectrum and a large spectrum, and the analysis of the small spectrum gives Morse inequalities for the Novikov numbers. Take any auxiliary vector field $X$ such that:
\begin{enumerate}[{\rm(a)}]
\setcounter{enumi}{1}
\item\label{i-a:  X ...} $X$ has Morse-type zeros, and is gradient-like and Smale; and
\item\label{i-a:  eta Lyapunov} $\eta$ is Lyapunov for $X$, and $\eta$ and $g$ are in standard form with respect to $X$.
\end{enumerate}  
Then the small complex approaches a perturbed Morse complex of $X$. We refer to work by Novikov \cite{Novikov1981,Novikov1982}, Pajitnov \cite{Pajitnov1987}, Braverman and Farber~\cite{BravermanFarber1997}, Burghelea and Haller  \cite{BurgheleaHaller2001,BurgheleaHaller2004,BurgheleaHaller2008}, and Harvey and Minervini \cite{HarveyMinervini2006,Minervini2015}.

We can similarly define the perturbation $d_z=d+z\,{\eta\wedge}$ with parameter $z=\mu+i\nu\in\C$ ($\mu,\nu\in\R$ and $i=\sqrt{-1}$). Its adjoint is $\delta_z=\delta-\bar z\,{\eta\lrcorner}$, and we have a corresponding perturbed Laplacian $\Delta_z=d_z\delta_z+\delta_z d_z$. As a first step in our study, we prove extensions of the above results to this case, taking limits as $|\mu|\to+\infty$, uniformly on $\nu$. First, assuming~\ref{i-a:  eta Morse}, we get the long gap in the spectrum of $\Delta_z$ separating the small and large spectrum, which depends only on $\mu$ (\Cref{t: spec Delta_z}). Second, assuming~\ref{i-a:  eta Morse}--\ref{i-a:  eta Lyapunov}, we show that the quasi-isomorphism $\Phi_z:(E_{z,\text{\rm sm}},d_z)\to(\bfC^\bullet,\bfd_z)$ becomes an isomorphism for $|\mu|\gg0$ and almost isometric as $|\mu|\to+\infty$ (\Cref{t: Phi_z+tau Psi_z}). To get that the convergence is uniform on $\nu$, the key ingredient is a version of a Sobolev inequality for integers $m>n/2$: on smooth complex differential forms,
\begin{equation}\label{| |_L^infty le C_m | |_m i nu}
\|\ \|_{L^\infty}\le C_m\|\ \|_{m,i\nu}\;,
\end{equation}
where $C_m>0$ is independent of $\nu$ and  $\|\alpha\|_{m,i\nu}=\sum_{k=0}^m\langle\Delta_{i\nu}^k\alpha,\alpha\rangle^{1/2}$ (\Cref{p: perturbed Sobolev}). (The analogous property for $\Delta_\mu$ is wrong.) Then we adapt the arguments of Bismut and Zhang \cite{BismutZhang1992,BismutZhang1994} (see also \cite{Zhang2001}).

The indicated properties of $\Delta_z$, holding uniformly on $\mu$, depend on remarkable differences between $\Delta_{i\nu}$ and $\Delta_\mu$. For instance, if $\eta$ is exact, all operators $\Delta_{i\nu}$ are gauge equivalent, whereas this is not true for the operators $\Delta_\mu$ when $\eta\ne0$. If $\eta$ is not exact, the operators $\Delta_{i\nu}$ are not gauge equivalent either. Moreover $\Delta_{i\nu}-\Delta$ is of order one when $\nu\ne0$, whereas $\Delta_{\mu}-\Delta$ is of order zero.

\subsection{Zeta invariants of Morse forms}\label{ss: intro zeta}

To begin with, $\eta$ is only assumed to be an arbitrary closed real 1-form. Let $\Pi_z^\perp$ and $\Pi_z^1$ be the orthogonal projections to the images of $\Delta_z$ and $d_z$. We consider a zeta function $\zeta(s,z)$ associated with $\eta$ and the parameter $z\in\C$. As a function of $s\in\C$, it is the meromorphic extension of the holomorphic function
\[
\zeta(s,z)=\Str({\eta\wedge}\,\delta_z\Delta_z^{-s}\Pi_z^\perp)=\Str({\eta\wedge}\,d_z^{-1}\Delta_z^{-s+1}\Pi^1_z)
\]
defined for $\Re s\gg0$, where $\Str$ stands for the super-trace. We are interested in the zeta invariant $\zeta(1,z)$ that can be interpreted as a renormalization of the super-trace of ${\eta\wedge}\,d_z^{-1}\Pi^1_z$, which is not of trace class by the Weyl's law. According to the general theory of zeta functions of elliptic operators, $\zeta(s,z)$ might have a simple pole at $s=1$. However our first main theorem states that $\zeta(s,z)$ is smooth at $s=1$ and gives a formula for $\zeta(1,z)$ in terms of the associated heat semigroup.

\begin{thm}\label{t: zeta(1 z)}
Let $M\equiv(M,g)$ be a closed Riemannian $n$-manifold, and let $\eta$ be a closed real $1$-form on $M$. If $n$ is even {\rm(}resp., odd\/{\rm)}, then, for any $z\in\C$, $s\mapsto\zeta(s,z)$ is smooth on the half-plane $\Re s>0$ {\rm(}resp., $\Re s>1/2$\/{\rm)}. Furthermore
\[
\zeta(1,z)=\lim_{t\downarrow0}\Str\big({\eta\wedge}\,d_z^{-1}e^{-t\Delta_z}\Pi^1_z\big)\;.
\]
\end{thm}

The existence of the limit of \Cref{t: zeta(1 z)} is surprising because ${\eta\wedge}\,d_z^{-1}e^{-t\Delta_z}\Pi_z^1$ is weakly convergent to ${\eta\wedge}\,d_z^{-1}\Pi_z^1$.  An expression similar to $\Str({\eta\wedge}\,d_z^{-1}e^{-t\Delta_z}\Pi_z^1)$ was used by Mrowka, Ruberman and Saveliev to define a cyclic eta invariant \cite{MrowkaRubermanSaveliev2016}. 

Next, we additionally assume that $\eta$ is a Morse form and use the results described in the previous section. The zeta-function decomposes as the sum of terms defined by the contributions from the small/large spectrum, $\zeta_{\text{\rm sm/la}}(s,z)=\zeta_{\text{\rm sm/la}}(s,z,\eta)$, where $\zetasm(s,z)$ is an entire function of $s$. Our second main theorem describes the asymptotic behavior of $\zeta(1,z)$ as $\mu\to\pm\infty$, uniformly on $\nu$. In fact, since
\begin{equation}\label{zeta(s z eta) = -zeta(s -z -eta)}
\zeta(s,z,\eta)=-\zeta(s,-z,-\eta)\;,\quad\zeta_{\text{\rm sm/la}}(s,z,\eta)=-\zeta_{\text{\rm sm/la}}(s,-z,-\eta)\;,
\end{equation}
it is enough to consider the case where $\mu\gg0$ and take the limit as $\mu\to+\infty$.

We use the current $\psi(M,\nabla^M)$ of degree $n-1$ on $TM$ constructed by Mathai and Quillen in \cite{MathaiQuillen1986}, depending on the Levi-Civita connection $\nabla^M$. This current is smooth on the complement of the zero section, where it is given by the solid angle. It is also locally integrable, and its wave front set is contained in the conormal bundle in $T^* TM$ of the zero section of $T M$. Since this set does not meet the conormal bundle of the map $X: M \rightarrow T M$ (assuming~\ref{i-a:  X ...}), $(-X)^*\psi(M,\nabla^M)$ is well defined as a current on $M$. Assuming also~\ref{i-a:  eta Morse}--\ref{i-a:  eta Lyapunov}, consider the real number
\[
\bfz_{\text{\rm la}}=\bfz_{\text{\rm la}}(M,g,\eta)=\int_M\eta\wedge (-X)^*\psi(M,\nabla^M)\;,
\]
which is known to be independent of $X$ \cite[Proposition~6.1]{BismutZhang1992}.

Now suppose also that:
\begin{enumerate}[{\rm(a)}]
\setcounter{enumi}{3}
\item\label{i-a:  tight} for every zero point $p$ of $X$ with Morse index $k$, the maximum value of the integrals of $\eta$ along the instantons of $X$ with $\alpha$-limit $p$ only depends on $k$.
\end{enumerate} 
This maximum value is denoted by $-a_k$ for some $a_k>0$. Let $m^1_k=\dim d_z(E_{z,\text{\rm sm}}^{k-1})$ for $\mu\gg0$, which is independent of $z$. Consider also the real number
\[
\bfz_{\text{\rm sm}}=\bfz_{\text{\rm sm}}(M,g,\eta,X)=\sum_{k=1}^n(-1)^k\big(1-e^{a_k}\big)m^1_k\;,
\]
and let $\bfz=\bfz(M,g,\eta,X)=\bfz_{\text{\rm sm}}+\bfz_{\text{\rm la}}$.

Recall that we write $z= \mu + i \nu$.

\begin{thm}\label{t: zeta_sm/la(1 z)}
Let $M\equiv(M,g)$ be a closed Riemannian $n$-manifold, let $\eta$ be a closed real $1$-form on $M$ satisfying~\ref{i-a:  eta Morse}, and let $X$ be a vector field on $M$ satisfying~\ref{i-a:  X ...}--\ref{i-a:  eta Lyapunov}.
\begin{enumerate}[{\rm(i)}]
\item\label{i: zetala(1 z)} We have
\[
\zetala(1,z)=\bfz_{\text{\rm la}}+O(\mu^{-1})
\]
as $\mu\to+\infty$, uniformly on $\nu$.
\item\label{i: zetasm(1 z)} If moreover~\ref{i-a:  tight} holds, then
\[
\zetasm(1,z)=\bfz_{\text{\rm sm}}+O(\mu^{-1})
\]
as $\mu\to+\infty$, uniformly on $\nu$.
\end{enumerate}
\end{thm}

\Cref{t: zeta_sm/la(1 z)}~\ref{i: zetasm(1 z)} shows that $\bfz_{\text{\rm sm}}$ and $\bfz$ are also independent of $X$. Thus $X$ will be omitted in their notation. In the notation of $\bfz_{\text{\rm sm/la}}$ and $\bfz$, we may also omit $M$ or $g$ if they are fixed. 

By~\eqref{zeta(s z eta) = -zeta(s -z -eta)}, if we take $\mu\to-\infty$ in \Cref{t: zeta_sm/la(1 z)}, we have to replace $\bfz_{\text{\rm sm/la}}(\eta)$ with $-\bfz_{\text{\rm sm/la}}(-\eta)$. Descriptions of $-\bfz_{\text{\rm sm/la}}(-\eta)$ are given in~\eqref{-bfz_la(-eta)} and~\eqref{-bfz(M g -eta)}.

Our third main theorem is about the prescription of $\bfz=\bfz(M,g,\eta)$ without changing the cohomology class of $\eta$.

\begin{thm}\label{t: prescription of lim zeta(1 z)}
Let $M$ be a smooth closed $n$-manifold. If $n$ is even {\rm(}resp., odd\/{\rm)}, for all $\xi\in H^1(M,\R)$ and $\tau\in\R$ {\rm(}resp., $\tau\gg0$\/{\rm)}, there is some $\eta\in\xi$, a Riemannian metric $g$ and a vector field $X$ satisfying~\ref{i-a:  eta Morse}--\ref{i-a:  tight} such that  $\pm\bfz(M,g,\pm\eta)=\tau$ {\rm(}resp., $\bfz(M,g,\eta)=\tau${\rm)}.
\end{thm}

\subsection{A distribution associated to some Morse forms}\label{ss: intro - Z^pm}

A trace formula for simple foliated flows on closed foliated manifolds was conjectured by C.~ Deninger (see e.g.\ \cite{Deninger2008}). He was motivated by analogies with Weil’s explicit formulas in Arithmetics, and previous work of Guillemin and Sternberg \cite{Guillemin1977}. This trace formula is an expression for a Lefschetz distribution in terms of infinitesimal data of the flow at the fixed points and closed orbits. This Lefschetz distribution should be an analogue of the Lefschetz number for the action induced by the flow on some leafwise cohomology, whose value is a distribution on $\R$---the precise definition of these notions is part of the problem. In \cite{AlvKordy2002,AlvKordy2008a}, the first two authors proved such a trace formula when the flow has no preserved leaves; see also the contributions \cite{Leichtnam2008,Leichtnam2014} by the third author. The general case is considerably more involved. In \cite{AlvKordyLeichtnam-atffff}, we propose a solution to this problem using a few additional ingredients. One of them is the b-trace introduced by Melrose \cite{Melrose1993}. Since the b-trace is not really a trace, it produces an extra term, denoted by $Z$, in the same way as the eta invariant shows up in Index Theory on manifolds with boundary. In our trace formula, the term $Z$ is a contribution from the compact leaves preserved by the flow, which depends on the choice of a form defining the foliation and a metric on the ambient manifold. But $Z$ may not be well defined in general; it will be proved that appropriate choices of the form and the metric guarantee its existence.

Precisely, we would like to define
\begin{equation}\label{Z}
Z=Z(M,g,\eta)=\lim_{\mu\to+\infty}Z_\mu\;,
\end{equation}
in the space of tempered distributions on $\R$, where $Z_\mu=Z_\mu(M,g,\eta)$ ($\mu\gg0$) should be a tempered distribution defined by
\begin{equation}\label{Z_mu}
\langle Z_\mu,f\rangle=-\frac{1}{2\pi}\int_0^\infty\int_{-\infty}^\infty\Str\left({\eta\wedge}\,\delta_ze^{-u\Delta_z}\right)
\,\hat f(\nu)\,d\nu\,du\;,
\end{equation}
for any Schwartz function $f$, where $\hat f$ stands for the Fourier transform of $f$. 

Let $\delta_0$ denote the Dirac distribution at $0$ on $\R$. The problem about the definition of $Z$ is solved in our fourth main theorem for the same class of Morse forms as before.

\begin{thm}\label{t: Z_pm}
Let $M\equiv(M,g)$ be a closed Riemannian $n$-manifold. Let $\eta$ be a closed $1$-form on $M$ satisfying~\ref{i-a:  eta Morse},~\ref{i-a:  eta Lyapunov} and~\ref{i-a:  tight} with some vector field satisfying~\ref{i-a:  X ...}. Then~\eqref{Z} and~\eqref{Z_mu} define the tempered distribution $Z=\bfz\delta_0$.
\end{thm}

According to \Cref{t: prescription of lim zeta(1 z),t: Z_pm}, we can choose $\eta$ and $g$ in the trace formula for foliated flows so that $Z(M,g,\pm\eta)=0$ if $n$ is even, achieving the original expression of Deninger's conjecture. 

It looks clear that extensions of \Cref{t: zeta(1 z),t: zeta_sm/la(1 z),t: prescription of lim zeta(1 z),t: Z_pm} with coefficients in flat vector bundles could be similarly proved. We only consider complex coefficients for the sake of simplicity since this is enough for our application.

\subsection{Some ideas of the proofs of \Cref{t: zeta(1 z),t: zeta_sm/la(1 z),t: prescription of lim zeta(1 z),t: Z_pm}}\label{ss: intro-proofs}

As mentioned before, the inequality~\eqref{| |_L^infty le C_m | |_m i nu} is essential to obtain the uniformity on $\nu$ of our estimates. To prove it, we can take $\nu=1$ by considering an arbitrary closed real 1-form $\eta$ (\Cref{p: perturbed Sobolev}). Let $\|\ \|_{m,i\eta}$ be the $m$th Sobolev norm defined with the perturbed Laplacian $\Delta_{i\eta}$ induced by $i\eta$ as above. By ellipticity, $\|\ \|_{L^\infty}\le C_{m,i\eta}\|\ \|_{m,i\eta}$ for some $C_{m,i\eta}>0$ depending on $\eta$, which can be chosen to be optimal. For two such forms, $\eta$ and $\eta'$, the cohomology class $[\eta-\eta']$ is in the lattice $2\pi H^1(M,\Z)$ of $H^1(M,\R)$ just when $\eta-\eta'=h^*d\theta$ for some smooth map $h:M\to\S^1$, where $\theta$ is the multivalued angle function on the circle $\S^1$. This gives the gauge equivalence $\Delta_{i\eta'}=e^{-ih^*\theta}\,\Delta_{i\eta}\,e^{ih^*\theta}$, where $e^{\pm ih^*\theta}$ is well defined on $M$. It follows that $\eta\mapsto C_{m,i\eta}$ induces a function on the torus $H^1(M,\R)/2\pi H^1(M,\Z)$. On the other hand, every $C_{m,i\eta}$ can be estimated in terms of the $C^m$ norm of $\eta$ (\Cref{p: | alpha |_m i eta}). Hence, by compactness of $H^1(M,\R)/2\pi H^1(M,\Z)$, the values $C_{m,i\eta}$ have an upper bound $C_m$, which satisfies the desired inequality $\|\ \|_{L^\infty}\le C_m\|\ \|_{m,i\eta}$.

For an arbitrary closed real 1-form $\eta$, and for all $t>0$ and $z\in\C$, a supersymmetric argument shows that (\Cref{p: partial_z Str(sN e^-t Delta_z) = -t Str(eta wedge D_z e^-t Delta_z)})
\begin{equation}\label{partial_z Str(fN e^-t Delta_z) = -t Str(eta wedge D_z e^-t Delta_z)}
\partial_z\Str\big(\sN e^{-t\Delta_z}\big)=-t\Str\big({\eta\wedge}\,D_ze^{-t\Delta_z}\big)\;,
\end{equation}
where $\sN$ is the number operator on $\Omega(M)$ (\Cref{sss: basic notation}). Then we apply that the coefficients of the asymptotic expansion of $\Str(\sN e^{-t\Delta_z})$ as $t\downarrow0$ (the derived heat trace invariants) are independent of $z$ up to order $n$ \cite[Theorem~7.10]{BismutZhang1992} (see also \cite{AlvGilkey2023}). Thus, by~\eqref{partial_z Str(fN e^-t Delta_z) = -t Str(eta wedge D_z e^-t Delta_z)}, the coefficients of the asymptotic expansion of $\Str({\eta\wedge}\,D_ze^{-t\Delta_z})$ as $t\downarrow0$ vanish up to order $n$. Now \Cref{t: zeta(1 z)} follows by the general theory of zeta functions of operators (\Cref{ss: regularity}).

The theta function $\theta(s,z)$ is defined like $\zeta(s,z)$ by using $-\Str(\sN\Delta_z^{-s}\Pi_z^\perp)$ instead of $\Str({\eta\wedge}\,\delta_z\Delta_z^{-s}\Pi_z^\perp)$. Assuming the hypotheses of \Cref{t: zeta_sm/la(1 z)}, write $\theta(s,z)$ as the sum of contributions from the small/large spectrum, $\theta_{\text{\rm sm/la}}(s,z)$, as before. Thus $e^{\theta'(0,z)/2}$ is the factor used to define the Ray-Singer metric on $\det H_z^\bullet(M)$ \cite{BismutZhang1992}, where the prime denotes $\partial_s$. We obtain (\Cref{c: zetala(1 z) = partial_z thetala'(0 z)})
\begin{equation}\label{zetala(1 z) = partial_z thetala'(0 z)}
\zetala(1,z)=\partial_z\thetala'(0,z)\;.
\end{equation}
This equality allows us to use the deep relation between the Ray-Singer metric and the Milnor metric on $\det H_z^\bullet(M)$, proved by Bismut and Zhang \cite{BismutZhang1992,BismutZhang1994}. To apply this result, we have to make involved computations concerning derivatives with respect to $z$ of the orthogonal projection to $E_{z,\text{\rm sm}}$ and of other operators related with the isomorphism $\Phi_z:E_{z,\text{\rm sm}}\to\bfC^\bullet$, as well as estimates of the asymptotic behavior as $\mu\to+\infty$ of these operators and their derivatives (\Cref{ss: asymptotic properties of the sm proj,ss: derivatives of the sm proj,ss: small complex vs Morse complex,ss: derivatives,ss: large zeta invariant}). In this way, we obtain that $\zetala(1,z)$ is asymptotic to $\bfz_{\text{\rm la}}$ as $\mu\to+\infty$ (\Cref{ss: large zeta invariant}). This proves \Cref{t: zeta_sm/la(1 z)}~\ref{i: zetala(1 z)}.

When $\eta$ is exact, we show this asymptotic expression of $\zeta_{\text{\rm la}}(1,z)$ assuming only~\ref{i-a:  eta Morse} (\Cref{ss: differential of a Morse function}), without using~\eqref{zetala(1 z) = partial_z thetala'(0 z)} and the indicated strong result of Bismut and Zhang. Instead, we apply that the index density of $\Delta_z$ is independent of $z$, also proved by Bismut and Zhang \cite[Theorem 13.4]{BismutZhang1992}; see also \cite[Theorem~1.5]{AlvGilkey2021a} and \cite{AlvKordyLeichtnam-atffff}.

On the other hand, given any $\xi\in H^1(M,\R)$ and a vector field $X$ satisfying~\ref{i-a:  X ...}, we prove that there is some $\eta\in\xi$ and a metric $g$ satisfying~\ref{i-a:  eta Morse},~\ref{i-a:  eta Lyapunov} and~\ref{i-a:  tight} (\Cref{t: extension of Smale}). This can be considered as an extension of a theorem of Smale stating the existence of nice Morse functions \cite[Theorem~B]{Smale1961} (the case where $\xi=0$). Its proof is relegated to \Cref{s: integrals along instantons} because of its different nature. 

The properties~\ref{i-a:  eta Morse}--\ref{i-a:  tight} are used to give an asymptotic description of $\bfd_z$ as $\mu\to+\infty$ (\Cref{ss: bfd'}). From this asymptotic description and using that $\Phi_z:E_{z,\text{\rm sm}}\to\bfC^\bullet$ is an isomorphism for $\mu\gg0$, we get upper and lower bounds of the nonzero small spectrum of $\Delta_z$ (\Cref{t: estimates of the small eigenvalues}), which are independent of $\nu$. This is a partial extension of accurate descriptions of the nonzero small eigenvalues achieved in the case where $\eta$ is exact and the parameter is real \cite{LePeutrecNierViterbo2013,Michel2019}. With the same procedure and using the bounds of the nonzero small spectrum, it also follows that $\zetasm(1,z)=\bfz_{\text{\rm sm}}+O(\mu^{-1})$ as $\mu\to+\infty$ (\Cref{ss: sm zeta inv}), showing \Cref{t: zeta_sm/la(1 z)}~\ref{i: zetasm(1 z)}.

Next, by modifying $\eta$ and $X$ around its zero points of index $0$ and $n$, without changing the cohomology class of $\eta$, we can achieve any real number as $\pm\bfz(\pm\eta)$ if $n$ is even, or any large enough real number as $\bfz(\eta)$ if $n$ is odd (\Cref{s: prescription}). This shows~\Cref{t: prescription of lim zeta(1 z)}.

If it is possible to switch the order of integration in~\eqref{Z_mu},
\begin{multline}\label{Z_mu switching integrals}
\langle Z_\mu,f\rangle
=-\frac{1}{2\pi}\int_{-\infty}^\infty\int_0^\infty\Str\left({\eta\wedge}\,\delta_ze^{-u\Delta_z}\right)\,\hat f(\nu)\,du\,d\nu\\
=\frac{1}{2\pi}\int_{-\infty}^\infty\lim_{t\downarrow0}
\Str\left({\eta\wedge}\,d_z^{-1}e^{-t\Delta_z}\Pi_z^1\right)\,\hat f(\nu)\,d\nu\;,
\end{multline}
then \Cref{t: Z_pm} is an easy consequence of \Cref{t: zeta(1 z)}. Thus it only remains to prove that both~\eqref{Z_mu} and~\eqref{Z_mu switching integrals} define the same tempered distribution $Z_\mu$. This follows from the Lebesgue's dominated convergence theorem and Fubini's theorem (\Cref{s: switch}). The verification of the hypothesis of the Fubini's theorem requires the above lower estimate of the nonzero spectrum.

For the readers convenience, we recall the needed preliminaries about the many topics involved: Witten's perturbations, Morse forms, asymptotic expansions of heat kernels, zeta functions of operators, Morse and Smale vector fields, the Morse complex and Quillen metrics (Reidemeister, Milnor and Ray-Singer metrics).



\section{Witten's perturbations}\label{s: Witten}

\subsection{Preliminaries on the Witten's perturbations}\label{ss: prelim Witten}

\subsubsection{Basic notation}\label{sss: basic notation}

Let $M\equiv(M,g)$ be a closed Riemannian $n$-manifold. For any smooth Euclidean/Hermitean vector bundle $E$ over $M$, let $C^m(M;E)$, $C^\infty(M;E)$, $L^2(M;E)$, $L^\infty(M;E)$ and $H^m(M;E)$ denote the spaces of distributional sections that are $C^m$, $C^\infty$, $L^2$, $L^\infty$ and of Sobolev order $m$, respectively; as usual, $E$ is removed from this notation if it is the trivial line bundle. Consider the induced scalar product $\langle\ ,\ \rangle$ and norm $\|\ \|$ on $L^2(M;E)$, and the induced norm $\|\ \|_{L^\infty}$ on $L^\infty(M;E)$. Fix also norms, $\|\ \|_m$ on every $H^m(M;E)$ and $\|\ \|_{C^m}$ on $C^m(M;E)$. If $P$ is the orthogonal projection of $L^2(M;E)$ to some closed subspace $V$, then $P^\perp$ denotes the orthogonal projection to $V^\perp$. Let $o(E)$ denote the flat real orientation line bundle of $E$. It is said that $E$ is orientable when $o(E)$ is trivial. In this case, an orientation of $E$ is described by a (necessarily smooth) non-vanishing flat section $\OO_E$ of $o(E)$; for simplicity, it will be said that $\OO_E$ itself is an orientation. In particular, an orientation of $M$ is described using $o(M):=o(TM)$. The flat line bundle $o(E)\otimes o(E)$ is always trivial.

Let $T_\C M=TM\otimes\C$ and $T_\C^*M=T^*M\otimes\C$. The exterior bundle with coefficients in $\K=\R,\C$ is denoted by $\Lambda_\K=\Lambda_\K M$, and let $\Omega(M,\K)=C^\infty(M;\Lambda_\K)$; in particular, $C^\infty(M,\K)=\Omega^0(M,\K)$. The Levi-Civita connection is denoted by $\nabla=\nabla^M$. As usual, $d$ and $\delta$ denote the de~Rham derivative and coderivative, and let $D=d+\delta$ and $\Delta=D^2=d\delta+\delta d$ (the Laplacian). Let $Z(M,\K)$ and $B(M,\K)$ denote the kernel and image of $d$ in $\Omega(M,\K)$. Thus $H^\bullet(M,\K)=Z(M,\K)/B(M,\K)$ is the de~Rham cohomology with coefficients in $\K$. We typically consider complex coefficients, so we will omit $\K$ from all of the above notation just when $\K=\C$. Take $\|\ \|_m$ and $\|\ \|_{C^m}$ given on $\Omega(M)$ by
\[
\|\alpha\|_m=\sum_{k=0}^m\|D^k\alpha\|\;,\quad\|\alpha\|_{C^m}=\sum_{k=0}^m\|\nabla^k\alpha\|_{L^\infty}\;.
\]
In particular, we take $\|\ \|=\|\ \|_0$ and $\|\ \|_{C^0}=\|\ \|_{L^\infty}|_{C^0(M;E)}$.

On any graded vector space $V^\bullet$, let $\sw$ and $\sN$ be the degree involution and number operator; i.e., $\sw=(-1)^k$ and $\sN=k$ on $V^k$. For any homogeneous linear operator between graded vector spaces, $T:V^\bullet\to W^\bullet$, the notation $T_k$ means its precomposition with the canonical projection of $V^\bullet$ to $V^k$. 
 If $T$ is of degree $l$ ($T(V^k)\subset W^{k+l}$ for all $k$), then
\begin{equation}\label{sN T = T (sN + l)}
\sw T=(-1)^lT\sw\;,\quad \sN T=T(\sN+l)\;.
\end{equation}

For any $\eta\in\Omega^1(M,\R)$ with $\eta^\sharp=X\in\fX(M):=C^\infty(M;TM)$ ($\eta=g(X,{\cdot})$), let $\LL_X$ and $\iota_X$ denote the Lie derivative and interior product with respect to $X$, and let ${\eta\lrcorner}=-(\eta\wedge)^*=-\iota_X$. Using the identity $\Cl(T^*M)\equiv\Lambda_\R M$ defined by the symbol of filtered algebras, the left Clifford multiplication by $\eta$ is $c(\eta)={\eta\wedge}+{\eta\lrcorner}$, and the composition of $\sw$ with the right Clifford multiplication by $\eta$ is $\hat c(\eta)={\eta\wedge}-{\eta\lrcorner}$; in particular, $c(\eta)^*=-c(\eta)$ and $\hat c(\eta)^*=\hat c(\eta)$. Recall that, for any $h\in C^\infty(M,\R)$,
\begin{equation}\label{[D h] = hat c(dh)}
[D,h]=\hat c(dh)\;.
\end{equation}

In the whole paper, unless otherwise indicated, we will use the following notation without further comment. We use constants $C,c>0$ without even mentioning their existence, and their precise values may change from line to line. We may add subindices or primes to these constants if needed. We also use a complex parameter $z=\mu+i\nu\in\C$ ($\mu,\nu\in\R$ and $i=\sqrt{-1}$). Recall that $\partial_z=(\partial_\mu-i\partial_\nu)/2$ and $\partial_{\bar z}=(\partial_\mu+i\partial_\nu)/2$.

\subsubsection{Perturbations defined by a closed real 1-form}\label{sss: perturbations}

For any $\omega\in Z^1(M)$, we have the Witten's type perturbations $d_\omega$, $\delta_\omega$, $D_\omega$ and $\Delta_\omega$ of $d$, $\delta$, $D$ and $\Delta$. Given $\eta\in Z^1(M,\R)$ and $z\in\C$, we write $d_z=d_{z\eta}$, $\delta_z=\delta_{z\eta}$, $D_z=D_{z\eta}$ and $\Delta_z=\Delta_{z\eta}$. These operators have the following expressions:
\begin{equation}\label{perturbed opers}
\left\{
\begin{aligned}
d_z&=d+z\,\eta\wedge\;,\quad
\delta_z=d_z^*=\delta-\bar z\,{\eta\lrcorner}\;,\\
D_z&=d_z+\delta_z=D+\mu\hat c(\eta)+i\nu c(\eta)
=D_{i\nu}+\mu\hat c(\eta)\;,\\
\Delta_z&=D_z^2=d_z\delta_z+\delta_zd_z
=\Delta+\mu\sH_\eta+i\nu\sJ_\eta+|z|^2|\eta|^2\\
&=\Delta_{i\nu}+\mu\sH_\eta+\mu^2|\eta|^2\;,
\end{aligned}
\right.
\end{equation}
where, for $X=\eta^\sharp$,
\[
\sH_\eta=D\hat c(\eta)+\hat c(\eta)D=\LL_X^*+\LL_X\;,\quad
\sJ_\eta=Dc(\eta)+c(\eta)D=\LL_X^*-\LL_X\;.
\]
Note that $\sH_\eta$ is of order zero and $\sJ_\eta$ of order one. 

As families of operators, $d_z$ and $\delta_z$ are holomorphic and anti-holomorphic functions of $z$, respectively. More precisely, 
it follows from~\eqref{perturbed opers} that
\begin{equation}\label{partial_z delta_z}
\left\{
\begin{alignedat}{3}
\partial_zd_z&={\eta\wedge}\;,&\quad\partial_z\delta_z&=0\;,&\quad\partial_z\Delta_z&={\eta\wedge}\,\delta_z+\delta_z\,{\eta\wedge}\;,\\
\partial_{\bar z}d_z&=0\;,&\quad\partial_{\bar z}\delta_z&=-{\eta\lrcorner}\;,&\quad\partial_{\bar z}\Delta_z&=-{\eta\lrcorner}\,d_z-d_z\,{\eta\lrcorner}\;.
\end{alignedat}
\right.
\end{equation}

The operator $d_z$ defines an elliptic complex on $\Omega(M)$, whose cohomology is denoted by $H_z^\bullet(M)$. Since $d_z$ has the same principal symbol as $d$, it is a generalized Dirac complex and $\Delta_z$ a self-adjoint generalized Laplacian \cite[Definition~2.2]{BerlineGetzlerVergne2004}. If $\theta=\eta+dh$ for some $h\in C^\infty(M,\R)$, then the multiplication operator
\begin{equation}\label{e^zh: (Omega(M) d_z theta) to (Omega(M) d_z eta)}
e^{zh}:(\Omega(M),d_{z\theta})\to(\Omega(M),d_{z\eta})
\end{equation}
is an isomorphism of differential complexes, and therefore it induces an isomorphism $H_{z\theta}^\bullet(M)\cong H_{z\eta}^\bullet(M)$. Thus the isomorphism class of $H_z^\bullet(M)$ only depends on $\xi:=[\eta]\in H^1(M,\R)$ and $z\in\C$. By ellipticity, $D_z$ and $\Delta_z$ have a discrete spectrum, and there is a decomposition, equalities and isomorphism of Hodge type,
\begin{equation}\label{Hodge dec Novikov}
\left\{
\begin{gathered}
\Omega(M)=\ker\Delta_z\oplus\im d_z\oplus\im\delta_z\;,\\
\ker\Delta_z=\ker D_z=\ker d_z\cap\ker\delta_z\;,\quad
\im\Delta_z=\im D_z=\im d_z\oplus\im\delta_z\;,\\
H_z^\bullet(M)\cong\ker\Delta_z\;,
\end{gathered}
\right.
\end{equation}
as topological vector spaces. The orthogonal projections of $\Omega(M)$ to $\ker\Delta_z$, $\im d_z$ and $\im\delta_z$ are denoted by $\Pi_z=\Pi^0_z$, $\Pi^1_z$ and $\Pi^2_z$, respectively; thus $\Pi_z^\perp=\Pi^1_z+\Pi^2_z$. The restrictions $d_z:\im\delta_z\to\im d_z$, $\delta_z:\im d_z\to\im\delta_z$ and $D_z:\im D_z\to\im D_z$ are topological isomorphisms, and therefore the compositions $d_z^{-1}\Pi_z^1$, $\delta_z^{-1}\Pi_z^2$ and $D_z^{-1}\Pi_z^\perp$ are defined and continuous on $\Omega(M)$. For every degree $k$, the diagram
\begin{equation}\label{CD im delta_z k+1 ...}
\begin{CD}
\im\delta_{z,k+1} @>{d_{z,k}}>> \im d_{z,k} \\
@V{\Delta_{z,k}}VV @VV{\Delta_{z,k+1}}V \\
\im\delta_{z,k+1} @>{d_{z,k}}>> \im d_{z,k}
\end{CD}
\end{equation}
is commutative. The twisted Betti numbers $\beta_z^k=\beta_z^k(M,\xi)=\dim H_z^k(M)$ give rise to the usual Euler characteristic \cite[Proposition~1.40]{Farber2004},
\begin{equation}\label{sum_k (-1)^k beta_z^k = chi(M)}
\sum_k(-1)^k\beta_z^k=\chi(M)\;.
\end{equation}
(This is also a consequence of the index theorem.) For every degree $k$, $\beta_z^k$ is independent of $z$ outside a discrete subset of $\C$, where $\beta_z^k$ jumps (Mityagin and Novikov \cite[Theorem~1]{Novikov2002}). This ground value of $\beta_z^k$ is called the $k$-th \emph{Novikov Betti number}, denoted by $\betaNo^k=\betaNo^k(M,\xi)$. It will be shown in \Cref{sss: perturbed Morse complex} that
\begin{equation}\label{beta_z^k = betaNo^k}
\beta_z^k=\betaNo^k\quad\text{for}\quad|\mu|\gg0\;.
\end{equation} 
(When $z$ is real, this is proved in \cite[Theorem~2.8]{Farber1995}, \cite[Lemma~1.3]{BravermanFarber1997}, \cite[Proposition~4]{BurgheleaHaller2004}.) Thus the discrete set of parameters $z\in\C$ with $\beta_z^k(M,\xi)>\betaNo^k(M,\xi)$ for some degree $k$ is contained in a strip $|\mu|\le C$.

By~\eqref{perturbed opers} and since $\eta$ is real, for all $\alpha\in \Omega(M)$,
\begin{equation}\label{overline d_z alpha}
\overline{d_z\alpha}=d_{\bar z}\bar\alpha\;,\quad\overline{\delta_z\alpha}=\delta_{\bar z}\bar\alpha\;,\quad
\overline{D_z\alpha}=D_{\bar z}\bar\alpha\;,\quad\overline{\Delta_z\alpha}=\Delta_{\bar z}\bar\alpha\;. 
\end{equation}
So conjugation induces $\C$-antilinear isomorphisms
\[
H_z^k(M)\cong H_{\bar z}^k(M)\;,\quad\ker\Delta_{z,k}\cong\ker\Delta_{\bar z,k}\;,
\]
yielding $\beta_z^k=\beta_{\bar z}^k$.

\subsubsection{Case of an exact form}\label{sss: case of an exact form}

When $\eta=dh$ for some $h\in C^\infty(M,\R)$, we have the original Witten's perturbations, which satisfy
\begin{equation}\label{Witten's opers}
\left\{
\begin{gathered}
d_z=e^{-zh}\,d\,e^{zh}=e^{-i\nu h}\,d_\mu\,e^{i\nu h}\;,\quad
\delta_z=e^{\bar zh}\,\delta\,e^{-\bar zh}=e^{-i\nu h}\,\delta_\mu\,e^{i\nu h}\;,\\
D_z=e^{-i\nu h}\,D_\mu\,e^{i\nu h}\;,\quad
\Delta_z=e^{-i\nu h}\,\Delta_\mu\,e^{i\nu h}\;. 
\end{gathered}
\right.
\end{equation}
Thus the multiplication operator
\begin{equation}\label{e^zh: (Omega(M) d_z) to (Omega(M) d)}
e^{zh}:(\Omega(M),d_z)\to(\Omega(M),d)
\end{equation} 
is an isomorphism of differential complexes. Therefore $H_z^\bullet(M)\cong H^\bullet(M)$, yielding $\beta_z^k=\beta^k=\beta^k(M)$ (the $k$th Betti number) in this case. Moreover multiplication by $e^{i\nu h}$ defines a unitary isomorphism $\ker\Delta_z\cong\ker\Delta_\mu$.

\subsubsection{Interpretation of the closed form as a flat connection}\label{sss: LL^z}

There is a unique flat connection $\nabla^{M\times\C}$ on the trivial complex line bundle $M\times\C$ so that $\nabla^{M\times\C}1=\eta$. The corresponding flat complex line bundle is denoted by $\LL=\LL_\eta$. Note that $\LL_{z\eta}=\LL^z$. Let $(\Omega(M,\LL^z)=(\Omega(M),d^{\LL^z})$ be the de Rham complex with coefficients in $\LL^z$. It is well-known that $d_z=d^{\LL^z}$ on $\Omega(M)=\Omega(M,\LL^z)$, and therefore $H^\bullet(M,\LL^z)=H_z^\bullet(M)$. Since every $\LL^z$ is canonically trivial as a line bundle, it has a canonical Hermitian structure $g^{\LL^z}$. An easy local computation shows that (see the example given in \cite[pp.~11--12]{BismutZhang1992})
\begin{equation}\label{nabla^LL^z g^LL^z = -2 mu eta otimes g^LL^z}
\nabla^{\LL^z}g^{\LL^z}=-2\mu\eta\otimes g^{\LL^z}\;.
\end{equation}



\subsubsection{Perturbed operators on oriented manifolds}\label{sss: perturbed ops on oriented mfds}

The mappings $(\alpha,\beta)\mapsto\alpha\wedge\beta$ and $(\alpha,\beta)\mapsto\alpha\wedge\bar\beta$ induce respective bilinear and sesquilinear maps,
\[
H_z^k(M)\times H_{-z}^l(M) \to H^{k+l}(M)\;,\quad H_z^k(M)\times H_{-\bar z}^l(M) \to H^{k+l}(M)\;,
\]
as follows from the interpretation of $d_z$ given in \Cref{sss: LL^z}, or by a direct check.

Now assume $M$ is oriented. Then the above maps and integration on $M$ define respective nondegenerate bilinear and sesquilinear pairings
\[
H_z^k(M)\times H_{-z}^{n-k}(M)\to\C\;,\quad H_z^k(M)\times H_{-\bar z}^{n-k}(M)\to\C\;.
\] Thus
\begin{equation}\label{beta_z^k = beta_-z^n-k}
\beta_z^k=\beta_{-z}^{n-k}=\beta_{-\bar z}^{n-k}=\beta_{\bar z}^k\;.
\end{equation}

Let $\star$ and $\bar\star$ denote the $\C$-linear and $\C$-antilinear extensions to $\Lambda M$ of the Hodge operator $\star$ on $\Lambda_\R M$, respectively. These operators are determined by the conditions
\[
\alpha\wedge\overline{\star\beta}=g(\alpha,\beta)\,\dvol=\alpha\wedge\bar\star\beta
\]
for $\alpha,\beta\in \Omega(M)$, where $\dvol=\star1$ is the volume form. The following equalities on $\Omega^k(M)$ follow from~\eqref{perturbed opers} and the usual equalities relating $\star$, $d$, $\delta$, ${\eta\wedge}$ and ${\eta\lrcorner}$ (see e.g.\ \cite[Chapters~1 and~3]{Roe1998}, \cite[Section~1.5.2]{Gilkey1995}, \cite[Section~3.6]{BerlineGetzlerVergne2004}):
\begin{equation}\label{star z}
\left\{
\begin{alignedat}{3}
d_z\,\star&=(-1)^k\star\,\delta_{-\bar z}\;,&\quad\delta_z\,\star&=(-1)^{k+1}\star\,d_{-\bar z}\;,&\quad
\Delta_z\,\star&=\star\,\Delta_{-\bar z}\;,\\
d_z\,\bar\star&=(-1)^k\;\bar\star\;\delta_{-z}\;,&\quad\delta_z\;\bar\star&=(-1)^{k+1}\,\bar\star\;d_{-z}\;,&\quad
\Delta_z\;\bar\star&=\bar\star\;\Delta_{-z}\;.
\end{alignedat}
\right.
\end{equation}
Then we get a linear isomorphism $\star:\ker\Delta_z\to\ker\Delta_{-\bar z}$ and an antilinear isomorphism $\bar\star:\ker\Delta_z\to\ker\Delta_{-z}$, inducing a linear isomorphism $H_z^k(M)\cong H_{-\bar z}^{n-k}(M)$ and an antilinear isomorphism $H_z^k(M)\cong H_{-z}^{n-k}(M)$ by~\eqref{Hodge dec Novikov}.

\subsection{Perturbation of the Sobolev norms}\label{ss: Sobolev norms}

For $m\in\N_0$ and $\omega\in Z^1(M)$, define the norm $\|\ \|_{m,\omega}$ on $H^m(M;\Lambda)$ by
\[
\|\alpha\|_{m,\omega}=\sum_{k=0}^m\big\|D_\omega^k\alpha\big\|\;.
\]

\begin{prop}\label{p: | alpha |_m i eta}
For all $\omega\in Z^1(M)$ and $\alpha\in H^m(M;\Lambda)$,
\begin{gather*}
\|\alpha\|_{m,\omega}\le C_m\sum_{k=0}^m\|\omega\|_{C^k}^{m-k}\|\alpha\|_k\;,\quad
\|\alpha\|_m\le C_m\sum_{k=0}^m\|\omega\|_{C^k}^{m-k}\|\alpha\|_{k,\omega}\;.
\end{gather*}
\end{prop}

\begin{proof}
We proceed by induction on $m$. We have  $\|\ \|_{0,\omega}=\|\ \|$. Now take $m>0$ and assume these inequalities hold for $m-1$. For $\eta\in Z^1(M,\R)$ and $\alpha\in\Omega(M)$, we have
\begin{equation}\label{| hat c(eta) alpha |_m}
\|\hat c(\eta)\alpha\|_m,\|c(\eta)\alpha\|_m\le C'_m\|\eta\|_{C^m}\|\alpha\|_m\;.
\end{equation}
Applying these inequalities to the real and imaginary parts of $\omega$, and using the induction hypothesis and~\eqref{perturbed opers}, we get
\begin{align*}
\|\alpha\|_{m,\omega}&=\|\alpha\|+\|D_{\omega}\alpha\|_{m-1,\omega}
\le\|\alpha\|+C_{m-1}\sum_{k=0}^{m-1}\|\omega\|_{C^k}^{m-1-k}\|D_{\omega}\alpha\|_k\\
&\le\|\alpha\|+C_{m-1}\sum_{k=0}^{m-1}\|\omega\|_{C^k}^{m-1-k}\big(\|D\alpha\|_k+C'_k\|\omega\|_{C^k}\|\alpha\|_k\big)\\
&\le\|\alpha\|+C_{m-1}\sum_{k=0}^{m-1}\|\omega\|_{C^k}^{m-1-k}\big(\|\alpha\|_{k+1}+C'_k\|\omega\|_{C^k}\|\alpha\|_k\big)\\
&\le C_m\sum_{l=0}^m\|\omega\|_{C^l}^{m-l}\|\alpha\|_l\;,
\end{align*}
\begin{align*}
\|\alpha\|_m&=\|\alpha\|+\|D\alpha\|_{m-1}
\le\|\alpha\|+\|D_{\omega}\alpha\|_{m-1}+C'_{m-1}\|\omega\|_{C^{m-1}}\|\alpha\|_{m-1}\\
&\le\|\alpha\|+C_{m-1}\sum_{k=0}^{m-1}\big(\|\omega\|_{C^k}^{m-1-k}\|D_{\omega}\alpha\|_{k,\omega}
+C'_{m-1}\|\omega\|_{C^k}^{m-k}\|\alpha\|_{k,\omega}\big)\\
&\le\|\alpha\|+C_{m-1}\sum_{k=0}^{m-1}\big(\|\omega\|_{C^k}^{m-1-k}\|\alpha\|_{k+1,\omega}
+C'_{m-1}\|\omega\|_{C^k}^{m-k}\|\alpha\|_{k,\omega}\big)\\
&\le C_m\sum_{l=0}^m\|\omega\|_{C^l}^{m-l}\|\alpha\|_{l,\omega}\;.\;\qedhere
\end{align*}
\end{proof}

Let $Z(M,\Z)\subset Z(M,\R)$ denote the graded additive subgroup of forms that represent cohomology classes in the image of the canonical homomorphism $H^\bullet(M,\Z)\to H^\bullet(M,\R)$. Recall that we can consider $H^1(M,\Z)$ as a lattice in $H^1(M,\R)$ by the universal coefficient theorem for cohomology. Let $\theta$ be the multivalued angle function on $\S^1$. Then $d\theta$ is the angular form on $\S^1$ with $\int_{\S^1}d\theta=2\pi$. For $\eta\in Z^1(M,\R)$, we have $\eta\in2\pi Z^1(M,\Z)$ if and only if there is some smooth map $h:M\to\S^1$ such that $\eta=h^*d\theta$ (see e.g.\ \cite[Lemma~2.1]{Farber2004}).

In \Cref{p: | alpha |_m i eta}, the dependence of the constants on $\omega$ cannot be avoided. For instance, for $M=\S^1$ with the standard metric $g=(d\theta)^2$, we have $\|1\|_m=\sqrt{2\pi}$, whereas $\|1\|_{m,i\eta}=\sqrt{2\pi}\sum_{k=0}^m|\nu|^k$ for $\eta=\nu\,d\theta$ ($\nu\in\R$). However, the following version of a Sobolev inequality for $\|\ \|_{m,i\eta}$ involves a constant independent of $\eta$.

\begin{prop}\label{p: perturbed Sobolev}
If $m>n/2$, for all $\eta\in Z^1(M,\R)$ and $\alpha\in H^m(M;\Lambda)$,
\[
\|\alpha\|_{L^\infty}\le C_m\|\alpha\|_{m,i\eta}\;.
\]
\end{prop}

\begin{proof}
By the Sobolev embedding theorem, we have
\[
C_{m,i\eta}:=\sup_{0\ne\alpha\in\Omega(M)}\frac{\|\alpha\|_{L^\infty}}{\|\alpha\|_{m,i\eta}}>0\;.
\]

Take any $\eta\in Z^1(M,\R)$ and $\omega\in 2\pi Z^1(M,\Z)$, and let $\eta'=\eta+\omega$. Then $\omega=h^*d\theta$ for some smooth function $h:M\to\S^1$. Since the difference between the multiple values of $\theta$ at every point of $\S^1$ are in $2\pi\Z$, the functions $e^{\pm ih^*\theta}$ are well defined and smooth on $M$. Moreover, applying~\eqref{Witten's opers} locally, we get $D_{i\eta'}=e^{-ih^*\theta}\,D_{i\eta}\,e^{ih^*\theta}$. So, for $0\ne\alpha\in\Omega(M)$,
\begin{align*}
\|\alpha\|_{L^\infty}&=\|e^{i\,h^*\theta}\alpha\|_{L^\infty}\le C_{m,i\eta}\|e^{ih^*\theta}\alpha\|_{m,i\eta}\\
&=C_{m,i\eta}\sum_{k=0}^m\|D_{i\eta}^k\,e^{ih^*\theta}\alpha\|
=C_{m,i\eta}\sum_{k=0}^m\|e^{-ih^*\theta}\,D_{i\eta}^k\,e^{ih^*\theta}\alpha\|\\
&=C_{m,i\eta}\sum_{k=0}^m\|D_{i\eta'}^k\alpha\|
=C_{m,i\eta}\|\alpha\|_{m,i\eta'}\;.
\end{align*}
This shows that
\begin{equation}\label{C_m i eta = C_m i eta'}
\eta-\eta'\in2\pi Z^1(M,\Z)\Rightarrow C_{m,i\eta}=C_{m,i\eta'}\;.
\end{equation}

Since $2\pi H^1(M,\Z)$ is a lattice in $H^1(M,\R)$, there is a compact subset $K\subset H^1(M,\R)$ such that
\begin{equation}\label{K + 2 pi H^1(M Z) = H^1(M R)}
K+2\pi H^1(M,\Z)=H^1(M,\R)\;.
\end{equation}
Take a linear subspace $V\subset Z^1(M,\R)$ such that the canonical projection $V\to H^1(M,\R)$ is an isomorphism, and let $L\subset V$ be the compact subset that corresponds to $K$. By~\eqref{K + 2 pi H^1(M Z) = H^1(M R)},
\begin{equation}\label{L + 2 pi Z^1(M Z) = Z^1(M R)}
L+2\pi Z^1(M,\Z)=Z^1(M,\R)\;.
\end{equation}
Moreover $L$ is bounded with respect to $\|\ \|_{C^m}$. Therefore, by \Cref{p: | alpha |_m i eta}, for all $\eta\in L$ and $\alpha\in\Omega(M)$,
\[
\|\alpha\|_{L^\infty}\le C_{m,0}\|\alpha\|_m\le C_m\|\alpha\|_{m,i\eta}\;,
\]
yielding
\begin{equation}\label{C_m i eta le C_m}
\sup_{\eta\in L}C_{m,i\eta}\le C_m\;.
\end{equation}
The result follows from~\eqref{C_m i eta = C_m i eta'},~\eqref{L + 2 pi Z^1(M Z) = Z^1(M R)} and~\eqref{C_m i eta le C_m}.
\end{proof}

Given $\eta\in Z^1(M,\R)$, we write $\|\ \|_{m,z}=\|\ \|_{m,z\eta}$.  \Cref{p: | alpha |_m i eta} has the following direct consequence.

\begin{cor}\label{c: | alpha |_m z}
For all $\alpha\in H^m(M;\Lambda)$ and $z\in\C$,
\[
\|\alpha\|_{m,z}\le C_m\sum_{k=0}^m|z|^{m-k}\|\alpha\|_k\;,\quad
\|\alpha\|_m\le C_m\sum_{k=0}^m|z|^{m-k}\|\alpha\|_{k,z}\;.
\]
\end{cor}

\begin{prop}\label{p: | alpha |_1 z}
For all $\alpha\in H^1(M;\Lambda)$ and $z\in\C$,
\[
\|\alpha\|_{1,z}\le C\big(\|\alpha\|_{1,i\nu}+|\mu|\|\alpha\|\big)\;,\quad
\|\alpha\|_{1,i\nu}\le C\big(\|\alpha\|_{1,z}+|\mu|\|\alpha\|\big)\;.
\]
\end{prop}

\begin{proof}
By~\eqref{perturbed opers} and~\eqref{| hat c(eta) alpha |_m},
\begin{align*}
\|\alpha\|_{1,z}&=\|\alpha\|+\|D_z\alpha\|
\le\|\alpha\|+\|D_{i\nu}\alpha\|+C'|\mu|\|\alpha\|
\le C\big(\|\alpha\|_{1,i\nu}+|\mu|\|\alpha\|\big)\;,\\
\|\alpha\|_{1,i\nu}&=\|\alpha\|+\|D_{i\nu}\alpha\|
\le\|\alpha\|+\|D_z\alpha\|+C'|\mu|\|\alpha\|
\le C\big(\|\alpha\|_{1,z}+|\mu|\|\alpha\|\big)\;.\;\qedhere
\end{align*}
\end{proof}

\section{Zeta invariants of closed real 1-forms}\label{s: zeta}

\subsection{Preliminaries on asymptotic expansions of heat kernels}\label{ss: heat}

Let $A$ be a positive semi-definite symmetric elliptic differential operator of order $a$, and $B$ a differential operator of order $b$; both of them are defined in $C^\infty(M;E)$ for some Hermitian vector bundle $E$ over $M$. Then $Be^{-tA}$ is a smoothing operator with Schwartz kernel $K_t(x,y)$ in $C^\infty(M^2;E\boxtimes E^*)$ (omitting the Riemannian density $\dvol(y)$ of the second factor). On the diagonal, there is an asymptotic expansion (as $t\downarrow0$) with respect to the semi-norms $\|\ \|_{C^m}$ ($m\in\N_0$) on $C^\infty(M;E\otimes E^*)$ \cite[Lemma~1.9.1]{Gilkey1995}, \cite[Theorem~2.30, Proposition~2.46 and the paragraph that follows]{BerlineGetzlerVergne2004},
\begin{equation}\label{K_t(x x) sim ...}
K_t(x,x)\sim\sum_{l=0}^\infty e_l(x)t^{(l-n-b)/a}\;,
\end{equation}
with $e_l\in C^\infty(M;E\otimes E^*)$. Moreover, using a local system of coordinates, a local trivialization of $E$ and standard multi-index notation, if $B=\sum_\alpha b_\alpha(x)D^\alpha_x$, then $e_l(x)=\sum_\alpha b_\alpha(x)e_{l,\alpha}(x)$, where the $e_{l,\alpha}(x)$ are smooth local invariants of the symbol of $A$ which are homogeneous of degree $l+|\alpha|-b$. They vanish if $l+b$ is odd or if $l+|\alpha|-b<0$. Hence the function
\[
h(t)=\Tr\big(Be^{-tA}\big)=\int_M\tr K_t(x,x)\,\dvol(x)
\]
has an asymptotic expansion
\begin{equation}\label{h(t) sim ...}
h(t)\sim\sum_{l=0}^\infty a_lt^{(l-n-b)/a}\;,
\end{equation}
where
\begin{equation}\label{a_l}
a_l=\int_M\tr e_l(x)\,\dvol(x)\;,
\end{equation}
which vanishes if $l+b$ is odd.

The case of truncated heat kernels, in the following sense, is also needed. Given any $\lambda\ge0$, let $P_{A,\lambda}$ be the spectral projection of $A$ corresponding to $[0,\lambda]$; thus $P_{A,\lambda}^\perp$ is the spectral projection corresponding to $(\lambda,\infty)$. By ellipticity, $P_{A,\lambda}$ is of finite rank, and $Be^{-tA}P_{A,\lambda}$ is a smoothing operator defined for all $t\in\R$. Take any orthonormal frame $\phi_1,\dots,\phi_\kappa$ of $\im P_{A,\lambda}$, consisting of eigensections with corresponding eigenvalues $0\le\lambda_1\le\dots\le\lambda_\kappa\le\lambda$. Then the Schwartz kernel $H_t(x,y)$ of $Be^{-tA}P_{A,\lambda}$ ($t\ge0$) is given by
\[
H_t(x,y)=\sum_{j=1}^\kappa e^{-t\lambda_j}(B\phi_j)(x)\otimes\phi_j(y)\;,
\]
using the isomorphism $E\cong E^*$ given by the Hermitian structure. Thus $H_t(x,y)$ is defined for all $t\in\R$ and smooth. So 
\[
\Tr(Be^{-tA}P_{A,\lambda})=\int_M\tr H_t(x,x)\,\dvol(x)\;.
\]
In particular, for $t=0$, we have
\begin{align}
H_0(x,x)&=\sum_{j=1}^\kappa(B\phi_j)(x)\otimes\phi_j(x)\;,\label{H_0(x y)}\\
\Tr(BP_{A,\lambda})&=\int_M\tr H_0(x,x)\,\dvol(x)\;.\label{Tr(BP_A lambda)}
\end{align}
The Schwartz kernel of $Be^{-tA}P_{A,\lambda}^\perp$ is $\widetilde K_t(x,y)=K_t(x,y)-H_t(x,y)$ ($t>0$), which has an asymptotic expansion
\begin{equation}\label{widetilde K_t(x x) sim ...}
\widetilde K_t(x,x)\sim\sum_{l=0}^\infty\tilde e_l(x)t^{(l-n-b)/a}\;,
\end{equation}
where the first $n+b$ sections $\tilde e_l$ are given by
\[
\tilde e_l(x)=
\begin{cases}
e_l(x) & \text{if $l<n+b$}\\
e_l(x)-H_0(x,x) & \text{if $l=n+b$}\;.
\end{cases}
\]
Then the function
\begin{equation}\label{h_lambda(t)}
\tilde h_\lambda(t)=\Tr\big(Be^{-tA}P_{A,\lambda}^\perp\big)=\Tr\big(Be^{-tA}\big)-\Tr(Be^{-tA}P_{A,\lambda})
\end{equation}
has an asymptotic expansion
\begin{equation}\label{tilde h_lambda(t) sim ...}
\tilde h_\lambda(t)=\int_M\widetilde K_t(x,x)\,\dvol(x)\sim\sum_{l=0}^\infty\tilde a_lt^{(l-n-b)/a}\;,
\end{equation}
where the first $n+b$ coefficients $\tilde a_l$ are given by
\begin{equation}\label{tilde a_l}
\tilde a_l=
\begin{cases}
a_l & \text{if $l<n+b$}\\
a_l-\Tr(BP_{A,\lambda}) & \text{if $l=n+b$}\;.
\end{cases}
\end{equation}

Consider also smooth families of such operators, $\{A_\epsilon\}$ and $\{B_\epsilon\}$, for $\epsilon$ in some parameter space. Then $\Tr(B_\epsilon e^{-tA_\epsilon})$ is smooth in $(t,\epsilon)$, and we add $\epsilon$ to the above notation, writing for instance $K_t(x,y,\epsilon)$, $e_l(x,\epsilon)$, $h(t,\epsilon)$, $a_l(\epsilon)$, $\widetilde K_t(x,y,\epsilon)$, $\tilde e_l(x,\epsilon)$, $\tilde h(t,\epsilon)$ and $\tilde a_l(\epsilon)$ in~\eqref{K_t(x x) sim ...},~\eqref{h(t) sim ...},~\eqref{widetilde K_t(x x) sim ...} and~\eqref{tilde h_lambda(t) sim ...}. The operator $B_\epsilon P_{A_\epsilon,\lambda}$ may not be smooth in $\epsilon$ when some non-constant spectral branch of $\{A_\epsilon\}$ reaches the value $\lambda$. If the values of all non-constant spectral branches of $\{A_\epsilon\}$ stay away from some neighborhood of $\lambda$, then $\tilde h_\lambda(t,\epsilon)$ is smooth in $(t,\epsilon)$.

\subsection{Preliminaries on zeta functions of operators}\label{ss: prelim zeta}

\begin{prop}[See {\cite[Theorems~1.12.2 and~1.12.5]{Gilkey1995}, \cite[Propositions~9.35--9.37]{BerlineGetzlerVergne2004}}] \label{p: zeta functions}
The following holds: 
\begin{enumerate}[{\rm(i)}]
\item\label{i: zeta(s P Q)} For every {$\lambda\in\R$}, there is a meromorphic function $\zeta(s,A,B,\lambda)$ on $\C$ such that, for $\Re s\gg0$, 
\begin{equation}\label{zeta(s A B lambda)}
\zeta(s,A,B,\lambda)=\Tr\big(BA^{-s}P_{A,\lambda}^\perp\big)
=\frac{1}{\Gamma(s)}\int_0^\infty t^{s-1}\tilde h_\lambda(t)\,dt\;.
\end{equation}
\item\label{i: Gamma(s) zeta(s P Q)} The meromorphic function $\Gamma(s)\zeta(s,A,B,\lambda)$ has simple poles at the points $s=(n+b-l)/a$, for $l\in\N_0$ with $\tilde a_l\ne0$. The corresponding residues are $\tilde a_l$, and $\zeta(s,A,B,\lambda)$ is smooth away from these exceptional values of $s$.
\item\label{i: zeta(s A B mu) - zeta(s A B lambda)} For $\mu>\lambda\ge0$, let $\lambda_1\le\dots\le\lambda_k$ denote the eigenvalues of $A$ in $(\lambda,\mu]$, taking multiplicities into account, and let $\psi_1,\dots,\psi_k$ be corresponding orthonormal eigensections. Then, for all $s$,
\[
\zeta(s,A,B,\mu)-\zeta(s,A,B,\lambda)=\sum_{j=1}^k\lambda_k^{-s}\langle B\psi_j,\psi_j\rangle\;.
\] 
\item\label{i: zeta(s P_epsilon Q_epsilon)} For smooth families $\{A_\epsilon\}$ and $\{B_\epsilon\}$ of such operators, if the values of all non-constant branches of eigenvalues of $\{A_\epsilon\}$ stay away from some neighborhood of $\lambda$, then $\zeta(s,A_\epsilon,B_\epsilon,\lambda)$ is smooth in $(s,\epsilon)$ away from the exceptional values of $s$ given in~\ref{i: Gamma(s) zeta(s P Q)}.
\item\label{i: partial/partial epsilon zeta(s P_epsilon Q_epsilon)} Consider the conditions of~\ref{i: zeta(s P_epsilon Q_epsilon)} for $\epsilon$ in some open neighborhood of $0$ in $\R$. If $A_0$ and $B_0$ commute, then
\[
\partial_\epsilon\zeta(s,A_\epsilon,B_\epsilon,\lambda)\big|_{\epsilon=0}
=\zeta(s,A_0,\dot B_0,\lambda)-s\zeta(s+1,A_0,\dot A_0B_0,\lambda)\;,
\]
where the dot denotes $\partial_\epsilon$.
\end{enumerate}
\end{prop}

The last expression of~\eqref{zeta(s A B lambda)} is the Mellin transform of the function $\tilde h_\lambda(t)$ divided by $\Gamma(s)$. This function $\zeta(s,A,B,\lambda)$ is called the \emph{zeta function} of $(A,B,\lambda)$. If $B=1$ or $\lambda=0$, they may be omitted from the notation.

We will also use $\zeta(s,A,B,\lambda)$ when $B$ is not a differential operator, with the same definition. Then the asymptotic expansion~\eqref{tilde h_lambda(t) sim ...} and the properties stated in \Cref{p: zeta functions} need to be checked. With this generality, we can write
\begin{gather*}
\zeta(s,A,B,\lambda)=\zeta(s,A,BP_{A,\lambda}^\perp)=\zeta(s,A,P_{A,\lambda}^\perp B)\;,\\
\zeta(s,A,B)=\zeta(s,A,BP_{A,\lambda})+\zeta(s,A,B,\lambda)\;.
\end{gather*}
Since $P_{A,\lambda}$ is of finite rank, $\zeta(s,A,BP_{A,\lambda})$ is always defined and holomorphic on $\C$.

\subsection{Zeta invariants of closed real 1-forms}\label{ss: zeta invariants}

According to \Cref{p: zeta functions}~\ref{i: zeta(s P Q)}, let
\[
\zeta(s,z)=\zeta(s,z,\eta)=\zeta(s,\Delta_z,{\eta\wedge}\,D_z\sw)\;,
\]
which is a meromorphic function of $s\in\C$. For $\Re s\gg0$,
\begin{align*}
\zeta(s,z)&=\Str\big({\eta\wedge}\,D_z\Delta_z^{-s}\Pi_z^\perp\big)
=\Str\big({\eta\wedge}\,\delta_z\Delta_z^{-s}\Pi_z^1\big)\\
&=\Str\big({\eta\wedge}\,D_z^{-1}\Delta_z^{-s+1}\Pi_z^\perp\big)=\Str\big({\eta\wedge}\,d_z^{-1}\Delta_z^{-s+1}\Pi^1_z\big)\;,
\end{align*}
using that ${\eta\wedge}\,d_z$ and ${\eta\wedge}\,\delta_z^{-1}$ change the degree of homogeneous forms. So, when $\zeta(s,z)$ is regular at $s=1$, the value $\zeta(1,z)$ is a renormalized version of the super-trace of ${\eta\wedge}\,d_z^{-1}\Pi^1_z$, which is called the \emph{zeta invariant} of $(M,g,\eta,z)$ for the scope of this paper. According to \Cref{p: zeta functions}~\ref{i: Gamma(s) zeta(s P Q)} and since $\Gamma(s)$ is regular at $s=1$, $\zeta(s,z)$ might have a simple pole at $s=1$. But it will be shown that $\zeta(s,z)$ is regular at $s=1$ for all $\eta\in Z^1(M,\R)$ and $z\in\C$ (\Cref{c: zeta(s z) = ...}).

\subsection{Heat invariants of perturbed operators}\label{ss: heat invs of perturbed ops}

Consider the notation of \Cref{sss: perturbations}. For $k=0,\dots,n$, let $K_{z,k,t}(x,y)$ denote the Schwartz kernel of $e^{-t\Delta_{z,k}}$. By~\eqref{K_t(x x) sim ...}, its restriction to the diagonal has an asymptotic expansion (as $t\downarrow0$),
\[
K_{z,k,t}(x,x)\sim\sum_{l=0}^\infty e_{k,l}(x,z)t^{(l-n)/2}\;,
\]
where every $e_{k,l}(x,z)$ is a smooth local invariant of $z$ and the jets of the local coefficients of $g$ and $\eta$, which is homogeneous of degree $l$, and vanishes if $l$ is odd. According to \eqref{h(t) sim ...} and~\eqref{a_l},
\[
h_k(t,z):=\Tr\big(e^{-t\Delta_{z,k}}\big)\sim\sum_{l=0}^\infty a_{k,l}(z)t^{(l-n)/2}\;,
\]
where
\[
a_{k,l}(z)=\int_M\str e_{k,l}(x,z)\,\dvol(x)\;.
\]

The Schwartz kernel of $e^{-t\Delta_z}\sw$ is 
\[
K_{z,t}(x,y)=\sum_{k=0}^n(-1)^kK_{z,k,t}(x,y)\;.
\]
We have induced asymptotic expansions,
\begin{gather*}
K_{z,t}(x,x)\sim\sum_{l=0}^\infty e_l(x,z)t^{(l-n)/2}\;,\\
h(t,z):=\Str\big(e^{-t\Delta_z}\big)\sim\sum_{l=0}^\infty a_l(z)t^{(l-n)/2}\;,
\end{gather*}
where
\[
e_l(x,z)=\sum_{k=0}^n(-1)^ke_{k,l}(x,z)\;,\quad a_l(z)=\sum_{k=0}^n(-1)^ka_{k,l}(z)\;.
\]

\begin{thm}[{\cite[Theorem 13.4]{BismutZhang1992}; see also \cite[Theorem~1.5]{AlvGilkey2021a} and \cite{AlvKordyLeichtnam-atffff}}]\label{t: e_l(x z)}
We have:
\begin{enumerate}[{\rm(i)}]
\item\label{i: e_l(x,z) = 0} $e_l(x,z)=0$ for $l<n$; and,
\item\label{i: e_l(x,z) = EE_n(x)} if $n$ is even, then $e_n(x,z)=e(M,\nabla^M)(x)$.
\end{enumerate}
\end{thm}

\begin{rem}\label{r: e_l(x z)}
The analog of \Cref{t: e_l(x z)} fails for Witten's type perturbations of the Dolbeault complex on K\"ahler manifolds \cite{AlvGilkey2021b}.
\end{rem}

\subsection{Derived heat invariants of perturbed operators}\label{ss: derived heat invs of perturbed ops}

The following are sometimes called the \emph{derived heat density} and \emph{derived heat invariant} of order $l$ of $d_z$ or $\Delta_z$ \cite{GuntherSchimming1977}, \cite{RaySinger1971}, \cite[page~181]{Gilkey1995}, \cite{AlvGilkey2023}:
\begin{gather*}
\fe_l(x,z)=\sum_{k=0}^n(-1)^kke_{k,l}(x,z)\;,\\
\fa_l(z)=\sum_{k=0}^n(-1)^kka_{k,l}(z)=\int_M\str\fe_l(x,z)\,\dvol(x)\;.
\end{gather*}
We have
\begin{equation}\label{Str(sN e^-t Delta_z) sim ...}
\Str\big(\sN e^{-t\Delta_z}\big)\sim\sum_{l=0}^\infty\fa_l(z)t^{(l-n)/2}\;.
\end{equation}

\begin{thm}[{\cite[Theorem~7.10]{BismutZhang1992}}]\label{t: fa_l(z) is independent of z}
For $l\le n$, $\fa_l(z)$ is independent of $z$.
\end{thm}

\begin{rem}\label{r: fa_l(z) is independent of z}
\cite[Theorem~7.10]{BismutZhang1992} gives \Cref{t: fa_l(z) is independent of z} for real $z$. But, since the functions $\fe_l(x,z)$ have local expressions, we can assume $\eta$ is exact. Then the result can be extended to non-real $z$ using~\eqref{Witten's opers}. The exactness of $\eta$ in \cite[Theorem~7.10]{BismutZhang1992} is irrelevant because a general flat vector bundle is considered. Moreover \cite[Theorem~7.10]{BismutZhang1992} gives an explicit expression of $\fa_l(z)$ for $l\le n$.
\end{rem}

\begin{rem}\label{r: fe_l(x z)}
A refinement of \Cref{t: fa_l(z) is independent of z} is given in \cite[Theorem~1.3~(1b)]{AlvGilkey2023}, where $\fe_l(x,z)$ is described for $l\le n$, showing its independence of $z$.
\end{rem}

\subsection{Regularity}\label{ss: regularity}

By~\eqref{h(t) sim ...} and~\eqref{a_l}, we have an asymptotic expansion of the form
\begin{equation}\label{Str(eta wedge delta_z e^-t Delta_z) sim ...}
\Str\big({\eta\wedge}\,D_ze^{-t\Delta_z}\big)\sim\sum_{l=0}^\infty b_l(z)t^{(l-n-1)/2}\;,
\end{equation}
where $b_l(z)=0$ if $l$ is even.



\begin{prop}\label{p: partial_z Str(sN e^-t Delta_z) = -t Str(eta wedge D_z e^-t Delta_z)}
For all $t>0$ and $z\in\C$, the equality~\eqref{partial_z Str(fN e^-t Delta_z) = -t Str(eta wedge D_z e^-t Delta_z)} is true.
\end{prop}

\begin{proof}
For all $k$, we have \cite[Corollary~2.50]{BerlineGetzlerVergne2004}
\[
\partial_z\Tr\big(e^{-t\Delta_{z,k}}\big)=-t\Tr\big((\partial_z\Delta_{z,k})e^{-t\Delta_{z,k}}\big)\;.
\]
So, by~\eqref{sN T = T (sN + l)} and~\eqref{partial_z delta_z},
\begin{align*}
\partial_z\Str\big(\sN e^{-t\Delta_z}\big)&=-t\Str\big(\sN(\partial_z\Delta_z)e^{-t\Delta_z}\big)\\
&=-t\Str\big(\sN\,{\eta\wedge}\,\delta_ze^{-t\Delta_z}\big)-t\Str\big(\sN\delta_z\,{\eta\wedge}\,e^{-t\Delta_z}\big)\\
&=-t\Str\big(\sN\,{\eta\wedge}\,\delta_ze^{-t\Delta_z}\big)-t\Str\big(\delta_z(\sN-1)\,{\eta\wedge}\,e^{-t\Delta_z}\big)\\
&=-t\Str\big(\sN\,{\eta\wedge}\,\delta_ze^{-t\Delta_z}\big)+t\Str\big((\sN-1)\,{\eta\wedge}\,\delta_ze^{-t\Delta_z}\big)\\
&=-t\Str\big({\eta\wedge}\,D_ze^{-t\Delta_z}\big)\;.\qedhere
\end{align*}
\end{proof}

\begin{cor}\label{c: b_l(z) = 0}
For $l\le  n-1$, $b_l(z)=0$.
\end{cor}

\begin{proof}
By~\eqref{Str(sN e^-t Delta_z) sim ...},~\eqref{Str(eta wedge delta_z e^-t Delta_z) sim ...}, \Cref{t: fa_l(z) is independent of z,p: partial_z Str(sN e^-t Delta_z) = -t Str(eta wedge D_z e^-t Delta_z)}, for $l\le  n-1$,
\[
b_l(z)=-\partial_z\fa_{l+1}(z)=0\;.\qedhere
\]
\end{proof}

\begin{cor}\label{c: zeta(s z) = ...}
If $n$ is even and $\Re s>0$, or $n$ is odd and $\Re s>1/2$, then
\[
\zeta(s,z)=\frac{1}{\Gamma(s)}\int_0^\infty t^{s-1}\Str\big({\eta\wedge}\,D_ze^{-t\Delta_z}\big)\,dt\;,
\]
where the integral is absolutely convergent, and therefore $\zeta(s,z)$ is smooth in this half-plane.
\end{cor}

\begin{proof}
By~\eqref{Str(eta wedge delta_z e^-t Delta_z) sim ...} and \Cref{c: b_l(z) = 0},
\begin{equation}\label{Str(theta wedge D_z e^-t Delta_z) = O(t^-1/2)}
\Str\big({\eta\wedge}\,D_ze^{-t\Delta_z}\big)=
\begin{cases}
O(1) & \text{if $n$ is even}\\
O\big(t^{-1/2}\big) & \text{if $n$ is odd}
\end{cases}
\quad(t\downarrow0)\;.
\end{equation}
On the other hand, there is some $c>0$ such that
\begin{equation}\label{Str(theta wedge D_z e^-t Delta_z) = O(e^-ct)}
\Str\big({\eta\wedge}\,D_ze^{-t\Delta_z}\big)=O(e^{-ct})\quad(t\uparrow+\infty)\;.
\end{equation}
So the stated integral is absolutely convergent for $\Re s>0$ if $n$ is even, or for $\Re s>1/2$ if $n$ is odd, defining a holomorphic function of $s$ on this half-plane. Then the stated equality is true because it holds for $\Re s\gg0$.
\end{proof}

\begin{rem}\label{r: regularity}
From \Cref{p: zeta functions}~\ref{i: Gamma(s) zeta(s P Q)} and \Cref{c: b_l(z) = 0}, it also follows that, if $n$ is even (resp., odd), then $\zeta(s,z)$ is smooth on $\C$ (resp., on $\C\setminus((1-\N_0)/2)$). But this additional regularity is not needed in this work.
\end{rem}

\begin{cor}\label{c: zeta(1 z) = ...}
For all $z\in\C$,
\[
\zeta(1,z)=\lim_{t\downarrow0}\Str\big({\eta\wedge}\,D_z^{-1}e^{-t\Delta_z}\Pi_z^\perp\big)\;.
\]
\end{cor}

\begin{proof}
By \Cref{c: zeta(s z) = ...},~\eqref{Str(theta wedge D_z e^-t Delta_z) = O(t^-1/2)} and~\eqref{Str(theta wedge D_z e^-t Delta_z) = O(e^-ct)}, and since
\[
\Str\big({\eta\wedge}\,D_z^{-1}e^{-t\Delta_z}\Pi_z^\perp\big)=O(e^{-ct})\quad(t\uparrow+\infty)\;,
\]
we get
\begin{align*}
\zeta(1,z)&=\int_0^\infty\Str\big({\eta\wedge}\,D_ze^{-u\Delta_z}\Pi_z^\perp\big)\,du
=\lim_{t\downarrow0}\int_t^\infty\Str\big({\eta\wedge}\,D_ze^{-u\Delta_z}\Pi_z^\perp\big)\,du\\
&=\lim_{t\downarrow0}\Str\big({\eta\wedge}\,D_z^{-1}e^{-t\Delta_z}\Pi_z^\perp\big)\;.\qedhere
\end{align*}
\end{proof}

\Cref{c: zeta(s z) = ...,c: zeta(1 z) = ...} give \Cref{t: zeta(1 z)}.

\subsection{The case of the differential of a function}\label{ss: differential of a function}

Let us consider the special case where $\eta=dh$ for a smooth real-valued function $h$.

\begin{lem}\label{l: Trs(eta wedge d_z^-1 e^-t Delta_z Pi^1_z) = -Trs(h e^-t Delta_z Pi^perp_z)}
We have
\[
\Str\big({\eta\wedge}\,d_z^{-1}e^{-t\Delta_z}\Pi^1_z\big)=-\Str\big(h\,e^{-t\Delta_z}\Pi_z^\perp\big)\;.
\]
\end{lem}

\begin{proof}
Since ${\eta\wedge}=[d,h]$,
\begin{align*}
\Str\big({\eta\wedge}\,d_z^{-1}e^{-t\Delta_z}\Pi^1_z\big)
&=\Str\big([d_z,h]\,d_z^{-1}e^{-t\Delta_z}\Pi^1_z\big)\\
&=\Str\big(d_z\,h\,d_z^{-1}e^{-t\Delta_z}\Pi^1_z\big)-\Str\big(h\,d_zd_z^{-1}e^{-t\Delta_z}\Pi^1_z\big)\\
&=-\Str\big(h\,d_z^{-1}e^{-t\Delta_z}\Pi^1_zd_z\big)-\Str\big(h\,e^{-t\Delta_z}\Pi^1_z\big)\\
&=-\Str\big(h\,d_z^{-1}d_ze^{-t\Delta_z}\Pi^2_z\big)-\Str\big(h\,e^{-t\Delta_z}\Pi^1_z\big)\\
&=-\Str\big(h\,e^{-t\Delta_z}\Pi^2_z\big)-\Str\big(h\,e^{-t\Delta_z}\Pi^1_z\big)\\
&=-\Str\big(h\,e^{-t\Delta_z}\Pi_z^\perp\big)\;.\qedhere
\end{align*}
\end{proof}

\begin{cor}\label{c: zeta(1 z) = - lim_t->0 Str(h e^-t Delta_z Pi_z^perp)}
We have
\[
\zeta(1,z)=-\lim_{t\downarrow0}\Str\big(h\,e^{-t\Delta_z}\Pi_z^\perp\big)\;.
\]
\end{cor}

\begin{proof}
Apply~\Cref{c: zeta(1 z) = ...,l: Trs(eta wedge d_z^-1 e^-t Delta_z Pi^1_z) = -Trs(h e^-t Delta_z Pi^perp_z)}.
\end{proof}

\begin{cor}\label{c: zeta(1 Delta_z theta wedge D_z bfw) in R}
We have $\zeta(1,z)\in\R$.
\end{cor}

\begin{proof}
By \Cref{c: zeta(1 z) = - lim_t->0 Str(h e^-t Delta_z Pi_z^perp)}, it is enough to prove that $\Str(h\,e^{-t\Delta_z}\Pi_z^\perp)\in\R$. But, taking adjoints,
\[
\Str\big(h\,e^{-t\Delta_z}\Pi_z^\perp\big)=\overline{\Str\big(\Pi_z^\perp e^{-t\Delta_z}\,h\big)}
=\overline{\Str\big(h\,\Pi_z^\perp e^{-t\Delta_z}\big)}
=\overline{\Str\big(h\,e^{-t\Delta_z}\Pi_z^\perp\big)}\;.\qedhere
\]
\end{proof}

\begin{cor}\label{c: zeta(1 Delta_z theta wedge D_z bfw) = zeta(1 Delta_-bar z theta wedge D_- bar z bfw)}
If $M$ is oriented, then
\[
\zeta(1,z)=\zeta(1,-\bar z)=\zeta(1,-z)=\zeta(1,\bar z)\;.
\]
\end{cor}

\begin{proof}
By~\eqref{star z},
\begin{align*}
\Str\big(h\,e^{-t\Delta_z}\Pi_z^\perp\big)&=\Str\big(\star\star^{-1}h\,e^{-t\Delta_z}\Pi_z^\perp\big)
=\Str\big(\star^{-1}h\,e^{-t\Delta_z}\Pi_z^\perp\star\big)\\
&=\Str\big(\star^{-1}\star h\,e^{-t\Delta_{-\bar z}}\Pi_{-\bar z}^\perp\big)
=\Str\big(h\,e^{-t\Delta_{-\bar z}}\Pi_{-\bar z}^\perp\big)\;.
\end{align*}
Thus the first equality of the statement holds by \Cref{c: zeta(1 z) = - lim_t->0 Str(h e^-t Delta_z Pi_z^perp)}. The second equality follows with a similar argument, using $\bar\star$ instead of $\star$. The third equality is equivalent to the first one.
\end{proof}

\section{Small and large complexes of Morse forms}

\subsection{Preliminaries on Morse forms}\label{ss: Morse forms}

Recall that a critical point $p$ of any $h\in C^\infty(M,\R)$ is called \emph{nondegenerate} if the symmetric bilinear form $\Hess_ph$ on $T_pM$ is nondegenerate; then the index of $\Hess_ph$ is denoted by $\ind(p)$. By the Morse lemma \cite[Lemma~2.2]{Milnor1963}, this means that
\begin{gather}
h-h(p)=\frac{1}{2}\sum_{j=1}^n\epsilon_{p,j}(x_p^j)^2=\frac{1}{2}\big(|x_p^+|^2-|x_p^-|^2\big)\;,\label{h - h(p) around p}\\
\intertext{where}
\epsilon_{p,j}=
\begin{cases}
-1 & \text{if $j\le\ind(p)$}\\
1 & \text{if $j>\ind(p)$}\;,
\end{cases}
\label{epsilon_p,j}
\end{gather}
on some chart $(U_p,x_p=(x_p^1,\dots,x_p^n))$ (centered) at $p$ (\emph{Morse coordinates}), where $x_p^-=(x_p^1,\dots,x_p^{\ind(p)})$ and $x_p^+=(x_p^{\ind(p)+1},\dots,x_p^n)$.

Recall that $h$ is called a \emph{Morse function} when all of its critical points are nondegenerate. Then its critical points form a finite set denoted by $\Crit(h)$. The Morse functions form an open and dense subset of $C^\infty(M,\R)$ \cite[Theorem~6.1.2]{Hirsch1976}. On every $U_p$, we can assume the metric is Euclidean with respect to Morse coordinates:
\begin{equation}\label{g around p}
g=\sum_{j=1}^n(dx_p^j)^2\;.
\end{equation}

Now take any $\eta\in Z^1(M,\R)$. We can show that if $p$ is a zero of $\eta$, then $(\nabla\eta)_p$ is independent of the choice of the connection $\nabla$, and is symmetric. The zero $p$ is called \emph{nondegenerate} of \emph{index} $k$ if $(\nabla\eta)_p$ is nondegenerate of index $k$. In this case, any local primitive $h_{\eta,p}$ of $\eta$ near $p$ is a Morse function, and we can choose it so that $h_{\eta,p}(p)=0$. On a domain $U_p$ of Morse coordinates $x_p=(x_p^1,\dots,x_p^n)$ for $h_{\eta,p}$ at $p$, also called \emph{Morse coordinates} for $\eta$ at $p$, $h_{\eta,p}$ is given by the center and right-hand side of~\eqref{h - h(p) around p}, and
\begin{equation}\label{eta around p}
\eta=\sum_{j=1}^n\epsilon_{p,j}x_p^j\,dx_p^j\;.
\end{equation}
If all zeros are nondegenerate, then $\eta$ is called a \emph{Morse form}. In this case, its zeros form a finite set, $\XX=\Zero(\eta)$; subsets of $\XX$ defined by conditions on the index are denoted by writing the conditions as subscripts; for instance, $\XX_k$, $\XX_+$ and $\XX_{<k}$ are the subsets of zeros of index $k$, of positive index, and of index${}<k$, respectively. For any $\xi\in H^1(M,\R)$, the Morse representatives of $\xi$ form a dense open subset of $\xi$, considering $\xi\subset\Omega^1(M,\R)$ with the $C^\infty$ topology (see e.g.\ \cite[Theorem~2.1.25]{Pajitnov2006}). If $\xi=0$, this is just the classical property of Morse functions mentioned before.

From now on, unless otherwise stated, we will use some $\eta\in Z^1(M,\R)$ and a Riemannian metric $g$ on $M$ satisfying~\ref{i-a:  eta Morse} (\Cref{ss: intro-Witten}).

The Hopf index of $\eta^\sharp$ at any $p\in\XX_k$ is $(-1)^k$ (\Cref{sss: Morse-type zeros}). Thus, by the Hopf index theorem,
\begin{equation}\label{sum_k=0^n (-1)^k | XX_k | = chi(M)}
\sum_{k=0}^n(-1)^k|\XX_k|=\chi(M)\;.
\end{equation}

\subsection{The small and large spectrum}\label{ss: sm and la spec}

Consider the perturbed operators~\eqref{perturbed opers} defined by $\eta$ and $g$. We can suppose the closures $\overline{U_p}$ ($p\in\XX$) are disjoint from each other, and $x_p(U_p)=(-4r,4r)^n$ for some $r>0$ independent of $p$ with $4r<1$. Let $U=\bigcup_{p\in\XX}U_p$.

Denoting also the coordinates of $\R^n$ by $(x_p^1,\dots,x_p^n)$, consider the function $h_p\in C^\infty(\R^n)$ defined by the center and right-hand side of~\eqref{h - h(p) around p}. Let $d'_{p,z}$, $\delta'_{p,z}$, $D'_{p,z}$ and $\Delta'_{p,z}$ ($z\in\C$) denote the corresponding Witten's operators on $\R^n$, whose restrictions to $(-4r,4r)^n$ agree via $x_p$ with $d_z$, $\delta_z$, $D_z$ and $\Delta_z$ on $U_p$.

\begin{prop}[See e.g.\ {\cite[Chapters~9 and~14]{Roe1998}, \cite[Sections~4.5 and~4.7]{Zhang2001}}]\label{p: Delta'_p mu}
The following holds for $\mu\in\R$:
\begin{enumerate}[{\rm(i)}]

\item\label{i: Delta'_p mu = ...} We have
\begin{equation}\label{Delta_mu on U_p}
\Delta'_{p,\mu}=\sum_{j=1}^n\Big(-\Big(\frac{\partial}{\partial x_p^j}\Big)^2+\mu^2(x_p^j)^2+\mu\epsilon_{p,j}[{dx_p^j\lrcorner},{dx_p^j\wedge}]\Big)\;.
\end{equation}
Here $[{\cdot},{\cdot}]$ stands for the commutator of operators. Using multi-index notation, we can write
\[
[{dx_p^j\lrcorner},{dx_p^j\wedge}]dx_p^J=
\begin{cases}
dx_p^J & \text{if $j\in J$}\\
-dx_p^J & \text{if $j\notin J$}\;.
\end{cases}
\]

\item\label{i: Delta'_p mu is ...} $\Delta'_{p,\mu}$ is a non-negative selfadjoint operator in $L^2(\R^n;\Lambda)$ with a discrete spectrum, which consists of the eigenvalues
\begin{equation}\label{eigenvalues of Delta'_p mu}
\mu\sum_{j=1}^n(1+2u_j+\epsilon_{p,j}v_j)\;,
\end{equation}
where $u_j\in\N_0$ and $v_j=\pm1$. For the restriction of $\Delta'_{p,\mu}$ to $k$-forms, the spectrum has the additional requirement that exactly $k$ of the numbers $v_j$ are equal to $1$. In particular, $0$ is an eigenvalue of $\Delta'_{p,\mu}$ with multiplicity $1$ {\rm(}choosing $u_j=0$ and $v_j=-\epsilon_{p,j}$ for all $j${\rm)}, and the nonzero eigenvalues are of order $\mu$ as $\mu\to+\infty$. $D'_{p,\mu}$ is also a selfadjoint operator in $L^2(\R^n;\Lambda)$ with a discrete spectrum, which consists of the positive and negative square roots of~\eqref{eigenvalues of Delta'_p mu}.

\item\label{i: e'_p mu} The kernel of $D'_{p,\mu}$ and $\Delta'_{p,\mu}$ is generated by the normalized form
\[
e'_{p,\mu}=\Big(\frac\mu\pi\Big)^{n/4}e^{-\mu|x_p|^2/2}\,dx_p^1\wedge\dots\wedge dx_p^{\ind(p)}\;.
\]

\end{enumerate}
\end{prop}

For any $z\in\C$ with $\mu>0$, let $\Delta'_{p,z}=e^{-i\nu h_p}\Delta'_{p,\mu}e^{i\nu h_p}$. Since the operator of multiplication by $e^{-i\nu h_p}$ is unitary, $\Delta'_{p,z}$ is also selfadjoint and non-negative in $L^2(\R^n;\Lambda)$, it has a discrete spectrum with the same eigenvalues and multiplicities as $\Delta'_{p,\mu}$, and its kernel is generated by the normalized form $e'_{p,z}:=e^{-i\nu h_p}e'_{p,\mu}$. We will also use the notation
\[
e'_{p,z}=x_p^*e'_{p,z}\in C^\infty\big(U_p;\Lambda^{\ind(p)}\big)\;.
\]
The function $x_p^*h_p\in C^\infty(U_p)$ agrees with $h_{\eta,p}$, which is also denoted by $h_p$ in this section.

Fix an even $C^\infty$ function $\rho:\R\to[0,1]$ such that $\rho=1$ on $[-r,r]$ and $\supp\rho\subset[-2r,2r]$. For every $p\in\XX$, let 
\begin{align}
\rho_p&=\rho(x_p^1)\cdots\rho(x_p^n)\in\Cinftyc(U_p)\;,\label{rho_p}\\
e_{p,\mu}&=\frac{\rho_p}{a_\mu}e'_{p,\mu}\in\Cinftyc\big(U_p;\Lambda^{\ind(p)}\big)\;,\label{e_p mu}\\
e_{p,z}&=e^{-i\nu h_p}e_{p,\mu}=\frac{\rho_p}{a_\mu}e'_{p,z}\in\Cinftyc\big(U_p;\Lambda^{\ind(p)}\big)\;,
\label{e_p z}
\end{align}
where
\begin{equation}\label{a_mu}
a_\mu=\bigg(\int_{-2r}^{2r}\rho(x)^2e^{-\mu x^2}\,dx\bigg)^{\frac n2}=\Big(\frac\pi\mu\Big)^{\frac n4}+O(e^{-c\mu})\;,
\end{equation}
as $\mu\to+\infty$. The extensions by zero of the forms $e_{p,z}$ to $M$ are also denoted by $e_{p,z}$. They form an orthonormal basis of a graded subspace $E_z\subset \Omega(M)$ with $\dim E_z=|\XX|$. Observe that $d_z$ does not preserve $E_z$, so that $E_z$ is not a subcomplex of $(\Omega(M), d_z)$. Let $P_z$ be the orthogonal projection of $L^2(M;\Lambda)$ to $E_z$.

\begin{rem}\label{r: | mu |}
For the sake of simplicity, most of our results are stated for $\mu\gg0$ or as $\mu\to+\infty$, but they have obvious versions for $\mu\ll0$ or as $\mu\to-\infty$, as follows by considering $-\eta$ and using that $\XX_k(-\eta)=\XX_{n-k}(\eta)$.
\end{rem}

\begin{prop}\label{p: | D_z beta | ge C sqrt | mu | | beta |}
If $\mu\gg0$ and $\beta\in H^1(M;\Lambda)$ with $\supp\beta\subset M\setminus U$, then
\[
\|D_z\beta\|\ge C\mu\,\|\beta\|\;.
\]
\end{prop}

\begin{proof}
This follows like \cite[Proposition~4.7]{Zhang2001}, using that $\sH_\eta$ is of order zero in~\eqref{perturbed opers}. Actually, according to the statement of \cite[Proposition~4.7]{Zhang2001}, this inequality would hold with $\sqrt{\mu}$ instead of $\mu$, but its proof clearly shows that using $\mu$ is fine.
\end{proof}

\begin{prop}\label{p: P_z D_z P_z = 0}
The following properties hold:
\begin{enumerate}[{\rm(i)}]

\item\label{i: P_z D_z P_z = 0} $P_zD_zP_z=0$.

\item\label{i: | P_z^perp D_z alpha | le e^-c | mu | | alpha |} If $\mu\gg0$, $\alpha\in E_z$ and $\beta\in E_z^\perp\cap H^1(M;\Lambda)$, then
\[
\|P_z^\perp D_z\alpha\|\le e^{-c\mu}\|\alpha\|\;,\quad\|P_zD_z\beta\|\le e^{-c\mu}\|\beta\|\;.
\]

\item\label{i: | P_z^perp D_z beta | ge C sqrt | mu | | beta |} If $\mu\gg0$ and $\beta\in E_z^\perp\cap H^1(M;\Lambda)$, then
\[
\|P_z^\perp D_z\beta\|\ge C\sqrt{\mu}\,\|\beta\|\;.
\]

\end{enumerate}
\end{prop}

\begin{proof}
This follows like \cite[Propositions~4.11,~4.12 and~5.6]{Zhang2001}. Property~\ref{i: P_z D_z P_z = 0} is true because every $D_z e_{p,z}$ is supported in $U_p$ and has homogeneous components of degree different from $\ind(p)$; therefore it is orthogonal to $\ker\Delta_z$. The other properties are consequences of~\Cref{p: Delta'_p mu,p: | D_z beta | ge C sqrt | mu | | beta |} and~\eqref{rho_p}--\eqref{a_mu}. According to \cite[Proposition~4.11]{Zhang2001}, the inequalities of~\ref{i: | P_z^perp D_z alpha | le e^-c | mu | | alpha |} hold with $1/\mu$ instead of $e^{-c\mu}$, but its proof shows that indeed $e^{-c\mu}$ can be achieved. 
\end{proof}

\begin{prop}\label{p: | D_z^l e_p z |_m i nu}
For all $m\in\N_0$, if $\mu\gg0$, then
\[
\|D_ze_{p,z}\|_m\le |\nu|^me^{-c_m\mu}\;,\quad\|D_ze_{p,z}\|_{m,i\nu}\le e^{-c_m\mu}\;.
\]
\end{prop}

\begin{proof}
From \Cref{p: Delta'_p mu}~\ref{i: e'_p mu},~\eqref{[D h] = hat c(dh)},~\eqref{e_p mu} and~\eqref{e_p z}, we get
\begin{equation}\label{D_z e_p,z}
D_ze_{p,z}=D_z\Big(\frac{\rho_p}{a_\mu}e'_{p,z}\Big)
=e^{-i\nu h_p}\frac1{a_\mu}\Big(\frac{\pi}4\Big)^{n/4}\hat c(d\rho_p)e'_{p,\mu}\;.
\end{equation}
Thus the stated estimate of $\|D_ze_{p,z}\|_m$ is true by~\eqref{e_p mu} and~\eqref{a_mu}, since $d\rho_p=0$ around $p$, and using the definition of $h_p$ and the condition $4r<1$. (When $\nu=0$, this is indicated in \cite[Eq.~(6.17)]{Zhang2001}.)

By~\eqref{Witten's opers}, for all $k\in\N_0$ and $p\in\XX$, the form $D_{i\nu}^kD_ze_{p,z}$ is the extension by zero of the form $e^{-i\nu h_p}D^kD_\mu e_{p,\mu}$ on $U_p$. Then the stated estimate of $\|D_ze_{p,z}\|_{m,i\nu}$ follows from the case $\nu=0$.
\end{proof}

\begin{prop}\label{p: | D_ze_p z |_L^infty}
If $\mu\gg0$, then
\[
\|D_ze_{p,z}\|_{L^\infty}\le e^{-c\mu}\;.
\]
\end{prop}

\begin{proof}
Apply~\eqref{e_p mu} and~\eqref{a_mu} in~\eqref{D_z e_p,z}, and use that $d\rho_p=0$ around $p$.
\end{proof}

Consider the partition of $\spec\Delta_z$ into its intersections with $[0,1]$ and $(1,\infty)$, called the \emph{small} and \emph{large spectrum}; the term \emph{small/large eigenvalues} may be also used. Let $E_{z,\text{\rm sm}}\subset \Omega(M)$ denote the graded finite dimensional subspace generated by the eigenforms of the small eigenvalues, let $E_{z,\text{\rm la}}=E_{z,\text{\rm sm}}^\perp$ in $L^2(M;\Lambda)$, and let $P_{z,\text{\rm sm/la}}$ be the orthogonal projection to $E_{z,\text{\rm sm/la}}$, called \emph{small/large projection}. Moreover $(\Omega(M),d_z)$ splits into a topological direct sum of the subcomplexes $E_{z,\text{\rm sm}}$ and $E_{z,\text{\rm la}}\cap \Omega(M)$, called the \emph{small} and \emph{large complexes}, and~\eqref{Hodge dec Novikov} gives
\begin{equation}\label{H^bullet(Omega_sm z(M C) d_z) cong H_z^*(M C)}
H^\bullet(E_{z,\text{\rm sm}},d_z)\cong H_z^\bullet(M)\;,\quad H^\bullet(E_{z,\text{\rm la}}\cap \Omega(M),d_z)=0\;.
\end{equation}
For any operator $B$ defined on $\Omega(M)$ or $L^2(M;\Lambda)$, let $B_{z,\text{\rm sm/la}}=BP_{z,\text{\rm sm/la}}$.  

\begin{prop}\label{p: | alpha - P_z sm alpha |_m i nu}
For all $m\in\N_0$, $\mu\gg0$ and $\alpha\in E_z$, 
\[
\|\alpha-P_{z,\text{\rm sm}}\alpha\|_{m,i\nu}\le e^{-c_m\mu}\|\alpha\|\;.
\]
\end{prop}

\begin{proof}
This follows like \cite[Lemma~5.8 and Theorem~6.7]{Zhang2001}, using $\|\ \|_{m,i\nu}$ instead of $\|\ \|_m$. The following are the main steps of the proof. 

Let $\S^1=\{\,\omega\in\C\mid|\omega|=1\,\}$. With the argument of the proof of \cite[Eq.~(5.27)]{Zhang2001}, using \Cref{p: P_z D_z P_z = 0}, we get that, for all $\alpha\in H^1(M;\Lambda)$, $w\in\S^1$ and $\mu\gg0$,
\[
\|(w-D_z)\alpha\|\ge C\|\alpha\|\;.
\]
Thus $w-D_z:H^1(M;\Lambda)\to L^2(M;\Lambda)$ is bijective, and, for all $\beta\in L^2(M;\Lambda)$, $w\in\S^1$ and $\mu\gg0$,
\begin{equation}\label{| (w - D_z)^-1 beta |}
\big\|(w-D_z)^{-1}\beta\big\|\le C^{-1}\|\beta\|\;.
\end{equation}

On the other hand, arguing like in the proof of \cite[Eq.~(6.18)]{Zhang2001}, it follows that, for all $\gamma\in H^m(M;\Lambda)$, $w\in\S^1$ and $\mu\gg0$,
\[
\|\gamma\|_{m,i\nu}\le C_m\big(\|(w-D_z)\gamma\big\|_{m-1,i\nu}+\mu\|\gamma\|_{m-1,i\nu}+\|\gamma\|\big)\;.
\]
Continuing by induction on $m\in\N_0$, we obtain
\[
\|\gamma\|_{m,i\nu}\le C_m\Big(\mu^m\|\gamma\|+\sum_{k=1}^m\mu^{k-1}\|(w-D_z)\gamma\big\|_{m-k,i\nu}\Big)\;.
\]
In other words, for all $\beta\in H^{m-1}(M;\Lambda)$,
\[
\big\|(w-D_z)^{-1}\beta\big\|_{m,i\nu}\le C_m\Big(\mu^m\big\|(w-D_z)^{-1}\beta\big\|+\sum_{k=1}^m\mu^{k-1}\|\beta\|_{m-k,i\nu}\Big)\;.
\]
Applying~\eqref{| (w - D_z)^-1 beta |} to this inequality, we get, for $m\ge1$,
\begin{equation}\label{|(w - D_z)^-1 beta |_m i nu}
\big\|(w-D_z)^{-1}\beta\big\|_{m,i\nu}\le C_m\mu^m\|\beta\|_{m-1,i\nu}\;.
\end{equation}
From~\eqref{| (w - D_z)^-1 beta |},~\eqref{|(w - D_z)^-1 beta |_m i nu} and \Cref{p: | D_z^l e_p z |_m i nu}, it follows that, for $m\in\N_0$,
\begin{equation}\label{| (w-D_z)^-1 D_z e_p z |_m i nu}
\big\|(w-D_z)^{-1}D_z e_{p,z}\big\|_{m,i\nu}=O\big(e^{-c_m\mu}\big)
\end{equation}
 as $\mu\to+\infty$, uniformly on $w\in\S^1$. But, endowing $\S^1$ with the counter-clockwise orientation, basic spectral theory gives (see e.g.\ \cite[Section~VII.3]{DunfordSchwartz1988-I})
\begin{multline}\label{P_z sm e_p z - e_p z}
P_{z,\text{\rm sm}} e_{p,z}- e_{p,z}=\frac{1}{2\pi i}\int_{\S^1}\big((w-D_z)^{-1}-w^{-1}\big) e_{p,z}\,dw\\
=\frac{1}{2\pi i}\int_{\S^1}w^{-1}(w-D_z)^{-1}D_z e_{p,z}\,dw\;.
\end{multline}
The result follows using~\eqref{| (w-D_z)^-1 D_z e_p z |_m i nu} in~\eqref{P_z sm e_p z - e_p z}.
\end{proof}

\begin{cor}\label{c: | alpha - P_z sm alpha |_L^infty}
For $\mu\gg0$ and $\alpha\in E_z$, 
\[
\|\alpha-P_{z,\text{\rm sm}}\alpha\|_{L^\infty}\le e^{-c\mu}\|\alpha\|\;.
\]
\end{cor}

\begin{proof}
Apply \Cref{p: perturbed Sobolev,p: | alpha - P_z sm alpha |_m i nu}.

Alternatively, the proof of \Cref{p: | alpha - P_z sm alpha |_m i nu} can be modified as follows to get this result (some step of this alternative argument will be used later). Iterating~\eqref{|(w - D_z)^-1 beta |_m i nu}, we get
\[
\big\|(w-D_z)^{-1}\beta\big\|_{m,i\nu}\le C'_m\mu^{(m+1)m/2}\|\beta\|\;,
\]
for all $\beta\in L^2(M;\Lambda)$. Then, by \Cref{p: perturbed Sobolev},
\begin{equation}\label{| (w - D_z)^-1 beta |_L^infty}
\big\|(w-D_z)^{-1}\beta\big\|_{L^\infty}\le C\mu^{(m+1)m/2}\|\beta\|\;.
\end{equation}
Thus, by \Cref{p: | D_z^l e_p z |_m i nu},
\[
\big\|(w-D_z)^{-1}D_z e_{p,z}\big\|_{L^\infty}=O\big(e^{-c_m\mu}\big)
\]
as $\mu\to+\infty$. Finally, apply this expression in~\eqref{P_z sm e_p z - e_p z}.
\end{proof}

\begin{cor}\label{c: P_z sm : E_z -> E_z sm iso}
If $\mu\gg0$, then $P_{z,\text{\rm sm}}:E_z\to E_{z,\text{\rm sm}}$ is an isomorphism; in particular,
$\dim E_{z,\text{\rm sm}}=|\XX|$ and $\dim E_{z,\text{\rm sm}}^k=|\XX_k|$.
\end{cor}

\begin{proof}
This follows from \Cref{p: P_z D_z P_z = 0,p: | alpha - P_z sm alpha |_m i nu} for $m=0$ like \cite[Proposition~5.5]{Zhang2001}.
\end{proof}

When $\mu\gg0$,~\eqref{sum_k=0^n (-1)^k | XX_k | = chi(M)} also follows from \Cref{c: P_z sm : E_z -> E_z sm iso},~\eqref{sum_k (-1)^k beta_z^k = chi(M)} and~\eqref{H^bullet(Omega_sm z(M C) d_z) cong H_z^*(M C)}.

\begin{thm}[{Cf.\ \cite[Theorem~3]{BurgheleaHaller2001}}]\label{t: spec Delta_z}
We have
\[
\spec\Delta_z\subset\big[0,e^{-c|\mu|}\big]\cup\big[C|\mu|,\infty\big)\;.
\]
\end{thm}

\begin{proof} 
First, we establish the theorem for $|\mu|\gg0$, and then the constants will be changed to cover all $\mu$.

We can assume $\mu\ge0$ according to \Cref{r: | mu |}. By \Cref{p: P_z D_z P_z = 0,p: | alpha - P_z sm alpha |_m i nu,p: | alpha |_1 z}, for all $\alpha\in E_z$,
\begin{align*}
\|D_z P_{z,\text{\rm sm}}\alpha\|
&\le\|D_z\alpha\|+\|D_z(\alpha- P_{z,\text{\rm sm}}\alpha)\|\le\|D_z\alpha\|+\|\alpha- P_{z,\text{\rm sm}}\alpha\|_{1,z}\\
&\le\| P_z^\perp D_z\alpha\|+C(\mu\|\alpha- P_{z,\text{\rm sm}}\alpha\|+\|\alpha- P_{z,\text{\rm sm}}\alpha\|_{1,i\nu})\\
&\le\big(e^{-c\mu}+C\big(\mu e^{-c_0\mu}+e^{-c_1\mu}\big)\big)\|\alpha\|\;.
\end{align*}
Hence, by \Cref{c: P_z sm : E_z -> E_z sm iso}, for all $\beta\in E_{z,\text{\rm sm}}$,
\[
0\le\langle\Delta_z\beta,\beta\rangle=\|D_z\beta\|^2\le e^{-c\mu}\,\|\beta\|^2\;.
\]
This shows that
\begin{equation}\label{spec Delta_z cap [0 1]}
\spec\Delta_z\cap[0,1]\subset\big[0,e^{-c\mu}\big]\;.
\end{equation}

Now let $\phi\in E_{z,\text{\rm la}}\cap H^1(M;\Lambda)$, and write $\alpha= P_z\phi\in E_z$ and $\beta= P_z^\perp\phi\in E_z^\perp\cap H^1(M;\Lambda)$. By \Cref{p: | alpha - P_z sm alpha |_m i nu},
\[
\|\alpha\|^2=\langle\alpha,\phi\rangle=\langle\alpha- P_{z,\text{\rm sm}}\alpha,\phi\rangle\le\|\alpha- P_{z,\text{\rm sm}}\alpha\|\|\phi\|\le e^{-c_0\mu}\|\alpha\|\|\phi\|\;,
\]
yielding
\[
\|\alpha\|\le e^{-c_0\mu}\|\phi\|\;.
\]
So
\[
\|\beta\|=\|\phi-\alpha\|\ge\|\phi\|-\|\alpha\|\ge\big(1-e^{-c_0\mu}\big)\|\phi\|\;.
\]
Then, by \Cref{p: P_z D_z P_z = 0},
\begin{align*}
\|D_z\phi\|&\ge\|D_z\beta\|-\|D_z\alpha\|
\ge\| P_z^\perp D_z\beta\|-e^{-c\mu}\|\alpha\|\\
&\ge C\sqrt{\mu}\,\|\beta\|-e^{-c\mu}\|\phi\|
\ge\big(C\sqrt{\mu}\big(1-e^{-c_0\mu}\big)-e^{-c\mu}\big)\|\phi\|\;.
\end{align*}
Therefore, for all $\phi\in E_{z,\text{\rm la}}\cap H^1(M;\Lambda)$,
\[
\langle\Delta_z\phi,\phi\rangle=\|D_z\phi\|^2\ge C\mu\|\phi\|^2\;.
\]
This proves that
\begin{equation}\label{spec Delta_z cap (1 infty)}
\spec\Delta_z\cap(1,\infty)\subset[C\mu,\infty)\;.
\end{equation}

The inclusions~\eqref{spec Delta_z cap [0 1]} and~\eqref{spec Delta_z cap (1 infty)} give the result for $\mu\gg0$. But, in those inclusions, we can take $c$ and $C$ so small that, if one of them is not true for some $\mu\ge0$, then $C\mu\le e^{-c\mu}$. 
\end{proof}


\subsection{Ranks of some projections in the small complex}\label{ss: ranks}

Recall that $(\Pi_z^\perp)_{\text{\rm sm},k}$, $\Pi^1_{z,\text{\rm sm},k}$ and $\Pi^2_{z,\text{\rm sm},k}$ denote the orthogonal projections to the images of $\Delta_{z,\text{\rm sm},k}$,  $d_{z,\text{\rm sm},k-1}$ and  $\delta_{z,\text{\rm sm},k+1}$, respectively. Let $m_{z,k}$, $m_{z,k}^1$ and $m_{z,k}^2$ be the corresponding ranks (or traces) of these projections. They satisfy
\begin{equation}\label{m_z k^j}
m_{z,k}=m_{z,k}^1+m_{z,k}^2\;,\quad m_{z,0}^1=m_{z,n}^2=0\;,\quad m_{z,k}^2=m_{z,k+1}^1\;,
\end{equation}
where the last equality is true because $d_z:\im\delta_z\to\im d_z$ is an isomorphism. For $\mu\gg0$, we have $m_{z,k},m_{z,k}^j\le|\XX_k|$ by \Cref{c: P_z sm : E_z -> E_z sm iso} and~\eqref{m_z k^j}.

\begin{lem}\label{l: m_z k^j determined by m_z k}
The numbers $m_{z,k}^j$ are determined by the numbers $m_{z,k}$:
\[
m_{z,k+1}^1=m_{z,k}^2=\sum_{p=0}^k(-1)^{k-p}m_{z,p}
=\sum_{q=k+1}^n(-1)^{q-k-1}m_{z,q}\;.
\]
\end{lem}

\begin{proof}
This follows from~\eqref{m_z k^j} with an easy induction argument on $k$.
\end{proof}

\begin{lem}\label{l: m_z k = |XX_k| - beta_z^k}
For $\mu\gg0$, we have $m_{z,k}=|\XX_k|-\beta_z^k$.
\end{lem}

\begin{proof}
This is a consequence of~\eqref{Hodge dec Novikov},~\eqref{H^bullet(Omega_sm z(M C) d_z) cong H_z^*(M C)} and \Cref{c: P_z sm : E_z -> E_z sm iso}.
\end{proof}

\begin{cor}\label{c: Trs(Pi^perp_z sm) = 0}
$\Str((\Pi_z^\perp)_{\text{\rm sm}})=0$.
\end{cor}

\begin{proof}
By~\eqref{sum_k (-1)^k beta_z^k = chi(M)},~\eqref{sum_k=0^n (-1)^k | XX_k | = chi(M)} and \Cref{l: m_z k = |XX_k| - beta_z^k}, 
\[
\Str\big((\Pi_z^\perp)_{\text{\rm sm}}\big)=\sum_k(-1)^k|\XX_k|-\sum_k(-1)^k\beta_z^k=\chi(M)-\chi(M)=0\;.\qedhere
\]
\end{proof}

\begin{lem}\label{l: m_k = m_n-k}
If $M$ is oriented, then, for $k=0,\dots,n$,
\[
m_{z,k}=m_{-\bar z,n-k}=m_{-z,n-k}\;,\quad m_{z,k}^1=m_{-\bar z,n-k}^2=m_{-z,n-k}^2\;.
\]
\end{lem}

\begin{proof}
This is true because, by~\eqref{star z},
\begin{alignat*}{2}
(\Pi_z^\perp)_{\text{\rm sm},k}\,\star&=\star\,(\Pi_{-\bar z}^\perp)_{\text{\rm sm},n-k}\;,&\quad
\Pi^1_{z,\text{\rm sm},k}\,\star&=\star\,\Pi^2_{-\bar z,\text{\rm sm},n-k}\;,\\
(\Pi_z^\perp)_{\text{\rm sm},k}\,\bar\star&=\bar\star\,(\Pi_{-z}^\perp)_{\text{\rm sm},n-k}\;,&\quad
\Pi^1_{z,\text{\rm sm},k}\,\bar\star&=\bar\star\,\Pi^2_{-z,\text{\rm sm},n-k}\;.\;\qedhere
\end{alignat*}
\end{proof}

\begin{cor}\label{c: m^j_z k only depends on | XX_k | and xi}
For $\mu\gg0$, $m_{z,k}$ and $m^j_{z,k}$ only depend on $|\XX_k|$ and the class $\xi=[\eta]\in H^1(M,\R)$.
\end{cor}

\begin{proof}
Apply~\eqref{beta_z^k = betaNo^k} and \Cref{l: m_z k^j determined by m_z k,l: m_z k = |XX_k| - beta_z^k}.
\end{proof}

By \Cref{c: m^j_z k only depends on | XX_k | and xi}, we write $m_k=m_k(\eta)=m_{z,k}$ and $m_k^j=m_k^j(\eta)=m_{z,k}^j$ for $\mu\gg0$.

\begin{cor}\label{c: m_k = m_n-k}
If $M$ is oriented, then, for $k=0,\dots,n$,
\[
m_k(\eta)=m_{n-k}(-\eta)\;,\quad
m_k^1(\eta)=m_{n-k}^2(-\eta)=m_{n-k+1}^1(-\eta)\;.
\]
\end{cor}

\begin{proof}
Apply~\eqref{m_z k^j}, \Cref{l: m_k = m_n-k,c: m^j_z k only depends on | XX_k | and xi}. Alternatively, we can apply~\eqref{beta_z^k = betaNo^k},~\eqref{beta_z^k = beta_-z^n-k},~\eqref{m_z k^j}, \Cref{r: | mu |,l: m_z k = |XX_k| - beta_z^k}.
\end{proof}

\begin{cor}\label{c: -Trs(Pi^1_z sm) = ...}
For $\mu\gg0$,
\[
\Str(\Pi_{z,\text{\rm sm}}^1)=-\Str(\Pi_{z,\text{\rm sm}}^2)=\sum_{k=0}^n(-1)^kkm_k\;.
\]
If moreover $M$ is oriented and $n$ is even, then
\[
\sum_{k=0}^n(-1)^kkm_k=\sum_{k=0}^n(-1)^kk|\XX_k|-\frac n2\chi(M)\;.
\]
\end{cor}

\begin{proof}
\Cref{c: Trs(Pi^perp_z sm) = 0} gives the first equality. By \Cref{l: m_z k^j determined by m_z k,c: Trs(Pi^perp_z sm) = 0},
\begin{align*}
\Str(\Pi_{z,\text{\rm sm}}^1)=\sum_{k=0}^n(-1)^k\sum_{q=k}^n(-1)^{q-k}m_q
=\sum_{q=0}^n(-1)^q(q+1)m_q=\sum_{q=0}^n(-1)^qqm_q\;.
\end{align*}
Now assume $M$ is oriented and $n$ is even. Then, by~\eqref{sum_k (-1)^k beta_z^k = chi(M)},~\eqref{beta_z^k = betaNo^k} and~\eqref{beta_z^k = beta_-z^n-k},
\begin{align*}
\sum_{k=0}^n(-1)^kk\betaNo^k&=\sum_{l=0}^n(-1)^{n-l}(n-l)\betaNo^{n-l}
=\sum_{l=0}^n(-1)^l(n-l)\betaNo^l\\
&=n\chi(M)-\sum_{l=0}^n(-1)^ll\betaNo^l\;.
\end{align*}
Hence the last equality of the statement follows from \Cref{l: m_z k = |XX_k| - beta_z^k}.
\end{proof}

\subsection{Asymptotic properties of the small projection}\label{ss: asymptotic properties of the sm proj}

\begin{notation}\label{n: O(f(mu))}
Consider a function $f(x)>0$ ($x>0$). When referring to vectors in Banach spaces, the order notation $O(f(|\mu|))$ ($\mu\to\pm\infty$) will be used for a family of vectors $v=v(z)$ ($z\in\C$) with $\|v(z)\|=O(f(|\mu|))$. This notation applies e.g.\ to bounded operators between Banach spaces. We may also consider this notation when the Banach spaces depend on $z$.
\end{notation}

\begin{prop}\label{p: P_z sm P_z+tau sm}
For every $\tau\in\R$, on $L^2(M;\Lambda)$, as $\mu\to+\infty$,
\[
P_{z,\text{\rm sm}}=P_z+O\big(e^{-c\mu}\big)
=P_{z,\text{\rm sm}}P_{z+\tau,\text{\rm sm}}P_{z,\text{\rm sm}}+O\big(\mu^{-2}\big)
=P_{z+\tau,\text{\rm sm}}+O\big(\mu^{-1}\big)\;.
\]
\end{prop}

\begin{proof}
By \Cref{c: P_z sm : E_z -> E_z sm iso}, for $\mu\gg0$, the elements $P_{z,\text{\rm sm}} e_{p,z}$ ($p\in\XX$) form a base of $E_{z,\text{\rm sm}}$. Applying the Gram-Schmidt process to this base, we get an orthonormal base $\tilde  e_{p,z}$. By \Cref{p: | alpha - P_z sm alpha |_m i nu},
\begin{equation}\label{bfe'_p z = bfe_p z + O(e^-c mu)}
\tilde  e_{p,z}=e_{p,z}+O\big(e^{-c\mu}\big)\;.
\end{equation}
This gives the first equality of the statement: for any $\alpha\in L^2(M;\Lambda)$,
\[
P_z\alpha=\sum_{p\in\XX}\langle\alpha, e_{p,z}\rangle  e_{p,z}=\sum_{p\in\XX}\langle\alpha,\tilde  e_{p,z}\rangle\tilde  e_{p,z}+O\big(e^{-c\mu}\big)\|\alpha\|
=P_{z,\text{\rm sm}}\alpha+O\big(e^{-c\mu}\big)\|\alpha\|\;.
\]

Since the sets $U_p$ ($p\in\XX$) are disjoint one another, for $p\ne q$ in $\XX$,
\begin{equation}\label{langle bfe_p z bfe_q z+tau rangle}
\langle  e_{p,z},e_{q,z+\tau}\rangle=0\;.
\end{equation}
On the other hand, by~\eqref{rho_p}--\eqref{a_mu}, we can also assume
\begin{align}
\langle  e_{p,z}, e_{p,z+\tau}\rangle&=\langle e^{-i\nu h_p}e_{p,\mu},e^{-i\nu h_p} e_{p,\mu+\tau}\rangle
=\langle e_{p,\mu}, e_{p,\mu+\tau}\rangle\notag\\
&=\frac{(\mu(\mu+\tau))^{n/4}}{\pi^{n/2}}\big\langle\rho_pe^{-\mu|x_p|^2/2},\rho_pe^{-(\mu+\tau)|x_p|^2/2}\big\rangle
+O\big(e^{-c\mu}\big)\notag\\
&=\frac{(\mu(\mu+\tau))^{n/4}}{\pi^{n/2}}\int_{\R^n}e^{-(\mu+\tau/2)|x_p|^2}\,dx_p+O\big(e^{-c\mu}\big)\notag\\
&=\frac{(\mu(\mu+\tau))^{n/4}}{(\mu+\tau/2)^{n/2}}+O\big(e^{-c\mu}\big)=1+O\big(\mu^{-2}\big)\;,
\label{langle e_p z e_p z+tau rangle}
\end{align}
where $dx_p=dx_p^1\dots dx_p^n=\dvol(x_p)$. 
Combining~\eqref{bfe'_p z = bfe_p z + O(e^-c mu)} for $z$ and $z+\tau$ with~\eqref{langle bfe_p z bfe_q z+tau rangle} and~\eqref{langle e_p z e_p z+tau rangle}, we obtain
\begin{align}
P_{z+\tau,\text{\rm sm}}\tilde  e_{p,z}&=\sum_{q\in\XX}\langle \tilde  e_{p,z},\tilde e_{q,z+\tau}\rangle\tilde  e_{q,z+\tau}
=\sum_{q\in\XX}\langle  e_{p,z}, e_{q,z+\tau}\rangle  e_{q,z+\tau}+O\big(e^{-c\mu}\big)\notag\\
&= e_{p,z+\tau}+O\big(\mu^{-2}\big)=\tilde  e_{p,z+\tau}+O\big(\mu^{-2}\big)\;.
\label{P_z+tau sm bfe'_p z = bfe_p z + O(e^-c mu)}
\end{align}
Repeating~\eqref{P_z+tau sm bfe'_p z = bfe_p z + O(e^-c mu)} interchanging the roles of $z$ and $z+\tau$, we get
\[
P_{z,\text{\rm sm}}P_{z+\tau,\text{\rm sm}}\tilde  e_{p,z}=P_{z,\text{\rm sm}}\tilde  e_{p,z+\tau}+O\big(\mu^{-2}\big)=\tilde  e_{p,z}+O\big(\mu^{-2}\big)\;.
\]
This gives the second equality of the statement: for any $\alpha\in L^2(M;\Lambda)$,
\begin{align*}
P_{z,\text{\rm sm}}\alpha&=\sum_{p\in\XX}\langle\alpha,\tilde e_{p,z}\rangle\tilde e_{p,z}
=P_{z,\text{\rm sm}}P_{z+\tau,\text{\rm sm}}\sum_{p\in\XX}\langle\alpha,\tilde  e_{p,z}\rangle\tilde  e_{p,z}
+O\big(\mu^{-2}\big)\|\alpha\|\\
&=P_{z,\text{\rm sm}}P_{z+\tau,\text{\rm sm}}P_{z,\text{\rm sm}}\alpha+O\big(\mu^{-2}\big)\|\alpha\|\;.
\end{align*}

By~\eqref{P_z+tau sm bfe'_p z = bfe_p z + O(e^-c mu)},
\begin{multline*}
\|\tilde  e_{p,z}-\tilde e_{p,z+\tau}\big\|^2
=\|\tilde  e_{p,z}\|^2-2\Re\langle\tilde  e_{p,z},\tilde e_{p,z+\tau}\rangle+\|\tilde e_{p,z+\tau}\|^2\\
=2-2\Re\langle P_{z+\tau,\text{\rm sm}}\tilde  e_{p,z},\tilde e_{p,z+\tau}\rangle
=2-2\Re\langle\tilde  e_{p,z+\tau},\tilde e_{p,z+\tau}\rangle+O\big(\mu^{-2}\big)=O\big(\mu^{-2}\big)\;,
\end{multline*}
which means
\begin{equation}\label{bfe'_p z = bfe''_p z + O(1-f(mu))}
\tilde  e_{p,z}=\tilde e_{p,z+\tau}+O\big(\mu^{-1}\big)\;.
\end{equation}
The last stated equality follows from~\eqref{P_z+tau sm bfe'_p z = bfe_p z + O(e^-c mu)} and~\eqref{bfe'_p z = bfe''_p z + O(1-f(mu))}: for any $\alpha\in L^2(M;\Lambda)$,
\begin{align*}
P_{z,\text{\rm sm}}\alpha
&=\sum_{p\in\XX}\langle\alpha,\tilde  e_{p,z}\rangle\tilde  e_{p,z}
=\sum_{p\in\XX}\langle\alpha,\tilde  e_{p,z+\tau}\rangle\tilde  e_{p,z+\tau}+O\big(\mu^{-1}\big)\alpha\\
&=P_{z+\tau,\text{\rm sm}}\alpha+O\big(\mu^{-1}\big)\alpha\;.\qedhere
\end{align*}
\end{proof}

\begin{cor}\label{c: d_z+tau sm - d_z+tau P_z sm}
For every $\tau\in\R$, on $L^2(M;\Lambda)$,
\[
d_{z+\tau,\text{\rm sm}}-d_{z+\tau}P_{z,\text{\rm sm}}=O\big(\mu^{-1}\big)\quad(\mu\to+\infty)\;.
\]
\end{cor}

\begin{proof}
Since $d_{z+\tau}=d_z+\tau\,{\eta\wedge}$, it follows from \Cref{t: spec Delta_z} that $d_{z+\tau}$ is bounded on $E_{z,\text{\rm sm}}+E_{z+\tau,\text{\rm sm}}$, uniformly on $\mu\gg0$. Hence, by \Cref{p: P_z sm P_z+tau sm},
\[
d_{z+\tau,\text{\rm sm}}-d_{z+\tau}P_{z,\text{\rm sm}}=d_{z+\tau}(P_{z+\tau,\text{\rm sm}}-P_{z,\text{\rm sm}})=O\big(\mu^{-1}\big)\;.\qedhere
\]
\end{proof}

\begin{prop}\label{p: P_z sm theta wedge}
On $L^2(M;\Lambda)$,
\[
P_{z,\text{\rm sm}}\,{\eta\wedge},{\eta\wedge}\,P_{z,\text{\rm sm}}=O\big(\mu^{-1}\big)\quad(\mu\to+\infty)\;.
\]
\end{prop}

\begin{proof}
By \Cref{t: spec Delta_z}, for all $\alpha\in\Omega(M)$,
\[
\|d_zP_{z,\text{\rm sm}}\alpha\|^2=\langle\delta_zd_zP_{z,\text{\rm sm}}\alpha,P_{z,\text{\rm sm}}\alpha\rangle
\le\langle\Delta_zP_{z,\text{\rm sm}}\alpha,P_{z,\text{\rm sm}}\alpha\rangle
\le O\big(e^{-c\mu}\big)\;,
\]
yielding $d_zP_{z,\text{\rm sm}}=O\big(e^{-c\mu}\big)$. This is also true with the parameter $z+1$. So, by \Cref{c: d_z+tau sm - d_z+tau P_z sm},
\[
{\eta\wedge}\,P_{z,\text{\rm sm}}=(d_{z+1}-d_z)P_{z,\text{\rm sm}}
=d_{z+1}P_{z+1,\text{\rm sm}}-d_zP_{z,\text{\rm sm}}+O\big(\mu^{-1}\big)
=O\big(\mu^{-1}\big)\;.\qedhere
\]
\end{proof}

\subsection{Derivatives of the small projection}\label{ss: derivatives of the sm proj}

\begin{rem}\label{r: partial_bar z}
For reasons of brevity, most of the results about derivatives are stated for $\partial_z$, which may be simply denoted with a dot. But there are obvious versions of those results for $\partial_{\bar z}$ with analogous proofs.
\end{rem}

\begin{prop}\label{p: partial_z P_z sm = O(mu^-1)}
We have
\[
\rank\partial_zP_{z,\text{\rm sm}}\le2|\XX|\quad(\mu\gg0)\;,\quad\partial_zP_{z,\text{\rm sm}}=O\big(\mu^{-1}\big)\quad(\mu\to+\infty)\;.
\]
\end{prop}

\begin{proof}
By~\eqref{partial_z delta_z} and \Cref{t: spec Delta_z}, for $\mu\gg0$ and every $\omega\in\S^1$, a standard computation gives
\begin{equation}\label{partial_z((w-D_z)^-1)}
\partial_z\big((w-D_z)^{-1}\big)=(w-D_z)^{-1}\,{\eta\wedge}\,(w-D_z)^{-1}\;.
\end{equation}
Then, by~\eqref{| (w - D_z)^-1 beta |}, $\partial_z\big((w-D_z)^{-1}\big)$ defines an operator on $L^2(M;\Lambda)$, bounded uniformly on $w\in\S^1$ and $z\in\C$. By~\eqref{| (w - D_z)^-1 beta |} and \Cref{p: P_z sm theta wedge}, we also get
\begin{multline*}
P_{z,\text{\rm la/sm}}\partial_z\big((w-D_z)^{-1}\big)\,P_{z,\text{\rm sm/la}}\\
=(w-D_z)^{-1}P_{z,\text{\rm la/sm}}\,{\eta\wedge}\,P_{z,\text{\rm sm/la}}(w-D_z)^{-1}
=O\big(\mu^{-1}\big)\;,
\end{multline*}
uniformly on $w\in\S^1$.

On the other hand, applying again basic spectral theory, we obtain
\[
P_{z,\text{\rm sm}}=\frac{1}{2\pi i}\int_{\S^1}(w-D_z)^{-1}\,dw
\]
for $\mu\gg0$, yielding
\begin{equation}\label{dot P_z sm = ... -1}
\dot P_{z,\text{\rm sm}}=\frac{1}{2\pi i}\int_{\S^1}\partial_z\big((w-D_z)^{-1}\big)\,dw\;,
\end{equation}
which defines an operator on $L^2(M;\Lambda)$, bounded uniformly on $z$.

Using that $P_{z,\text{\rm sm}}$ is an orthogonal projection, the argument of the proof of \cite[Proposition~9.37]{BerlineGetzlerVergne2004} shows that
\begin{equation}\label{dot P_z sm = ...-2}
\dot P_{z,\text{\rm sm}}=P_{z,\text{\rm la}}\dot P_{z,\text{\rm sm}}P_{z,\text{\rm sm}}
+P_{z,\text{\rm sm}}\dot P_{z,\text{\rm sm}}P_{z,\text{\rm la}}\;.
\end{equation}
So $\rank\dot P_{z,\text{\rm sm}}\le2\rank P_{z,\text{\rm sm}}\le2|\XX|$ by \Cref{c: P_z sm : E_z -> E_z sm iso}, and
\begin{align*}
\dot P_{z,\text{\rm sm}}&=\frac{1}{2\pi i}\int_{\S^1}P_{z,\text{\rm la}}\partial_z\big((w-D_z)^{-1}\big)\,P_{z,\text{\rm sm}}\,dw\\
&\phantom{=\text{}}\text{}+\frac{1}{2\pi i}\int_{\S^1}P_{z,\text{\rm sm}}\partial_z\big((w-D_z)^{-1}\big)\,P_{z,\text{\rm la}}\,dw
=O\big(\mu^{-1}\big)\;.\;\qedhere
\end{align*}
\end{proof}

\begin{lem}\label{l: partial_z e_p z}
For all $p\in\XX$,
\[
\partial_ze_{p,z}=\bigg(\frac n{8\mu}-\frac{|x^+_p|^2}{2}+O(e^{-c\mu})\bigg)e_{p,z}\quad(\mu\to+\infty)\;.
\]
\end{lem}

\begin{proof}
Using integration by parts, and since $\rho$ is an even function and $\rho'$ vanishes on $[-r,r]$, we obtain
\begin{multline}\label{int_-2r^2r rho(x)^2 x^2 e^-mu x^2 dx}
\int_{-2r}^{2r}\rho(x)^2x^2e^{-\mu x^2}\,dx
=\frac1{2\mu}\int_{-2r}^{2r}(2\rho(x)\rho'(x)x+\rho(x)^2)e^{-\mu x^2}\,dx\\
=\frac1{2\mu}\Big(\frac\pi\mu\Big)^{\frac12}+O(e^{-c\mu})\;.
\end{multline}
So
\begin{align*}
\partial_\mu a_\mu&=\partial_\mu\bigg(\bigg(\int_{-2r}^{2r}\rho(x)^2e^{-\mu x^2}\,dx\bigg)^{\frac n2}\bigg)\\
&=-\frac n2\bigg(\int_{-2r}^{2r}\rho(x)^2e^{-\mu x^2}\,dx\bigg)^{\frac n2-1}
\int_{-2r}^{2r}\rho(x)^2x^2e^{-\mu x^2}\,dx\\
&=-\frac n2\Big(\frac\pi\mu\Big)^{\frac n4-\frac12}\frac1{2\mu}\Big(\frac\pi\mu\Big)^{\frac12}+O(e^{-c\mu})
=-\frac n{4\mu}\Big(\frac\pi\mu\Big)^{\frac n4}+O(e^{-c\mu})\;.
\end{align*}
Hence, by~\eqref{a_mu},
\begin{equation}\label{partial_mu(1/a_mu)}
\partial_\mu\Big(\frac1{a_\mu}\Big)
=-\frac{\partial_\mu a_\mu}{a_\mu^2}
=\frac n{4\mu}\Big(\frac\pi\mu\Big)^{\frac n4}
\Big(\frac\mu\pi\Big)^{\frac n2}+O(e^{-c\mu})
=\frac n{4\mu}\Big(\frac\mu\pi\Big)^{\frac n4}+O(e^{-c\mu})\;.
\end{equation}
It also follows from \Cref{p: Delta'_p mu}~\ref{i: e'_p mu},~\eqref{e_p mu},~\eqref{a_mu}  and~\eqref{partial_mu(1/a_mu)} that
\begin{multline}\label{partial_mu e_p mu}
\partial_\mu e_{p,\mu}
=\partial_\mu\Big(\frac{\rho_p}{a_\mu}e^{-\mu|x_p|^2/2}\,dx_p^1\wedge\dots\wedge dx_p^{\ind(p)}\Big)\\
=\bigg(\partial_\mu\Big(\frac1{a_\mu}\Big)a_\mu-\frac{|x_p|^2}2\bigg)e_{p,\mu}
=\bigg(\frac n{4\mu}-\frac{|x_p|^2}2+O(e^{-c\mu})\bigg)e_{p,\mu}\;.
\end{multline}
So, by~\eqref{e_p z},
\begin{equation}\label{partial_mu e_p z}
\partial_\mu e_{p,z}=\bigg(\frac n{4\mu}-\frac{|x_p|^2}2+O(e^{-c\mu})\bigg)e_{p,z}\;,\quad
\partial_\nu e_{p,z}=-ih_pe_{p,z}\;.
\end{equation}
Then the result follows using the right-hand side of~\eqref{h - h(p) around p}.
\end{proof}

\begin{prop}\label{p: | partial_z(D_z e_p z) |_L^infty}
For all $p\in\XX$,
\[
\|\partial_z(D_ze_{p,z})\|_{L^\infty}=O(e^{-c\mu})\quad(\mu\to+\infty)\;.
\]
\end{prop}

\begin{proof}
From~\eqref{D_z e_p,z}, we get
\begin{multline}\label{partial_z(D_ze_p z)}
\partial_z(D_ze_{p,z})
=\frac12\bigg(e^{-i\nu h_p}
\partial_\mu\Big(\frac1{a_\mu}\Big(\frac\pi\mu\Big)^{\frac n4}\Big)\hat c(d\rho_p)e_{p,\mu}\\
{}+e^{-i\nu h_p}\frac1{a_\mu}\Big(\frac\pi\mu\Big)^{\frac n4}\hat c(d\rho_p)\partial_\mu e_{p,\mu}
-h_pe^{-i\nu h_p}\frac1{a_\mu}\Big(\frac\pi\mu\Big)^{\frac n4}\hat c(d\rho_p)e_{p,\mu}\bigg)\;.
\end{multline}
By~\eqref{a_mu} and~\eqref{partial_mu(1/a_mu)},
\begin{multline}\label{partial_mu(1/a_mu (pi/mu)^n/4)}
\partial_\mu\Big(\frac1{a_\mu}\Big(\frac\pi\mu\Big)^{\frac n4}\Big)
=\partial_\mu\Big(\frac1{a_\mu}\Big)\Big(\frac\pi\mu\Big)^{\frac n4}
-\frac{n\pi}{4a_\mu\mu^2}\Big(\frac\pi\mu\Big)^{\frac n4-1}\\
=\frac n{4\mu}\Big(\frac\mu\pi\Big)^{\frac n4}\Big(\frac\pi\mu\Big)^{\frac n4}
-\frac{n\pi}{4\mu^2}\Big(\frac\pi\mu\Big)^{\frac n4-1}\Big(\frac\mu\pi\Big)^{\frac n4}
+O(e^{-c\mu})
=O(e^{-c\mu})\;.
\end{multline}
The result follows applying \Cref{p: Delta'_p mu}~\ref{i: e'_p mu},~\eqref{e_p mu},~\eqref{a_mu},~\eqref{partial_mu e_p mu} and~\eqref{partial_mu(1/a_mu (pi/mu)^n/4)} to~\eqref{partial_z(D_ze_p z)}, and using that $d\rho_p=0$ around $p$.
\end{proof}

\begin{prop}\label{p: | partial_z(P_z sm e_p z - e_p z) |_m i nu}
For every $p\in\XX$,
\[
\|\partial_z(P_{z,\text{\rm sm}}e_{p,z}-e_{p,z})\|_{L^\infty}=O(e^{-c\mu})\quad(\mu\to+\infty)\;.
\]
\end{prop}

\begin{proof}
By~\eqref{P_z sm e_p z - e_p z},
\begin{multline*}
\partial_z(P_{z,\text{\rm sm}}e_{p,z}-e_{p,z})
=\frac1{2\pi i}\int_{\S^1}w^{-1}\partial_z\big((w-D_z)^{-1}\big)D_ze_{p,z}\,dw\\
{}+\frac1{2\pi i}\int_{\S^1}w^{-1}(w-D_z)^{-1}\partial_z(D_ze_{p,z})\,dw\;.
\end{multline*}
Now apply~\eqref{| (w - D_z)^-1 beta |_L^infty},~\eqref{partial_z((w-D_z)^-1)}, \Cref{p: | D_ze_p z |_L^infty,p: | partial_z(D_z e_p z) |_L^infty}.
\end{proof}

\section{Small and large zeta invariants of Morse forms}

\subsection{Small and large zeta invariants}\label{ss: sm/la zeta invariants}

According to \Cref{ss: prelim zeta,ss: sm and la spec}, if $B$ is an operator in $L^2(M;\Lambda)$ so that $\zeta(s,\Delta_z,B)$ is defined, we have
\[
\zeta(s,\Delta_z,B)=\zetasm(s,\Delta_z,B)+\zetala(s,\Delta_z,B)\;,
\]
where
\[
\zeta_{\text{\rm sm/la}}(s,\Delta_z,B)=\zeta(s,\Delta_z,B_{z,\text{\rm sm/la}})\;.
\]
These are the contributions from the small/large spectrum to $\zeta(s,\Delta_z,B)$, which are called the \emph{small/large zeta functions} of $(\Delta_z,B)$. In particular, we can write
\[ 
\zeta(s,z)=\zetasm(s,z)+\zetala(s,z)\;,
\]
where $\zeta_{\text{\rm sm/la}}(s,z)=\zeta_{\text{\rm sm/la}}(s,z,\eta)$ is the small/large zeta function of $(\Delta_z,{\eta\wedge}\,D_z\sw)$. Since $\zetasm(s,z)$ is an entire function, $\zetala(s,z)$ has the same poles as $\zeta(s,z)$ (\Cref{r: regularity}), with the same residues. The value $\zeta_{\text{\rm sm/la}}(1,z)$ will be called the \emph{small/large zeta invariant} of $(M,g,\eta,z)$. The following results follow like \Cref{c: zeta(s z) = ...,c: zeta(1 z) = ...}.

\begin{cor}\label{c: zetala(s z) = ...}
If $\Re s>1/2$, then
\[
\zetala(s,z)=\frac{1}{\Gamma(s)}\int_0^\infty t^{s-1}\Str\big({\eta\wedge}\,D_ze^{-t\Delta_z}P_{z,\text{\rm la}}\big)\,dt\;,
\]
where the integral is absolutely convergent.
\end{cor}

\begin{cor}\label{c: zetala(1 z) = ...}
We have
\begin{align*}
\zetasm(1,z)&=\Str({\eta\wedge}\,D_z^{-1}(\Pi_z^\perp)_{\text{\rm sm}})\;,\\
\zetala(1,z)&=\lim_{t\downarrow0}\Str\big({\eta\wedge}\,D_z^{-1}e^{-t\Delta_z}P_{z,\text{\rm la}}\big)\;.
\end{align*}
\end{cor}

\subsection{Truncated heat invariants of perturbed operators}\label{ss: truncated heat invs of perturbed ops}

For $k=0,\dots,n$, let $K'_{z,k,t}(x,y)$ and $\widetilde K_{z,k,t}(x,y)$ denote the Schwartz kernels of $e^{-t\Delta_{z,k}}\Pi_z^\perp$ and $e^{-t\Delta_{z,k}}P_{z,\text{\rm la},k}$, respectively. According to \Cref{ss: heat}, their restrictions to the diagonal have asymptotic expansions (as $t\downarrow0$),
\begin{equation}\label{widetilde K_z k t(x x)}
K'_{z,k,t}(x,x)\sim\sum_{l=0}^\infty e'_{k,l}(x,z)t^{(l-n)/2}\;,\quad
\widetilde K_{z,k,t}(x,x)\sim\sum_{l=0}^\infty\tilde e_{k,l}(x,z)t^{(l-n)/2}\;.
\end{equation}
We have
\begin{align}
e'_{k,l}(x,z)&=
\begin{cases}
e_{k,l}(x,z) & \text{if $l<n$}\\
e_{k,n}(x,z)-\beta_z^k & \text{if $l=n$}\;,
\end{cases}\notag\\
\tilde e_{k,l}(x,z)&=
\begin{cases}
e_{k,l}(x,z) & \text{if $l<n$}\\
e_{k,n}(x,z)-H_{z,k,0}(x,x) & \text{if $l=n$}\;,
\end{cases}
\label{tilde e_k l(x z)}
\end{align}
where $H_{z,k,t}(x,y)$ is the Schwartz kernel of $e^{-t\Delta_{z,k}}P_{z,\text{\rm sm},k}$, which is defined for all $t\in\R$ and is smooth. We also have asymptotic expansions
\begin{align}
h'_k(t,z)&:=\Tr\big(e^{-t\Delta_{z,k}}\Pi_z^\perp\big)\sim\sum_{l=0}^\infty a'_{k,l}(z)t^{(l-n)/2}\;,
\label{asymptotic expansion of h'_k(t z)}\\
\tilde h_k(t,z)&:=\Tr\big(e^{-t\Delta_{z,k}}P_{z,\text{\rm la},k}\big)\sim\sum_{l=0}^\infty\tilde a_{k,l}(z)t^{(l-n)/2}\;.
\label{asymptotic expansion of tilde h_k(t z)}
\end{align}
By~\eqref{H_0(x y)},~\eqref{Tr(BP_A lambda)} and~\eqref{tilde a_l}, 
\begin{align}
a'_{k,l}(z)&=\int_M\str e'_{k,l}(x,z)\,\dvol(x) =
\begin{cases}
a_{k,l}(z) & \text{if $l<n$}\\
a_{k,l}(z)-\beta_z^k & \text{if $l=n$}\;.
\end{cases}
\label{a'_k l(z)}\\
\tilde a_{k,l}(z)&=\int_M\str\tilde e_{k,l}(x,z)\,\dvol(x)=
\begin{cases}
a_{k,l}(z) & \text{if $l<n$}\\
a_{k,l}(z)-\dim E_{z,\text{\rm sm}}^k & \text{if $l=n$}\;.
\end{cases}
\label{tilde a_k l(z)}
\end{align}

The operators $e^{-t\Delta_z}\Pi_z^\perp\sw$ and $e^{-t\Delta_z}P_{z,\text{\rm la}}\sw$ have Schwartz kernels 
\[
K'_{z,t}(x,y)=\sum_{k=0}^n(-1)^kK'_{z,k,t}(x,y)\;,\quad
\widetilde K_{z,t}(x,y)=\sum_{k=0}^n(-1)^k\widetilde K_{z,k,t}(x,y)\;,
\]
with induced asymptotic expansions
\[
K'_{z,t}(x,x)\sim\sum_{l=0}^\infty e'_l(x,z)t^{(l-n)/2}\;,\quad
\widetilde K_{z,t}(x,x)\sim\sum_{l=0}^\infty\tilde e_l(x,z)t^{(l-n)/2}\;,
\]
where
\begin{gather*}
e'_l(x,z)=\sum_{k=0}^n(-1)^ke'_{k,l}(x,z)\;,\quad
\tilde e_l(x,z)=\sum_{k=0}^n(-1)^k\tilde e_{k,l}(x,z)\;.
\end{gather*}
We also have induced asymptotic expansions,
\begin{gather*}
h'(t,z):=\Str\big(e^{-t\Delta_z}\Pi_z^\perp\big)\sim\sum_{l=0}^\infty a'_l(z)t^{(l-n)/2}\;,\\
\tilde h(t,z):=\Str\big(e^{-t\Delta_z}P_{z,\text{\rm la}}\big)\sim\sum_{l=0}^\infty\tilde a_l(z)t^{(l-n)/2}\;,
\end{gather*}
where
\[
a'_l(z)=\sum_{k=0}^n(-1)^ka'_{k,l}(z)\;,\quad
\tilde a_l(z)=\sum_{k=0}^n(-1)^k\tilde a_{k,l}(z)\;.
\]

If $\mu\gg0$, by~\eqref{beta_z^k = betaNo^k}, \Cref{c: P_z sm : E_z -> E_z sm iso,t: spec Delta_z}, $e'_{k,l}(x,z)$ and $\tilde e_{k,l}(x,z)$ depend smoothly on $z$ (\Cref{ss: heat}), and therefore so do $h'_k(t,z)$, $\tilde h_k(t,z)$,  $a'_{k,l}(z)$, $\tilde a_{k,l}(z)$, $e'_l(x,z)$, $\tilde e_l(x,z)$, $h'(t,z)$, $\tilde h(t,z)$, $a'_l(z)$ and $\tilde a_l(z)$.

\subsection{Truncated derived heat invariants of perturbed operators}\label{ss: truncated derived heat invs of perturbed ops}

For $k=0,\dots,n$ and $j=1,2$, let
\begin{gather*}
h^j_k(t,z)=\Tr\big(e^{-t\Delta_{z,k}}\Pi^j_{z,k}\big)\;,\quad
\tilde h^j_k(t,z)=\Tr\big(e^{-t\Delta_{z,k}}\Pi^j_{z,\text{\rm la},k}\big)\;.
\end{gather*}

\begin{lem}\label{l: h^j_k(t z)}
We have
\[
h_{k+1}^1(t,z)=h_k^2(t,z)=\sum_{p=0}^k(-1)^{k-p}h'_p(t,z)
=\sum_{q=k+1}^n(-1)^{q-k-1}h'_q(t,z)\;.
\]
\end{lem}

\begin{proof}
This follows by induction on $k$, using that 
\[
h_0^1(t,z)=h_n^2(t,z)=0\;,\quad
h'_k(t,z)=h_k^1(t,z)+h_k^2(t,z)\;,\quad
h_k^2(t,z)=h_{k+1}^1(t,z)\;.
\]
The last equality holds because~\eqref{CD im delta_z k+1 ...} is commutative.
\end{proof}

Let
\begin{gather*}
h^j(t,z)=\Str\big(e^{-t\Delta_z}\Pi^j_z\big)=\sum_{k=0}^n(-1)^kh^j_k(t,z)\;,\\
\tilde h^j(t,z)=\Str\big(e^{-t\Delta_z}\Pi^j_{z,\text{\rm la}}\big)=\sum_{k=0}^n(-1)^k\tilde h^j_k(t,z)\;.
\end{gather*}
Thus
\begin{equation}\label{h'(t z)}
h'(t,z)=h^1(t,z)+h^2(t,z)\;,\quad\tilde h(t,z)=\tilde h^1(t,z)+\tilde h^2(t,z)\;.
\end{equation}

\begin{cor}\label{c: h(t z) = 0}
We have $h'(t,z)=0$.
\end{cor}

\begin{proof}
This is a direct consequence of \Cref{l: h^j_k(t z)} and~\eqref{h'(t z)}.
\end{proof}

\begin{cor}\label{c: h^1(t z)}
We have
\[
h^1(t,z)=-h^2(t,z)=\sum_{k=0}^n(-1)^kkh'_k(t,z)=\Str\big(\sN e^{-t\Delta_z}\Pi_z^\perp\big)\;.
\]
\end{cor}

\begin{proof}
\Cref{c: h(t z) = 0} and~\eqref{h'(t z)} give the first equality. By \Cref{l: h^j_k(t z),c: h(t z) = 0},
\begin{align*}
h^1(t,z)&=\sum_{k=0}^n(-1)^k\sum_{q=k}^n(-1)^{q-k}h'_q(t,z)=\sum_{q=0}^n(-1)^q(q+1)h'_q(t,z)\\
&=h'(t,z)+\sum_{q=0}^n(-1)^qqh'_q(t,z)=\sum_{q=0}^n(-1)^qqh'_q(t,z)\;.\;\qedhere
\end{align*}
\end{proof}

\begin{rem}\label{r: similarity between -Trs(Pi^1_z sm) = ... and h^1(t z)}
Note the similarity between \Cref{c: -Trs(Pi^1_z sm) = ...,c: h^1(t z)}.
\end{rem}

Applying~\eqref{asymptotic expansion of h'_k(t z)} and \Cref{l: h^j_k(t z)}, we get
\begin{equation}\label{h^1(t z) sim ...}
h^j_k(t,z)\sim\sum_{l=0}^\infty a_{k,l}^j(z)t^{(l-n)/2}\;,\quad
h^j(t,z)\sim\sum_{l=0}^\infty a^j_l(z)t^{(l-n)/2}\;,
\end{equation}
where
\begin{gather*}
a^1_{k+1,l}(z)=a^2_{k,l}(z)=\sum_{p=0}^k(-1)^{k-p}a'_{p,l}(t,z)
=\sum_{q=k+1}^n(-1)^{q-k-1}a'_{q,l}(t,z)\;,\\
a^1_l(z)=-a^2_l(z)=\sum_{k=0}^n(-1)^kka'_{k,l}(z)\;.
\end{gather*}

\Cref{l: h^j_k(t z)}, \Cref{c: h(t z) = 0} and~\eqref{h^1(t z) sim ...} have obvious versions for $\tilde h^j_k(t,z)$ and $\tilde h^j(t,z)$, with similar proofs. The coefficients of the corresponding asymptotic expansions are denoted by $\tilde a^j_{k,l}(z)$ and $\tilde a^j_l(z)$.

\begin{cor}\label{c: tilde a^1_l(z) is independent of z}
For all $l\le n$ and $\mu\gg0$, $a^1_l(z)$ and $\tilde a^1_l(z)$ are independent of $z$.
\end{cor}

\begin{proof}
Apply~\eqref{beta_z^k = betaNo^k},~\eqref{a'_k l(z)},~\eqref{tilde a_k l(z)}, \Cref{c: P_z sm : E_z -> E_z sm iso,t: spec Delta_z,t: fa_l(z) is independent of z}.
\end{proof}

\subsection{Zeta function vs theta function}\label{ss: zeta vs theta}

Consider also the meromorphic function
\begin{equation}\label{theta(s,z)}
\theta(s,z)=\theta(s,z,\eta)=-\zeta(s,\Delta_z,\sN\sw)\;,
\end{equation}
called theta function of $\Delta_z$, and write
\[
\theta(s,z)=\thetasm(s,z)+\thetala(s,z)\;,\\
\]
where
\begin{equation}\label{theta_sm/la(s,z)}
\theta_{\text{\rm sm/la}}(s,z)=\theta_{\text{\rm sm/la}}(s,z,\eta)=-\zeta_{\text{\rm sm/la}}(s,\Delta_z,\sN\sw)\;.
\end{equation}
By \Cref{c: h^1(t z)},
\begin{gather}
-\zeta(s,\Delta_z,\Pi_z^1\sw)=\zeta(s,\Delta_z,\Pi_z^2\sw)=\theta(s,z)\;,\notag\\
-\zeta_{\text{\rm sm/la}}(s,\Delta_z,\Pi_z^1\sw)=\zeta_{\text{\rm sm/la}}(s,\Delta_z,\Pi_z^2\sw)=\theta_{\text{\rm sm/la}}(s,z)\;.
\label{-zeta_sm/la(s Delta_z Pi_z^1 sw) = theta_sm/la(s z)}
\end{gather}


Recall that $\zeta(s,z)$ is smooth at $s=1$ (\Cref{c: zeta(s z) = ...}). Moreover $\theta(s,z)$ is smooth at $s=0$ \cite{Seeley1967}. The same is true for $\zetala(s,z)$ and $\thetala(s,z)$.

\begin{prop}\label{p: partial_z thetala(s z)}
If $\mu\gg0$, then
\[
\partial_z\thetala(s,z)=s\zetala(s+1,z)\;.
\]
\end{prop}

\begin{proof}
Recall that a dot may be used to denote $\partial_z$. Like in~\eqref{dot P_z sm = ...-2},
\[
\dot\Pi^1_z=\big(\Pi^1_z\big)^\perp\dot\Pi^1_z\Pi^1_z+\Pi^1_z\dot\Pi^1_z\big(\Pi^1_z\big)^\perp\;.
\]
Therefore, since $\Pi^1_z$ and $(\Pi^1_z)^\perp$ commute with $\Delta_z^{-s}$ and $P_{z,\text{\rm la}}$, for $\Re s\gg0$,
\[
\zetala(s,\Delta_z,\dot\Pi^1_z\sw)=\Str\big(\dot\Pi^1_z\Delta_z^{-s}P_{z,\text{\rm la}}\big)=0\;,
\]
yielding $\zetala(s,\Delta_z,\dot\Pi^1_z\sw)=0$ for all $s$ because this is a meromorphic function. Hence, since $\Delta_z$ and $\Pi^1_{z,\text{\rm la}}\sw$ commute, \Cref{p: zeta functions}~\ref{i: zeta(s P Q)},\ref{i: partial/partial epsilon zeta(s P_epsilon Q_epsilon)} gives
\begin{equation}\label{partial_z zetala(s Delta_z Pi^1_z bfw)}
\partial_z\zetala(s,\Delta_z,\Pi^1_z\sw)=-s\zetala(s+1,\Delta_z,\dot\Delta_z\Pi^1_z\sw)
=-s\Str\big(\dot\Delta_z\Delta_z^{-s-1}\Pi^1_{z,\text{\rm la}}\big)\;.
\end{equation}
Next, by~\eqref{partial_z delta_z},
\begin{equation}\label{dot Delta_z Pi^1_z la}
\dot\Delta_z\Pi^1_{z,\text{\rm la}}
=({\eta\wedge}\,\delta_z+\delta_z\,{\eta\wedge})\Pi^1_{z,\text{\rm la}}
={\eta\wedge}\,\delta_z\Pi^1_{z,\text{\rm la}}+\delta_z\,{\eta\wedge}\,\Pi^1_{z,\text{\rm la}}\;.
\end{equation}
But, since $\Pi^1_z\delta_z=0$,
\begin{equation}\label{Str(delta_z theta wedge delta_z^-1 d_z^-1 Pi^1_z Delta_z^-s Pi^perp_z)}
\Str\big(\delta_z\,{\eta\wedge}\,\Delta_z^{-s-1}\Pi^1_{z,\text{\rm la}}\big)
=-\Str\big({\eta\wedge}\,\Delta_z^{-s-1}\Pi^1_{z,\text{\rm la}}\delta_z\big)=0\;.
\end{equation}
From~\eqref{-zeta_sm/la(s Delta_z Pi_z^1 sw) = theta_sm/la(s z)}--\eqref{Str(delta_z theta wedge delta_z^-1 d_z^-1 Pi^1_z Delta_z^-s Pi^perp_z)} and \Cref{p: zeta functions}~\ref{i: zeta(s P Q)}, we get
\begin{align*}
\partial_z\thetala(s,z)&=-\partial_z\zetala(s,\Delta_z,\Pi^1_z\sw)
=s\Str\big({\eta\wedge}\,\delta_z\Delta_z^{-s-1}\Pi^1_{z,\text{\rm la}}\big)\\
&=s\Str\big({\eta\wedge}\,D_z\Delta_z^{-s-1}\Pi^1_{z,\text{\rm la}}\big)
=s\zetala(s+1,z)\;.\qedhere
\end{align*}
\end{proof}

\begin{rem}\label{r: regularity - 2}
In the case where $\eta$ is a Morse form and $\mu\gg0$, the regularity of $\zeta(s,z)$ indicated in \Cref{r: regularity} also follows from \Cref{c: tilde a^1_l(z) is independent of z,p: partial_z thetala(s z)}.
\end{rem}

\begin{cor}\label{c: zetala(1 z) = partial_z thetala'(0 z)}
If $\mu\gg0$, then~\eqref{zetala(1 z) = partial_z thetala'(0 z)} is true.
\end{cor}

\begin{proof}
Apply \Cref{p: partial_z thetala(s z),c: zetala(s z) = ...}.
\end{proof}

\subsection{The case of the differential of a Morse function}\label{ss: differential of a Morse function}

Let us consider the special case where $\eta=dh$ for a Morse function $h$. The following four results follow like \Cref{l: Trs(eta wedge d_z^-1 e^-t Delta_z Pi^1_z) = -Trs(h e^-t Delta_z Pi^perp_z),c: zeta(1 z) = - lim_t->0 Str(h e^-t Delta_z Pi_z^perp),c: zeta(1 Delta_z theta wedge D_z bfw) in R,c: zeta(1 Delta_z theta wedge D_z bfw) = zeta(1 Delta_-bar z theta wedge D_- bar z bfw)}.

\begin{lem}\label{l: Trs(eta wedge d_z^-1 Pi^1_z sm) = -Trs(h Pi^perp_z sm)}
For $\mu\gg0$,
\begin{align*}
\Str\big({\eta\wedge}\,d_z^{-1}\Pi^1_{z,\text{\rm sm}}\big)&=-\Str\big(h\,(\Pi_z^\perp)_{\text{\rm sm}}\big)\;,\\
\Str\big({\eta\wedge}\,d_z^{-1}e^{-t\Delta_z}\Pi^1_{z,\text{\rm la}}\big)&=-\Str\big(h\,e^{-t\Delta_z}P_{z,\text{\rm la}}\big)\;.
\end{align*}
\end{lem}

\begin{cor}\label{c: zetasm(1 Delta_z eta wedge D_z bfw) = - Str (h (Pi_z^perp)_sm)}
For $\mu\gg0$, 
\begin{align*}
\zetasm(1,z)&=-\Str\big(h\,(\Pi_z^\perp)_{\text{\rm sm}}\big)\;,\\
\zetala(1,z)&=-\lim_{t\downarrow0}\Str\big(h\,e^{-t\Delta_z}P_{z,\text{\rm la}}\big)\;.
\end{align*}
\end{cor}

\begin{cor}\label{c: zeta_sm/la(1 Delta_z theta wedge D_z bfw) in R}
If $\mu\gg0$, then $\zeta_{\text{\rm sm/la}}(1,z)\in\R$.
\end{cor}

\begin{cor}\label{c: zeta_sm/la(1 Delta_z theta wedge D_z bfw) = zeta_sm/la(1 Delta_-bar z theta wedge D_- bar z bfw)}
If $M$ is oriented and $|\mu|\gg0$, then
\[
\zeta_{\text{\rm sm/la}}(1,z)=\zeta_{\text{\rm sm/la}}(1,-\bar z)
=\zeta_{\text{\rm sm/la}}(1,-z)=\zeta_{\text{\rm sm/la}}(1,\bar z)\;.
\]
\end{cor}

\begin{cor}\label{c: zetasm(1 Delta_z theta wedge D_z bfw) is uniformly bd}
The value $\zetasm(1,z)$ is uniformly bounded on $z$ for $\mu\gg0$.
\end{cor}

\begin{proof}
The operator $h\,(\Pi_z^\perp)_{\text{\rm sm}}$ is uniformly bounded and, for $\mu\gg0$, has uniformly bounded rank. So $\Str(h\,(\Pi_z^\perp)_{\text{\rm sm}})$ is uniformly bounded on $z$ for $\mu\gg0$, and therefore the result follows from \Cref{c: zetasm(1 Delta_z eta wedge D_z bfw) = - Str (h (Pi_z^perp)_sm)}.
\end{proof}

\begin{thm}\label{t: zeta_la(1 z) approx ... mu -> infty exact form}
The following limit holds uniformly on $\nu$: 
\[
\lim_{\mu\to+\infty}\zeta_{\text{\rm la}}(1,z)=-\int_Mh\,e(M,\nabla^M)\,\dvol
+\sum_{p\in\XX}^n(-1)^{\ind(p)}h(p)\;.
\]
\end{thm}

\begin{proof}
By~\eqref{widetilde K_z k t(x x)},~\eqref{tilde e_k l(x z)}, \Cref{t: e_l(x z),c: zetasm(1 Delta_z eta wedge D_z bfw) = - Str (h (Pi_z^perp)_sm)}, for $\mu\gg0$,
\begin{align*}
\zetala(1,z)&=-\lim_{t\downarrow0}\Str\big(h\,e^{-t\Delta_z}P_{z,\text{\rm la}}\big)
=-\int_Mh(x)\,\str\tilde e_n(x,z)\,\dvol(x)\\
&=-\int_Mh(x)\,\str e_n(x,z)\,\dvol(x)+\Str(hP_{z,\text{\rm sm}})\\
&=-\int_Mh\,e(M,\nabla^M)\,\dvol+\Str(hP_{z,\text{\rm sm}})\;.
\end{align*}

According to \Cref{c: P_z sm : E_z -> E_z sm iso}, the elements $P_{z,\text{\rm sm}} e_{p,z}$ ($p\in\XX$) form a base of $E_{z,\text{\rm sm}}^k$ when $\mu\gg0$. Applying the Gram-Schmidt process to this base, we get an orthonormal frame $\tilde  e_{p,z}$ ($p\in\XX$) of $E_{z,\text{\rm sm}}$. By \Cref{p: | alpha - P_z sm alpha |_m i nu} for $m=0$ and~\eqref{rho_p}--\eqref{a_mu},
\[
\lim_{\mu\to+\infty}\langle h\,\tilde  e_{p,z},\tilde e_{q,z}\rangle=\lim_{\mu\to+\infty}\langle h e_{p,z},e_{q,z}\rangle
=h(p)\delta_{pq}\;.
\]
Hence
\[
\lim_{\mu\to+\infty}\Str(hP_{z,\text{\rm sm}})=\sum_{k=0}^n(-1)^k\sum_{p\in\XX_k}h(p)\;.\qedhere
\]
\end{proof}

\section{The small complex vs the Morse complex}

\subsection{Preliminaries on Morse and Smale vector fields}\label{ss: Smale}

\subsubsection{Vector fields with Morse-type zeros}\label{sss: Morse-type zeros} 

Let $X$ be a real smooth vector field on $M$ with flow $\phi=\{\phi^t\}$.  Let $\YY=\Zero(X)$ denote the set of zeros of $X$ (or rest points $\phi$). It is said that a zero $p$ of $X$ is of \emph{Morse type} with (\emph{Morse}) \emph{index} of $\ind(p)$ if, using the notation~\eqref{epsilon_p,j},
\begin{equation}\label{X around p}
X=-\sum_{j=1}^n\epsilon_{p,j}x_p^j\,\frac{\partial}{\partial x_p^j}
\end{equation}
on the domain $U_p$ of some coordinates $x_p=(x_p^1,\dots,x_p^n)$ at $p$, also called \emph{Morse coordinates}. This condition means that $X=-\grad_gh_{X,p}$ on $U_p$, where $h_{X,p}$ and $g$ are given on $U_p$ by the center and right-hand side of~\eqref{h - h(p) around p} and~\eqref{g around p}. The coordinates $x_p$ used in~\eqref{X around p} are not unique; that expression is invariant by taking positive multiples of the coordinates (contrary to the expressions~\eqref{h - h(p) around p},~\eqref{g around p} and~\eqref{eta around p}). But $\ind(p)$ is independent of $x_p$. Note that the Hopf index of $-X$ at $p$ is $(-1)^{\ind(p)}$. 

Let us consider $\eta\in Z^1(M,\mathbb R)$ and use the notation of \Cref{ss: Morse forms}. For $p\in\XX\cap\YY$, if~\eqref{g around p},~\eqref{eta around p} and~\eqref{X around p} hold with the same coordinates, then $\eta$ and $g$ are said to be in \emph{standard form} with respect to $X$ around $p$. In this case, $C\eta$ and $Cg$ ($C>0$) are also in standard form with respect to $X$ around $p$; indeed, $C\eta$, $X$ and $Cg$ satisfy~\eqref{g around p},~\eqref{eta around p} and~\eqref{X around p} with the coordinates $\sqrt{C}x_p$. If $\XX=\YY$, and $\eta$ and $g$ are in standard form with respect to $X$ around every $p\in\XX$, then $\eta$ and $g$ are said to be in \emph{standard form} with respect to $X$. This concept is also applied to any Morse function $h$ on $M$ referring to $dh$ and $g$. The reference to $g$ may be omitted in this terminology. 

Unless otherwise indicated, we assume from now on that $X$ has Morse-type zeros. Then $\YY$ is finite, and the sets $\YY_k$, $\YY_+$ and $\YY_{<k}$ are defined like in \Cref{ss: Morse forms}. 

\subsubsection{Stable/unstable manifolds}\label{sss: stable/unstable mds}

For $k=0,\dots,n$ and $p\in\YY_k$, the \emph{stable/unstable manifolds} of $p$ are smooth injective immersions, $\iota^\pm_p:W^\pm_p\to M$, where the images $\iota^\pm_p(W^\pm_p)$ consist of the points satisfying $\phi^t(x)\to p$ as $t\to\pm\infty$, and the manifolds $W^+_p$ and $W^-_p$ are diffeomorphic to $\R^{n-k}$ and $\R^k$, respectively \cite[Theorem~9.1]{Smale1963}. In particular, $p\in\iota^\pm_p(W^\pm_p)$, and the maps $\iota^+_p$ and $\iota^-_p$ meet transversely at $p$. Let $p^\pm=(\iota^\pm_p)^{-1}(p)$. Assume every $U_p$ is connected, and let $U^\pm_p$ be the connected component of $(\iota^\pm_p)^{-1}(U_p)$ that contains $p^\pm$. The restriction $\iota^\pm_p:U^\pm_p\to(x_p^\pm)^{-1}(0)$ is a diffeomorphism, and therefore $(U^\pm_p,x^\pm_p\iota^\pm_p)$ is a chart of $W^\pm_p$ at $p^\pm$.

\subsubsection{Gradient-like vector fields}\label{sss: gradient-like}

Given a Morse function $h$ on $M$ in standard form with respect to $X$, we have $X=-\grad_gh$ on $M$ for some Riemannian metric $g$ if and only if $Xh<0$ on $M\setminus\YY$ \cite[Lemma~2.1]{BurgheleaFriedlanderKappeler2010}, \cite[Section~6.1.3]{Laudenbach2012}; in this case, $X$ is said to be \emph{gradient-like} (with respect to $h$). If $X$ is gradient-like, then the maps $\iota^\pm_p$ are embeddings \cite[Lemma~3.8]{Smale1960a}, \cite[Lemma 2.2]{BurgheleaFriedlanderKappeler2010}, and their images cover $M$ \cite[Theorem~B and Lemma~1.1]{Smale1961}, \cite[Corollary~2.5]{BurgheleaFriedlanderKappeler2010}. Thus, in this case, the $\alpha$- and $\omega$-limits of the orbits of $X$ are zero points, we can write $W^\pm_p=\iota^\pm_p(W^\pm_p)$ and $p^\pm=p$, and $\iota^\pm_p$ becomes the inclusion map.

Unless otherwise indicated, we also assume in the rest of the paper that $X$ is gradient-like.

\subsubsection{Smale vector fields}\label{sss: integrals along instantons}

$X$ is said to be \emph{Smale} if $W^+_p\pitchfork W^-_q$ for all $p,q\in\YY$. Then $\MM(p,q):=W^+_p\cap W^-_q$ is a $\phi$-saturated smooth submanifold of dimension $\ind(p)-\ind(q)$. If $p=q$, we have $\MM(p,p)=\{p\}$; in this case, define $\TT(p,p)=\emptyset$. If $p\ne q$, the induced $\R$-action on $\MM(p,q)$ is free and proper; in this case, define $\TT(p,q)=\MM(p,q)/\R$, which is a smooth manifold of dimension $\ind(p)-\ind(q)-1$. The elements of $\TT(p,q)$ are the (unparameterized) trajectories with $\alpha$-limit $p$ and $\omega$-limit $q$, which are oriented by $X$. If $\ind(p)\le \ind(q)$, then $\TT(p,q)=\emptyset$. If $\ind(p)-\ind(q)=1$, then $\TT(p,q)$ consists of isolated points,
each of them representing a trajectory in $M$. Let
\[
\TT=\bigcup_{p,q\in\XX}\TT(p,q)\;,\quad\TT^1_p=\bigcup_{q\in\XX_{\ind(p)-1}}\TT(p,q)\;,\quad\TT^1_k=\bigcup_{p\in\XX_k}\TT^1_p\;,\quad\TT^1=\bigcup_{k=0}^n\TT^1_k\;.
\]
The elements of $\TT^1$ are called \emph{instantons}.\footnote{In \cite{Bott1988}, the elements of $\TT$ are called \emph{instantons}, and the elements of $\TT^1$ \emph{proper instantons}.}

$X$ can be $C^\infty$-approximated by gradient-like Smale vector fields that agree with $X$ around $\XX$ \cite[Proposition~2.4]{BurgheleaHaller2008} (this follows from \cite[Theorem~A]{Smale1961}). A well known consequence is that, for any Morse function $h$, there is a $C^\infty$-dense set of Riemannian metrics $g$ on $M$ such that $-\grad_gh$ is Smale; this density is also true in the subspace of metrics that are Euclidean with respect to Morse coordinates on given neighborhoods of the critical points.

Unless otherwise indicated, besides the above conditions, we assume from now on that $X$ is Smale; i.e., we assume~\ref{i-a:  X ...} (\Cref{ss: intro-Witten}).

\subsubsection{Lyapunov forms}\label{sss: Lyapunov}

Any $\eta\in Z^1(M,\R)$ is said to be \emph{Lyapunov} for $X$ if $\eta(X)<0$ on $M\setminus\YY$ \cite[Definition~2.3]{BurgheleaHaller2008}. Note that this condition implies that $\Zero(\eta)=\YY$. By~\ref{i-a:  X ...}, every class in $H^1(M,\R)$ has a representative $\eta$ which is Lyapunov for $X$ and $\eta^\sharp=-X$ for some Riemannian metric $g$ on $M$, with $\eta$ and $g$ in standard form with respect to $X$  \cite[Proposition~16~(i)]{BurgheleaHaller2004}, \cite[Observations~2.5 and~2.6]{BurgheleaHaller2008}, \cite[Lemma~3.7]{HarveyMinervini2006}, \cite[Section~6.1.3]{Laudenbach2012}.

\subsubsection{Completion of the unstable manifolds}\label{sss: completion}

\begin{prop}[{\cite[Appendix by F.~Laudenbach, Proposition~2]{BismutZhang1992}, \cite[Chapter~2]{Latour1994}, \cite[Theorem~2.1]{Burghelea1997}, \cite[Theorem~1]{BurgheleaHaller2001}, \cite[Theorem~4.4]{BurgheleaFriedlanderKappeler2010}, \cite[Sections~A.2 and~A.8]{Laudenbach2012}, \cite[Corollary~2.3.2]{Minervini2015}}]
\label{p: completion}
The following holds for every $p\in\YY_k\ (k=0,\dots,n)$:
\begin{enumerate}[{\rm(i)}]

\item\label{i: overline W^-_q cap overline W^-_p} $\overline{W^-_p}$ is a $C^1$ submanifold with conic singularities\footnote{In the sense of \cite[Appendix by F.~Laudenbach, Section~a)]{BismutZhang1992} and \cite[Appendix~A.1]{Laudenbach2012}.} and a Whitney stratified subspace\footnote{Introduced by H.~Whitney \cite{Whitney1965a,Whitney1965b}, and the definition was simplified by J.~Mather \cite{Mather1970}.}. Its strata are the submanifolds $W^-_q$ for $q\in\YY_{<k}$ with $\TT(p,q)\ne\emptyset$. As a consequence, $W^-_p$ has finite volume, and
\[
\overline{W^-_q}\cap\overline{W^-_p}\subset\bigcup_{x\in\YY_{<k}}W_x^-
\]
if $q\ne p$ in $\YY_k$; in particular, $p\notin\overline{W^-_q}$.

\item\label{i: partial_l widehat W^-_p}
 There is a compact $k$-manifold with corners\footnote{In the sense of \cite[Section~1.1.8]{Melrose1996}.} $\widehat W^-_p$ whose $l$-corner\footnote{The union of the interiors of the boundary faces of codimension $l$.} is
\[
\partial_l\widehat W^-_p=\bigsqcup_{(q_0,\dots,q_l)\in\{p\}\times\YY^l}\bigg(\prod_{j=1}^l\TT(q_{j-1},q_j)\bigg)\times W_{q_l}^-
\quad(0\le l\le k)\;.
\]
In particular, the interior of $\widehat W^-_p$ is $\partial_0\widehat W^-_p=W^-_p$, and the set $\TT(p,q)$ is finite if $q\in\YY_{k-1}$.

\item\label{i: hat iota^-_p} There is a smooth map $\hat\iota^-_p:\widehat W^-_p\to M$ whose restriction to every component of $\partial_l\widehat W^-_p$ is given by the factor projection to $W_{q_l}^-$, according to~\ref{i: partial_l widehat W^-_p}. In particular, $\hat\iota^-_p=\iota^-_p$ on $W^-_p$, and $\hat\iota_p^-:\widehat W^-_p\to\overline{W^-_p}$ is a stratified map.

\end{enumerate}
\end{prop}

By \Cref{p: completion}~\ref{i: overline W^-_q cap overline W^-_p}, we can choose the open sets $U_p$ ($p\in\YY_k$, $k=0,\dots,n$) so small that $U_p\cap\overline{W^-_q}=\emptyset$ if $q\ne p$ in $\YY_k$.

For every $q\in\YY_{k-1}$ and $\gamma\in\TT(p,q)$, the closure $\bar\gamma$ in $M$ is a compact oriented submanifold with boundary of dimension one, and $\partial\bar\gamma=\{p,q\}$. We may also consider $\bar\gamma$ as the closure of $\gamma$ in $\widehat W^-_p$.

\subsection{Preliminaries on the Morse complex}\label{ss: prelim Morse}

\subsubsection{The Morse complex when $M$ is oriented}\label{sss: Morse complex - M oriented}

For reasons of clarity, assume first that $M$ is oriented. Fix an orientation $\OO_p^-$ of every unstable manifold $W^-_p$ ($p\in\YY_k$, $k=0,\dots,n$), which can be also considered as an orientation of $\widehat W^-_p$. Then $W^-_p\equiv(W^-_p,\OO^-_p)$ defines a current of dimension $k$ on $M$, also denoted by $W^-_p$; namely, for $\alpha\in \Omega^k(M)$,
\begin{equation}\label{langle W^-_p alpha rangle}
\langle W^-_p,\alpha\rangle=\int_{W^-_p}\alpha=\int_{\widehat W^-_p}(\hat\iota_p^-)^*\alpha\;.
\end{equation}

Let  $\partial_1\OO_p^-$ be the orientation of $\partial_1\widehat W^-_p$ induced by $\OO_p^-$ like in the Stokes' theorem; precisely, it is determined by $\OO_p^-=\nu_p^-\otimes\partial_1\OO_p^-$ along $\partial_1\widehat W^-_p$ for any outward-pointing normal vector $\nu_p^-$. The restriction of $\partial_1\OO_p^-$ to every component $\TT(p,q)\times W^-_q$ ($q\in\YY_{k'}$) of $\partial_1\widehat W^-_p$ is of the form $\OO_{p,q}\otimes\OO_q^-$ for a unique orientation $\OO_{p,q}$ of $\TT(p,q)$. If $k'=k-1$, then $\OO_{p,q}$ can be represented by a unique function $\epsilon_{p,q}:\TT(p,q)\to\{\pm1\}$; combining these functions, we get a map $\epsilon:\TT^1\to\{\pm1\}$. By the descriptions of $\partial_1\widehat W^-_p$ and $\hat\iota^-_p:\partial_1\widehat W^-_p\to M$, and by the Stokes' theorem for manifolds with corners, we have \cite[Appendix by F.~Laudenbach]{BismutZhang1992}, \cite[Remark~1.9]{HarveyMinervini2006}, \cite[Theorem~3.6 and Proposition~5.3]{BurgheleaFriedlanderKappeler2010}, \cite[Section~6.5.3]{Laudenbach2012}
\begin{equation}\label{partial W^-_p}
\partial W^-_p=\sum_{q\in\YY_{k-1},\ \gamma\in\TT(p,q)}\epsilon(\gamma)\,W^-_q\;.
\end{equation}
Thus the currents $W^-_p$ ($p\in\XX$) generate over $\C$ a finite dimensional subcomplex $(C_\bullet(X,W^-),\partial)$ of the complex $(\Omega(M)',\partial)$ of currents on $M$, called the \emph{Morse complex}.
The simpler notation $\bfC_\bullet=\bfC_\bullet(X)=C_\bullet(X,W^-)$ may be also used. Moreover $\bfC_\bullet\hookrightarrow \Omega(M)'$ is a quasi-isomorphism,\footnote{Actually, $H_\bullet(M,\Z)$ is isomorphic to the homology of the complex of free Abelian groups generated by the currents $W^-_p$.} $H_\bullet(\bfC_\bullet,\partial)\cong H_\bullet(M,\C)$ \cite{Thom1949,Smale1960b,Milnor1965} (see also \cite{Floer1989,Schwarz1993,Schwarz1999}, \cite[Theorem~0.1]{HelfferSjostrand1985}, \cite[Appendix by F.~Laudenbach, Proposition~7]{BismutZhang1992},  \cite[Section~6.6.5]{Laudenbach2012}). 

Let $(C_\bullet(X,W^+),\partial)=(C_\bullet(-X,W^-),\partial)$, involving the stable Morse cells $W^+_p$. If $M$ is oriented by $\OO_M$ and the orientation $\OO_p^+$ of every $W^+_p$ is choosen so that $\OO_p^+\otimes\OO_p^-=\OO_M$ at $p$, then the canonical pairing
\begin{equation}\label{pairing between C_bullet(X W^-) and C_n-bullet(X W^+)}
\langle{\cdot},{\cdot}\rangle:C_\bullet(X,W^-)\times C_{n-\bullet}(X,W^+)\to\K\;,\quad\langle W^-_p,W^+_q\rangle=\delta_{pq}\;,
\end{equation}
satisfies \cite[Section~6.6.2]{Laudenbach2012}
\begin{equation}\label{langle partial W^-_p W^+_q rangle}
\langle\partial W^-_p,W^+_q\rangle=(-1)^k\,\langle W^-_p,\partial W^+_q\rangle\quad(p\in\XX_k,\ q\in\XX_{k-1})\;.
\end{equation}

\subsubsection{The Morse complex when $M$ may not be oriented}\label{sss: Morse complex - M non-oriented}

When $M$ is not assumed to be oriented, the concepts of \Cref{sss: Morse complex - M oriented} can be extended as follows. We fix an orientation $N\OO_p^-$ of every normal bundle $NW^-_p$, which can be also considered as an orientation of $N\widehat W^-_p$ (the normal bundle of the immersion $\hat\iota_p^-$). Then we can consider $W^-_p\equiv(W^-_p,N\OO^-_p)\in\Omega^k(M,o(M))'$, by using $N\OO_p^-\otimes\alpha$ as integrand in~\eqref{langle W^-_p alpha rangle} for every $\alpha\in \Omega^k(M,o(M))$; note that $N\OO_p^-\otimes\alpha\in\Omega^k(\widehat W^-_p,o(\widehat W^-_p))=\Omega^k(\widehat W^-_p)$. With the notation of \Cref{sss: Morse complex - M oriented}, $\partial_1N\OO_p^-:=N\OO_p^-\otimes\nu_p^-$ describes an orientation of $N\partial_1\widehat W^-_p$, and the Stokes theorem has the extension (see \cite[Theorem~7.7]{BottTu1982} for the case without boundary)
\begin{equation}\label{Stokes - non-orientable}
\int_{\widehat W^-_p}N\OO_p^-\otimes d\beta=\int_{\partial_1\widehat W^-_p}\partial_1N\OO_p^-\otimes\beta\quad
\big(\beta\in\Omega^{k-1}(M,o(M))\big)\;.
\end{equation}

If $M$ is oriented by $O_M$, then $N\OO_p^-$ and $\OO_p^-$ determine each other by the condition $O_M=N\OO_p^-\otimes\OO_p^-$. Then $\partial_1N\OO_p^-$ and $\partial_1\OO_p^-$ determine each other in the same way:
\[
O_M=N\OO_p^-\otimes\OO_p^-=N\OO_p^-\otimes\nu^-_p\otimes\partial_1\OO_p^-
=\partial_1N\OO_p^-\otimes\partial_1\OO_p^-\;.
\]
So~\eqref{Stokes - non-orientable} agrees with the usual Stokes' theorem in this way. 

If $M$ is not oriented, by using local orientations of $M$, the above argument shows that~\eqref{Stokes - non-orientable} also agrees with the usual Stokes' theorem for $o(M)$-valued forms $\beta$ with small enough support. Then, like in \Cref{sss: Morse complex - M oriented}, we get the same map $\epsilon:\TT^1\to\{\pm1\}$, and therefore the same definition of $(\bfC_\bullet,\partial)$.

\subsubsection{The dual Morse complex}\label{sss: dual Morse complex}

Let $C^k(X,W^-)=(\bfC_k)^*\equiv\C^{\YY_k}$ ($k=0,\dots,n$) and $\bfd=\partial^*$. The simpler notation $\bfC^\bullet=\bfC^\bullet(X)$ will be preferred. It is said that $(\bfC^\bullet,\bfd)$ is the \emph{dual Morse complex}. Boldface notation is also used for elements of $\bfC^\bullet$ and other operators on $\bfC^\bullet$. Let $\bfe_p$ ($p\in\YY$) denote the elements of the canonical base of $\bfC^\bullet$, determined by $\bfe_p(q)=\delta_{pq}$. By~\eqref{partial W^-_p}, for $q\in\YY_{k-1}$,
\begin{equation}\label{bfd bfe_q}
\bfd\bfe_q=\sum_{p\in\YY_k,\ \gamma\in\TT(p,q)}\epsilon(\gamma)\,\bfe_p\;.
\end{equation}
Comparing~\eqref{partial W^-_p} and~\eqref{bfd bfe_q}, we see that $(C^\bullet(X,W^-),\bfd)\equiv(C_\bullet(-X,W^+),\partial)$. Thus, from now on, $(\bfC^\bullet,\bfd)$ will be also called a \emph{Morse complex}. If $M$ is oriented, it also follows from~\eqref{pairing between C_bullet(X W^-) and C_n-bullet(X W^+)} and~\eqref{langle partial W^-_p W^+_q rangle} that $(C^\bullet(X,W^-),\sw\bfd)\equiv(C_{n-\bullet}(X,W^+),\partial)$.

\subsubsection{The perturbed Morse complex}\label{sss: perturbed Morse complex}

Take any $\eta\in Z^1(M,\R)$ defining a class $\xi\in H^1(M,\R)$ (there is no need of any condition on $\eta$ or $g$ in \Cref{sss: perturbed Morse complex,sss: Morse complex with F,sss: bfDelta_z}). For reasons of brevity, write $\eta(\gamma)=\int_\gamma\eta$ for every $\gamma\in\TT^1$. According to \cite{BurgheleaHaller2001,BurgheleaHaller2004,BurgheleaHaller2008}, $(\bfC^\bullet,\bfd)$ has an analog of the Witten's perturbation, $(\bfC^\bullet,\bfd_z=\bfd_{z\eta})$ ($z\in\C$), where, for $q\in\YY_{k-1}$ ($k=1,\dots,n$),
\begin{equation}\label{bfd_z bfe_q}
\bfd_z\bfe_q=\sum_{p\in\YY_k,\ \gamma\in\TT(p,q)}\epsilon(\gamma)e^{z\eta(\gamma)}\bfe_p\;.
\end{equation}
If $\eta=dh$ for some $h\in C^\infty(M,\R)$, then $\bfd_z=e^{-zh}\bfd e^{zh}$ on $\bfC^\bullet$ because $\eta(\gamma)=h(q)-h(p)$ for $p\in\YY_k$, $q\in\YY_{k-1}$ and $\gamma\in\TT(p,q)$; here, $e^{\pm zh}$ also denotes the operator of multiplication by the restriction of this function to $\YY$. It will be said that $(\bfC^\bullet,\bfd_z)$ ($z\in\C$) is the \emph{perturbed dual Morse complex} defined by $X$ and $\eta$. A perturbation $(\bfC_\bullet,\partial^z)$ is similarly defined, multiplying by $e^{z\eta(\gamma)}$ the terms of the right-hand side of~\eqref{partial W^-_p}.

Since $W^-_p$ ($p\in\YY_k$, $k=0,\dots,n$) is diffeomorphic to $\R^k$, there is a unique $h_{\eta,p}^-\in C^\infty(W^-_p,\R)$ such that $h_{\eta,p}^-(\hat p^-)=0$ and $dh_{\eta,p}^-=(\iota_p^-)^*\eta$, where $\hat p^-\in W_p^-\subset\widehat W_p^-$ is determined by $\iota_p^-(\hat p^-)=p$. Indeed $h_{\eta,p}^-$ has a smooth extension $\hat h_{\eta,p}^-$ to $\widehat W^-_p$ because $\widehat W^-_p$ is contractile. By \Cref{p: completion}~\ref{i: partial_l widehat W^-_p}, for all $q\in\YY_{k-1}$ and $\gamma\in\TT(p,q)$, we have $\hat h_{\eta,p}^-(\gamma,\hat q^-)=\eta(\gamma)$ at $(\gamma,\hat q^-)\in\{\gamma\}\times\widehat W^-_q\subset\partial_1\widehat W^-_p$. Therefore $\hat h_{\eta,q}^-$ corresponds to the restriction of $\hat h_{\eta,p}^--\eta(\gamma)$ via the canonical diffeomorphism $\widehat W^-_q\approx\{\gamma\}\times\widehat W^-_q$.

According to \cite[Proposition~4]{BurgheleaHaller2001}, \cite[Proposition~10]{BurgheleaHaller2004}, \cite[Propositions~2.15 and~2.16 and Section~6.2]{BurgheleaHaller2008}, a sujective homomorphism of complexes,
\[
\Phi_z:(\Omega(M),d_z)\to(\bfC^\bullet,\bfd_z)\;,
\]
is defined by 
\[
\Phi_z(\omega)(p)=\int_{W^-_p}e^{zh_{\eta,p}^-}\omega=\int_{\widehat W^-_p}e^{z\hat h_{\eta,p}^-}(\hat\iota_p^-)^*\omega\;.
\]
Moreover $\Phi_z$ is a quasi-isomorphism for all $z\in\C$ \cite[Proposition~7 in the Appendix by F.~Laudenbach]{BismutZhang1992} (see also \cite[Theorem~2.9]{BismutZhang1992}, \cite[Theorem~1.6]{BismutZhang1994}, \cite[Proposition~2.17 and Section~6.2]{BurgheleaHaller2008}). If $\eta$ and $g$ satisfy~\ref{i-a:  eta Morse}, then, by~\eqref{H^bullet(Omega_sm z(M C) d_z) cong H_z^*(M C)},
\[
\Phi_z:(E_{z,\text{\rm sm}},d_z)\to(\bfC^\bullet,\bfd_z)
\]
is also a quasi-isomorphism. Since a direct adaptation of \cite[Appendix~A]{BurgheleaHaller2004} shows that, for $k=0,\dots,n$, $\dim H^k(\bfC^\bullet,\bfd_z)$ is independent of $z\in\C$ with $|\mu|\gg0$, we get~\eqref{beta_z^k = betaNo^k} because any $\xi\in H^1(M,\R)$ is represented by a Morse form.

\subsubsection{Morse complex with coefficients in a flat vector bundle}\label{sss: Morse complex with F}

With more generality, for a flat vector bundle $F$, we may consider $(C^\bullet(X,W^-,F),\bfd^F)$, where $C^k(X,W^-,F)=\bigoplus_{p\in\YY_k}F_p$, and $\bfd^F\bfe$ ($\bfe\in F_q$, $q\in\YY_{k-1}$) is defined like in the right-hand side of~\eqref{bfd bfe_q}, replacing $\bfe_p$ with the parallel transport of $\bfe$ along $\bar\gamma^{-1}$ \cite[Section~1c)]{BismutZhang1992}. This is the dual of the complex $(C_\bullet(X,W^-,F^*),\partial^{F^*})$, where $C_k(X,W^-,F^*)=\bigoplus_{p\in\YY_k}F_p^*$, and $\partial^Ff$ ($f\in F_p^*$, $p\in\XX_k$) is defined like in the right-hand side of~\eqref{partial W^-_p}, replacing $W^-_q$ with the parallel transport of $f$ along $\bar\gamma$. A quasi-isomorphism
\[
\Phi^F=\Phi^{X,F}:(\Omega(M,F),d)\to\big(C^\bullet(X,W^-,F),\bfd^F\big)
\]
can be defined like $\Phi_z$ \cite[Theorem~2.9]{BismutZhang1992}, using the isomorphism
\[
\Omega^\bullet\big(\widehat W^-_p,(\hat\iota^-_p)^*F\big)
\cong\Omega^\bullet\big(\widehat W^-_p\big)\otimes F_p
\]
given by the parallel transport of $(\hat\iota^-_p)^*F$. If $F=\LL^z$ (\Cref{sss: LL^z}), then
\[
\big(C^\bullet(X,W^-,\LL^z),\bfd^{\LL^z}\big)\equiv(\bfC^\bullet,\bfd_z)\;,\quad
\Phi^{\LL^z}\equiv\Phi_z\;.
\]

\subsubsection{Hodge theory of the Morse complex}\label{sss: bfDelta_z}

Consider the Hermitian scalar product on $\bfC^\bullet$ so that the canonical base $\bfe_p$ ($p\in\YY$) is orthonormal. All operators induced by $\bfd_z$ and this Hermitian structure are called \emph{perturbed Morse operators}. For instance, besides $\bfd_z$, we have the perturbed Morse operators
\begin{gather*}
\bfdelta_z=\bfd_z^*\;,\quad\bfD_z=\bfd_z+\bfdelta_z\;,\quad\bfDelta_z=\bfD_z^2=\bfd_z\bfdelta_z+\bfdelta_z\bfd_z\;.
\end{gather*}
In particular, it will be said that $\bfDelta_z$ is the \emph{perturbed Morse Laplacian}, and its eigenvalues will be called \emph{perturbed Morse eigenvalues}. If $z=0$, we omit the subscript ``$z$'' and the word ``perturbed''. From~\eqref{bfd_z bfe_q}, we easily get
\begin{equation}\label{bfdelta_z bfe_p}
\bfdelta_z\bfe_p=\sum_{q\in\YY_{k-1},\ \gamma\in\TT(p,q)}e^{\bar z\eta(\gamma)}\epsilon(\gamma)\,\bfe_q\;,
\end{equation}
for $p\in\YY_k$. We also have
\begin{gather*}
\bfC^\bullet=\ker\bfDelta_z\oplus\im\bfd_z\oplus\im\bfdelta_z\;,\\
\ker\bfDelta_z=\ker\bfD_z=\ker\bfd_z\cap\ker\bfdelta_z\;,\quad
\im\bfDelta_z=\im\bfD_z=\im\bfd_z\oplus\im\bfdelta_z\;.
\end{gather*}
The orthogonal projections of $\bfC^\bullet$ to $\ker\bfDelta_z$, $\im\bfd_z$ and $\im\bfdelta_z$ are denoted by $\bfPi_z=\bfPi^0_z$, $\bfPi_z^1$ and $\bfPi_z^2$, respectively. The compositions $\bfd_z^{-1}\bfPi_z^1$, $\bfdelta_z^{-1}\bfPi_z^2$ and $\bfD_z^{-1}\bfPi_z^\perp$ are defined like in \Cref{sss: perturbations}, and there is an obvious version of the commutative diagram~\eqref{CD im delta_z k+1 ...}.

\subsection{The small complex vs the Morse complex}\label{ss: small complex vs Morse complex}

Our main objects of interest are the form $\eta\in Z^1(M;\R)$ and the Riemannian metric $g$; $X$ plays an auxiliary role. As indicated in \Cref{sss: Lyapunov}, by~\ref{i-a:  X ...}, we can choose some $\eta\in\xi$ and $g$ satisfying~\ref{i-a:  eta Morse} and~\ref{i-a:  eta Lyapunov}  (\Cref{ss: intro-Witten}). Thus, unless otherwise indicated, assume from now on that $X$, $\eta$ and $g$ satisfy~\ref{i-a:  eta Lyapunov}, besides~\ref{i-a:  eta Morse} and~\ref{i-a:  X ...}. In particular, $\YY=\Zero(\eta)$.

For every $p\in\YY$, consider the functions $h_{\eta,p}$, $h_{X,p}$, $h_{\eta,p}^-$ and $\hat h_{\eta,p}^-$ defined in \Cref{ss: Morse forms,sss: Morse-type zeros,sss: perturbed Morse complex}. By~\ref{i-a:  eta Lyapunov}, we have 
\begin{alignat}{2}
h_{\eta,p}&=h_{X,p}&\quad\text{on}\quad&U_p\;,\notag\\
h_{\eta,p}^-&=h_{\eta,p}=-\frac12|x_p^-|^2&\quad\text{on}\quad&U_p^-\;,\label{h_eta p^-}\\
h_{\eta,p}^-&<0&\quad\text{on}\quad&W^-_p\setminus\{p\}\;.\label{h_eta p^- < 0}
\end{alignat}
From now on, the subscripts $X$ and $\eta$ will be dropped from the notation of these functions.

Continuing with the notation of \Cref{sss: perturbed Morse complex}, let $J_z:\bfC^\bullet\to E_z$ be the $\C$-linear isometry given by $J_z(\bfe_p)= e_{p,z}$, and let $\Psi_z=P_{z,\text{\rm sm}}J_z:\bfC^\bullet\to E_{z,\text{\rm sm}}$, which is an isomorphism for $\mu\gg0$ (\Cref{c: P_z sm : E_z -> E_z sm iso}). By \Cref{p: | alpha - P_z sm alpha |_m i nu},
\[
\|\Psi_z\bfe\|=\big(1+O\big(e^{-c\mu}\big)\big)\|\bfe\|\quad(\mu\to+\infty)
\]
for all $\bfe\in\bfC^\bullet$. Using polarization (see e.g.\ \cite[Section~I.6.2]{Kato1995}) and conjugation, this means that, as $\mu\to+\infty$,
\begin{equation}\label{Psi_z^* Psi_z = 1 + O(e^-c mu)}
\Psi_z^*\Psi_z=1+O\big(e^{-c\mu}\big)\;,\quad\Psi_z\Psi_z^*=1+O\big(e^{-c\mu}\big)\;.
\end{equation}

\begin{notation}\label{n: asymp_0}
Consider functions $u(z)$ and $v(z)$ ($z\in\C)$ with values in Banach spaces. The notation $u(z)\asymp_0 v(z)$ ($\mu\to\pm\infty$) means
\[
u(z)=v(z)+O\big(e^{-c|\mu|}\big)\quad(\mu\to\pm\infty)\;.
\]
This notation may be used even when the Banach spaces depend on $z$.
\end{notation}

\begin{thm}[{Cf.\ \cite[Theorem 6.11]{BismutZhang1994}, \cite[Theorem~6.9]{Zhang2001}, \cite[Theorem~4]{BurgheleaHaller2001}}]\label{t: Phi_z+tau Psi_z}
For every $\tau\in\R$, as $\mu\to+\infty$,
\[
\Phi_{z+\tau}\Psi_z
\asymp_0\Big(\frac{\pi}{\mu+\tau/2}\Big)^{\sN/2}\Big(\frac{\mu}{\pi}\Big)^{n/4}\;.
\]
\end{thm}

\begin{proof}
We adapt the proof of \cite[Theorem~6.9]{Zhang2001} to the case of complex parameter. For every $p\in\YY_k$,
\begin{equation}\label{Phi_z + tau Psi_z bfe_p}
\Phi_{z+\tau}\Psi_z\bfe_p=\sum_{q\in\YY_k}\bfe_q
\int_{\widehat W^-_q}e^{(z+\tau)\hat h_q^-}(\hat\iota_q^-)^*P_{z,\text{\rm sm}} e_{p,z}\;.
\end{equation}
Then the result follows by checking the asymptotics of these integrals using the compactness of $\widehat W^-_q$.

In the case $q=p$, by~\eqref{h_eta p^- < 0} and \Cref{c: | alpha - P_z sm alpha |_L^infty},
\[
\int_{\widehat W^-_p}e^{(z+\tau)\hat h_p^-}(\hat\iota_p^-)^*(P_{z,\text{\rm sm}}-1) e_{p,z}\asymp_00\;.
\]
But, by \Cref{p: Delta'_p mu}~\ref{i: e'_p mu},~\eqref{rho_p}--\eqref{a_mu} and~\eqref{h_eta p^-},
\begin{multline}\label{int_widehat W^-_p e^(z+tau) hat h_p^- (hat iota_p^-)^* e_p z}
\int_{\widehat W^-_p}e^{(z+\tau)\hat h_p^-}(\hat\iota_p^-)^* e_{p,z}
=\int_{\widehat W^-_p}e^{(z+\tau)\hat h_p^-}(\hat\iota_p^-)^*\big(e^{-i\nu h_p}e_{p,\mu}\big)\\
=\int_{\widehat W^-_p}e^{(\mu+\tau)\hat h_p^-}(\hat\iota_p^-)^*e_{p,\mu}
=\frac1{a_\mu}\Big(\int_{-2r}^{2r}\rho(x)e^{-(2\mu+\tau)x^2/2}\,dx\Big)^k\\
=\Big(\frac{\pi}{\mu+\tau/2}\Big)^{k/2}\Big(\frac\mu\pi\Big)^{n/4}\big(1+O\big(e^{-c\mu}\big)\big)\;.
\end{multline}
(When $\tau=0$, the last equality is the same as \cite[Eq.~(6.30)]{Zhang2001}.)

For $q\ne p$ in $\YY_k$, since $e_{p,z}=0$ on $\overline{W^-_q}$ because $U_p\cap\overline{W^-_q}=\emptyset$ (\Cref{sss: completion}), like in the previous case, we get
\[
\int_{\widehat W^-_q}e^{(z+\tau)\hat h_q^-}(\hat\iota_q^-)^*P_{z,\text{\rm sm}} e_{p,z}\asymp_00\;.\qedhere
\]
\end{proof}

\begin{cor}\label{c: Phi_z: E_z sm^bullet to bfC^bullet iso}
For every $\tau\in\R$, if $\mu\gg0$, then $\Phi_{z+\tau}:E_{z,\text{\rm sm}}\to\bfC^\bullet$ is a linear\footnote{It is an isomorphism of complexes if $\tau=0$.} isomorphism.
\end{cor}

\begin{proof}
Apply \Cref{t: Phi_z+tau Psi_z,c: P_z sm : E_z -> E_z sm iso}.
\end{proof}

\begin{rem}\label{r: Phi_z: E_z to bfC^bullet iso}
The argument of the proof of \Cref{t: Phi_z+tau Psi_z} shows that
\[
\Phi_zJ_z
=\Big(\frac\pi{\mu}\Big)^{\sN/2-n/4}+O\big(e^{-c\mu}\big)
\quad(\mu\to+\infty)\;.
\]
So $\Phi_z:E_z\to\bfC^\bullet$ is an isomorphism for $\mu\gg0$ (see also \cite[Lemma~5.2]{BurgheleaHaller2008}).
\end{rem}

Let
\[
\widetilde\Psi_z=\Big(\frac{\mu}{\pi}\Big)^{\sN/2-n/4}\Psi_z:
\bfC^\bullet\to E_{z,\text{\rm sm}}\;.
\]

\begin{cor}\label{c: widetilde Psi_z^* widetilde Psi_z}
Consider $\widetilde\Psi_z^*:E_{z,\text{\rm sm}}\to\bfC^\bullet$. As $\mu\to+\infty$,
\[
\widetilde\Psi_z^*\widetilde\Psi_z=\Big(\frac\mu\pi\Big)^{\sN-n/2}+O\big(e^{-c\mu}\big)\;,\quad
\widetilde\Psi_z\widetilde\Psi_z^*=\Big(\frac\mu\pi\Big)^{\sN-n/2}+O\big(e^{-c\mu}\big)\;.
\]
\end{cor}

\begin{proof}
This is a direct consequence of~\eqref{Psi_z^* Psi_z = 1 + O(e^-c mu)}.
\end{proof}

\begin{cor}\label{c: Phi_z+tau widetilde Psi_z}
For any $\tau\in\R$, consider $\Phi_{z+\tau}:E_{z,\text{\rm sm}}\to\bfC^\bullet$. As $\mu\to+\infty$,
\[
\Phi_{z+\tau}\widetilde\Psi_z\asymp_0\Big(\frac{\mu}{\mu+\tau/2}\Big)^{\sN/2}\;,\quad
\widetilde\Psi_z\Phi_{z+\tau}\asymp_0\Big(\frac{\mu}{\mu+\tau/2}\Big)^{\sN/2}\;.
\]
\end{cor}

\begin{proof}
The first relation is a restatement of \Cref{t: Phi_z+tau Psi_z}. The second relation follows by conjugating the first one by $\widetilde\Psi_z$ and using~\Cref{c: widetilde Psi_z^* widetilde Psi_z}.
\end{proof}

\begin{cor}\label{c: widetilde Psi_z^-1 approx Phi_z}
As $\mu\to+\infty$, $\widetilde\Psi_z^{-1}\asymp_0\Phi_z$ on $E_{z,\text{\rm sm}}$.
\end{cor}

\begin{proof}
By \Cref{c: Phi_z+tau widetilde Psi_z,c: widetilde Psi_z^* widetilde Psi_z}, on $E_{z,\text{\rm sm}}$,
\[
\widetilde\Psi_z^{-1}\asymp_0\widetilde\Psi_z^{-1}\widetilde\Psi_z\Phi_z=\Phi_z\;.\qedhere
\]
\end{proof}

In the rest of this section, consider $\Phi_z:E_{z,\text{\rm sm}}\to\bfC^\bullet$ unless otherwise indicated.

\begin{cor}\label{c: Phi_z^* Phi_z}
As $\mu\to+\infty$,
\[
\Phi_z^*\Phi_z\asymp_0\Big(\frac\pi\mu\Big)^{\sN-n/2}\;,\quad
\Phi_z\Phi_z^*\asymp_0\Big(\frac\pi\mu\Big)^{\sN-n/2}\;.
\]
\end{cor}

\begin{proof}
We show the first relation, the other one being similar. By~\Cref{c: widetilde Psi_z^* widetilde Psi_z,c: widetilde Psi_z^-1 approx Phi_z}, on $E_{z,\text{\rm sm}}$,
\[
\Phi_z^*\Phi_z\asymp_0\big(\widetilde\Psi_z^{-1}\big)^*\widetilde\Psi_z^{-1}
=\big(\widetilde\Psi_z^*\big)^{-1}\widetilde\Psi_z^{-1}
=\big(\widetilde\Psi_z\widetilde\Psi_z^*\big)^{-1}
\asymp_0\Big(\frac{\pi}{\mu}\Big)^{\sN-n/2}\;.\qedhere
\]
\end{proof}

\begin{cor}\label{c: widetilde Psi_z approx Phi_z^*}
As $\mu\to+\infty$,
\[
\widetilde\Psi_z\asymp_0\Big(\frac\mu\pi\Big)^{\sN-n/2}\Phi_z^*\;.
\]
\end{cor}

\begin{proof}
By \Cref{c: Phi_z+tau widetilde Psi_z,c: Phi_z^* Phi_z},
\[
\widetilde\Psi_z\asymp_0\Big(\frac\mu\pi\Big)^{\sN-n/2}\widetilde\Psi_z\Phi_z\Phi_z^*
\asymp_0\Big(\frac\mu\pi\Big)^{\sN-n/2}\Phi_z^*\;.\qedhere
\]
\end{proof}

\begin{cor}\label{c: Phi_z+tau P_z+tau sm widetilde Psi_z}
For every $\tau\in\R$, as $\mu\to+\infty$,
\[
\Phi_{z+\tau}P_{z+\tau,\text{\rm sm}}\widetilde\Psi_z\asymp_0\Big(\frac{\mu}{\mu+\tau/2}\Big)^{\sN/2}
+O\big(\mu^{-1}\big)\;.
\]
\end{cor}

\begin{proof}
By \Cref{c: Phi_z+tau widetilde Psi_z,c: widetilde Psi_z^* widetilde Psi_z,c: Phi_z^* Phi_z,p: P_z sm P_z+tau sm},
\begin{align*}
\Phi_{z+\tau}P_{z+\tau,\text{\rm sm}}\widetilde\Psi_z
&=\Phi_{z+\tau}(P_{z+\tau,\text{\rm sm}}-P_{z,\text{\rm sm}})\widetilde\Psi_z+\Phi_{z+\tau}\widetilde\Psi_z\\
&\asymp_0 O\big(\mu^{-1}\big)+\Big(\frac{\mu}{\mu+\tau/2}\Big)^{\sN/2}\;.\qedhere
\end{align*}
\end{proof}

\begin{cor}\label{c: d_z sm approx ...}
As $\mu\to+\infty$,
\[
d_{z,\text{\rm sm}}\asymp_0\widetilde\Psi_z\bfd_z\Phi_z\;,\quad
\delta_{z,\text{\rm sm}}\asymp_0\widetilde\Psi_z\bfdelta_z\Phi_z\;.
\]
\end{cor}

\begin{proof}
By \Cref{t: spec Delta_z,c: Phi_z+tau widetilde Psi_z},
\[
d_{z,\text{\rm sm}}\asymp_0\widetilde\Psi_z\Phi_zd_{z,\text{\rm sm}}
=\widetilde\Psi_z\bfd_z\Phi_z\;.
\]
Now, taking adjoints and using \Cref{c: widetilde Psi_z^* widetilde Psi_z,c: Phi_z^* Phi_z,c: widetilde Psi_z approx Phi_z^*}, we obtain
\[
\delta_{z,\text{\rm sm}}=\Phi_{z,\text{\rm sm}}^*\bfdelta_z\widetilde\Psi_z^*
\asymp_0\widetilde\Psi_z\bfdelta_z\Phi_z\;.\qedhere
\]
\end{proof}

Let $\widetilde\bfPi_z=\widetilde\bfPi^0_z$, $\widetilde\bfPi^1_z$ and $\widetilde\bfPi^2_z$ be the orthogonal projections of $\bfC^\bullet$ to $\Phi_z(\ker\Delta_{z,\text{\rm sm}})$, $\Phi_z(\im d_{z,\text{\rm sm}})$ and $\Phi_z(\im\delta_{z,\text{\rm sm}})$, respectively. Note that $\widetilde\bfPi^1_z=\widetilde\bfPi^1_z\bfPi^1_z$.

\begin{cor}\label{c: Pi^1_z sm}
For $j=0,1,2$, as $\mu\to+\infty$,
\[
\Phi_z\Pi^j_{z,\text{\rm sm}}\asymp_0\widetilde\bfPi^j_z\Phi_z\;,\quad
\Pi^j_{z,\text{\rm sm}}\asymp_0\widetilde\Psi_z\widetilde\bfPi^j_z\Phi_z\;,\quad
\Pi^j_{z,\text{\rm sm}}\widetilde\Psi_z\asymp_0\widetilde\Psi_z\widetilde\bfPi^j_{z,\text{\rm sm}}\;.
\]
\end{cor}

\begin{proof}
We only prove the case of $\widetilde\bfPi^2_z$, the other cases being similar. Let $\alpha_{z,1},\dots,\alpha_{z,p_z}$ be an orthonormal frame of $\delta_z(E_{z,\text{\rm sm}}^{k+1})$. So $\Phi_z\alpha_{z,1},\dots,\Phi_z\alpha_{z,p_z}$ is a base of $\Phi_z\delta_z(E_{z,\text{\rm sm}}^{k+1})$ for $\mu\gg0$ by \Cref{c: Phi_z: E_z sm^bullet to bfC^bullet iso}. Applying the Gram-Schmidt process to this base, we get an orthonormal base $\bff_{z,1},\dots,\bff_{z,p_z}$ of $\Phi_z\delta_z(E_{z,\text{\rm sm}}^{k+1})$. By \Cref{c: Phi_z^* Phi_z}, 
\[
\langle\Phi_z\alpha_{z,a},\Phi_z\alpha_{z,b}\rangle\asymp_0\Big(\frac\pi\mu\Big)^{k-n/2}\delta_{ab}\;,
\]
for $1\le a,b\le p_z$. So
\[
\bff_{z,a}\asymp_0\Big(\frac\mu\pi\Big)^{k/2-n/4}\Phi_z\alpha_{z,a}\;.
\]
Hence, by \Cref{c: Phi_z^* Phi_z}, for any $\beta\in E_{z,\text{\rm sm}}^k$,
\begin{align*}
\widetilde\bfPi^2_z\Phi_z\beta&=\sum_{a=1}^{p_z}\langle\Phi_z\beta,\bff_{z,a}\rangle\bff_{z,a}
\asymp_0\Big(\frac\mu\pi\Big)^{k-n/2}\sum_{a=1}^{p_z}\langle\Phi_z\beta,\Phi_z\alpha_{z,a}\rangle\Phi_z\alpha_{z,a}\\
&\asymp_0\sum_{a=1}^m\langle\beta,\alpha_{z,a}\rangle\Phi_z\alpha_{z,a}
=\Phi_z\Pi^2_{z,\text{\rm sm}}\beta\;.
\end{align*}
This shows the first relation of the statement because $\dim E_{z,\text{\rm sm}}^k<\infty$. Then the other stated relations follow using \Cref{c: Phi_z+tau widetilde Psi_z,c: widetilde Psi_z^* widetilde Psi_z,c: Phi_z^* Phi_z}.
\end{proof}

According to \Cref{c: Phi_z: E_z sm^bullet to bfC^bullet iso}, in the following corollaries, we take $\mu\gg0$ so that $\Phi_z:E_{z,\text{\rm sm}}\to\bfC^\bullet$ is an isomorphism. 

\begin{cor}\label{c: (Phi_z^-1)^* Phi_z^-1}
As $\mu\to+\infty$,
\[
(\Phi_z^{-1})^*\Phi_z^{-1}\asymp_0\Big(\frac\mu\pi\Big)^{\sN-n/2}\;,\quad
\Phi_z^{-1}(\Phi_z^{-1})^*\asymp_0\Big(\frac\mu\pi\Big)^{\sN-n/2}\;.
\]
\end{cor}

\begin{proof}
By \Cref{c: Phi_z^* Phi_z}, for $\bfe\in\bfC^k$ with $\|\bfe\|=1$,
\[
\big\|\Phi_z^{-1}\bfe\big\|\asymp_0\Big(\frac\mu\pi\Big)^{k/2-n/4}\big\|\Phi_z\Phi_z^{-1}\bfe\big\|=\Big(\frac\mu\pi\Big)^{k/2-n/4}\;,
\]
yielding the first stated relation. The second one has a similar proof.
\end{proof}

\begin{cor}\label{c: Phi_z^-1}
As $\mu\to+\infty$,
\[
\Phi_z^*\asymp_0\Big(\frac\pi\mu\Big)^{\sN-n/2}\Phi_z^{-1}\;,\quad
\widetilde\Psi_z\asymp_0\Phi_z^{-1}\;.
\]
\end{cor}

\begin{proof}
By \Cref{c: Phi_z^* Phi_z,c: (Phi_z^-1)^* Phi_z^-1},
\[
\Phi_z^*=\Phi_z^*\Phi_z\Phi_z^{-1}\asymp_0\Big(\frac\pi\mu\Big)^{\sN-n/2}\Phi_z^{-1}\;,\quad
\widetilde\Psi_z=\widetilde\Psi_z\Phi_z\Phi_z^{-1}\asymp_0\Phi_z^{-1}\;.\qedhere
\]
\end{proof}

\begin{cor}\label{c: widetilde bfPi^j_z}
We have $\widetilde\bfPi^1_z=\bfPi^1_z$ for $\mu\gg0$, and $\widetilde\bfPi^2_z\asymp_0\bfPi^2_z$ as $\mu\to+\infty$.
\end{cor}

\begin{proof}
Since $\Phi_z(\im d_{z,\text{\rm sm}})=\im\bfd_z$ for $\mu\gg0$, we get $\widetilde\bfPi^1_z=\bfPi^1_z$. 

To prove $\widetilde\bfPi^2_z\asymp_0\bfPi^2_z$ as $\mu\to+\infty$, consider the notation of the proof of \Cref{c: Pi^1_z sm}. We have $\alpha_{z,a}=\delta_z\beta_{z,a}$ ($a=1,\dots,p_z$) for some base $\beta_{z,1},\dots,\beta_{z,p_z}$ of $\im d_{z,\text{\rm sm},k}$. Hence, by \Cref{c: Phi_z+tau widetilde Psi_z,c: Phi_z^* Phi_z,c: d_z sm approx ...},
\begin{equation}\label{Phi_z alpha_z a approx bfdelta_z Phi_z beta_z a}
\Phi_z\alpha_{z,a}=\Phi_z\delta_z\beta_{z,a}
\asymp_0\Phi_z\widetilde\Psi_z\bfdelta_z\Phi_z\beta_{z,a}
\asymp_0\bfdelta_z\Phi_z\beta_{z,a}\;,
\end{equation}
and $\bfdelta_z\Phi_z\beta_{z,1},\dots,\bfdelta_z\Phi_z\beta_{z,p_z}$ is a base of $\im\bfdelta_{z,k+1}$. Applying the Gram-Schmidt process to this base, we get an orthonormal base $\bfg_{z,1},\dots,\bfg_{z,p_z}$ of $\im\bfdelta_{z,k+1}$ satisfying $\bfg_{z,a}\asymp_0\bff_{z,a}$ by~\eqref{Phi_z alpha_z a approx bfdelta_z Phi_z beta_z a}. Then, for any $\bfe\in\bfC^k$ with $\|\bfe\|=1$,
\[
\widetilde\bfPi^2_z\bfe=\sum_{a=1}^{p_z}\langle\bfe,\bfg_{z,a}\rangle\bfg_{z,a}
\asymp_0\sum_{a=1}^{p_z}\langle\bfe,\bff_{z,a}\rangle\bff_{z,a}
=\bfPi^2_z\bfe\;.\qedhere
\]
\end{proof}

\begin{cor}\label{c: d_z sm = ...}
We have
\[
d_{z,\text{\rm sm}}=\Phi_z^{-1}\bfd_z\Phi_z\;,\quad
d_{z,\text{\rm sm}}^{-1}\Pi^1_{z,\text{\rm sm}}
=\Pi^2_{z,\text{\rm sm}}\Phi_z^{-1}\bfd_z^{-1}\Phi_z\Pi^1_{z,\text{\rm sm}}\;.
\]
\end{cor}

\begin{proof}
The first equality follows like the first relation of \Cref{c: d_z sm approx ...}, using $\Phi_z^{-1}$ instead of $\widetilde\Psi_z$. To prove the second one, take any $\alpha\in\im d_{z,\text{\rm sm}}$. Since
\[
d_z\Pi^2_{z,\text{\rm sm}}\Phi_z^{-1}\bfd_z^{-1}\Phi_z\alpha
=d_z\Phi_z^{-1}\bfd_z^{-1}\Phi_z\alpha
=\Phi_z^{-1}\bfd_z\bfd_z^{-1}\Phi_z\alpha
=\alpha
\]
with $\Pi^2_{z,\text{\rm sm}}\Phi_z^{-1}\bfd_z^{-1}\Phi_z\alpha\in\im\delta_{z,\text{\rm sm}}$, we obtain
\[
\Pi^2_{z,\text{\rm sm}}\Phi_z^{-1}\bfd_z^{-1}\Phi_z\alpha
=d_{z,\text{\rm sm}}^{-1}\alpha\;.\qedhere
\]
\end{proof}

\subsection{Derivatives of some homomorphisms}\label{ss: derivatives}

\begin{thm}\label{t: partial_z(Phi_z Psi_z)}
As $\mu\to+\infty$,
\[
\partial_z(\Phi_z\Psi_z),\partial_{\bar z}(\Phi_z\Psi_z)
\asymp_0\Big(\frac n{8\mu}-\frac \sN{4\mu}\Big)\Big(\frac{\pi}{\mu}\Big)^{\sN/2-n/4}\;.
\]
\end{thm}

\begin{proof}
By~\eqref{Phi_z + tau Psi_z bfe_p},
\begin{multline}\label{partial_z(Phi_z Psi_z)}
\partial_z(\Phi_z\Psi_z\bfe_p)=\\
\sum_{q\in\YY_k}\bfe_q
\bigg(\int_{\widehat W^-_q}\hat h_q^-e^{z\hat h_q^-}(\hat\iota_q^-)^*P_{z,\text{\rm sm}} e_{p,z}
+\int_{\widehat W^-_q}e^{z\hat h_q^-}(\hat\iota_q^-)^*\partial_z(P_{z,\text{\rm sm}} e_{p,z})\bigg)\;,
\end{multline}
for every $p\in\YY_k$ ($k=0,\dots,n$). We estimate each of these integrals.

Like in the proof of \Cref{t: Phi_z+tau Psi_z}, we get, for any $q\ne p$ in $\YY_k$,
\begin{align}
\int_{\widehat W^-_p}\hat h_p^-e^{z\hat h_p^-}(\hat\iota_p^-)^*(P_{z,\text{\rm sm}}-1) e_{p,z}
&\asymp_00\;,
\label{int_widehat W^-_p hat h_p^- e^z hat h_p^- (hat iota_p^-)^*(P_z sm - 1) e_p z}\\
\int_{\widehat W^-_q}\hat h_q^-e^{z\hat h_q^-}(\hat\iota_q^-)^*P_{z,\text{\rm sm}} e_{p,z}
&\asymp_00\;.
\label{int_widehat W^-_q hat h_q^- e^z hat h_q^- (hat iota_q^-)^* P_z sm e_p z}
\end{align}
Moreover, by \Cref{p: Delta'_p mu}~\ref{i: e'_p mu},~\eqref{rho_p}--\eqref{a_mu} and~\eqref{int_-2r^2r rho(x)^2 x^2 e^-mu x^2 dx},
\begin{multline}\label{int_widehat W^-_p hat h_p^- e^z hat h_p^- (hat iota_p^-)^*e_p z}
\int_{\widehat W^-_p}\hat h_p^-e^{z\hat h_p^-}(\hat\iota_p^-)^*e_{p,z}\\
=-\frac k{2a_\mu}\bigg(\int_{-2r}^{2r}\rho(x)e^{-\mu x^2/2}\,dx\bigg)^{k-1}
\int_{-2r}^{2r}\rho(x)x^2e^{-\mu x^2/2}\,dx\\
=-\frac k{4\mu}\Big(\frac\pi\mu\Big)^{\frac k2-\frac n4}+O(e^{-c\mu})\;.
\end{multline}

On the other hand, by~\eqref{h_eta p^- < 0} and \Cref{p: | partial_z(P_z sm e_p z - e_p z) |_m i nu} ,
\[
\int_{\widehat W^-_q}e^{z\hat h_q^-}(\hat\iota_q^-)^*\partial_z(P_{z,\text{\rm sm}}e_{p,z}-e_{p,z})
\asymp_00\;,
\]
for all $q\in\YY_k$. In the case $q=p$, by~\eqref{int_widehat W^-_p e^(z+tau) hat h_p^- (hat iota_p^-)^* e_p z} and \Cref{l: partial_z e_p z},
\begin{multline}\label{int_widehat W^-_p e^z hat h_q^- (hat iota_q^-)^* partial_z e_p z}
\int_{\widehat W^-_p}e^{z\hat h_p^-}(\hat\iota_p^-)^*\partial_ze_{p,z}
=\Big(\frac n{8\mu}+O(e^{-c\mu})\Big)\int_{\widehat W^-_p}e^{z\hat h_p^-}(\hat\iota_p^-)^*e_{p,z}\\
=\Big(\frac n{8\mu}+O(e^{-c\mu})\Big)\bigg(\Big(\frac \pi\mu\Big)^{\frac k2-\frac n4}+O(e^{-c\mu})\bigg)
=\frac n{8\mu}\Big(\frac \pi\mu\Big)^{\frac k2-\frac n4}+O(e^{-c\mu})\;.
\end{multline}
In the case $q\ne p$, using \Cref{l: partial_z e_p z} and arguing again like in the proof of \Cref{t: Phi_z+tau Psi_z}, we get
\begin{equation}\label{int_widehat W^-_q e^z hat h_q^- (hat iota_p^-)^* partial_z e_p z}
\int_{\widehat W^-_q}e^{z\hat h_q^-}(\hat\iota_q^-)^*\partial_ze_{p,z}\asymp_00\quad(\mu\to+\infty)\;.
\end{equation}
Now the result for $\partial_z$ follows from~\eqref{partial_z(Phi_z Psi_z)}--\eqref{int_widehat W^-_p hat h_p^- e^z hat h_p^- (hat iota_p^-)^*e_p z},~\eqref{int_widehat W^-_p e^z hat h_q^- (hat iota_q^-)^* partial_z e_p z} and~\eqref{int_widehat W^-_q e^z hat h_q^- (hat iota_p^-)^* partial_z e_p z}.

If we consider $\partial_{\bar z}$, the proof has to be modified as follows. In the analogue of~\eqref{partial_z(Phi_z Psi_z)}, the first term of the right-hand side must be removed. In the analogue of \Cref{l: partial_z e_p z}, we get $|x_p^-|^2$ instead of $|x_p^+|^2$ by the right-hand side of~\eqref{h - h(p) around p} and~\eqref{partial_mu e_p z}. So $\partial_{\bar z}(\Phi_z\Psi_z)$ has the same final expression as $\partial_z(\Phi_z\Psi_z)$ by~\eqref{int_widehat W^-_p hat h_p^- e^z hat h_p^- (hat iota_p^-)^*e_p z}.
\end{proof}

\begin{thm}\label{t: partial_z((Psi_z^* Psi_z)^pm1)}
As $\mu\to+\infty$,
\[
\partial_z\big((\Psi_z^*\Psi_z)^{\pm1}\big),\partial_{\bar z}\big((\Psi_z^*\Psi_z)^{\pm1}\big)=O\big(\mu^{-1}\big)\;.
\]
\end{thm}

\begin{proof}
We only show the case of $\partial_z$. Consider $P_{z,\text{\rm sm}}:E_z\to E_{z,\text{\rm sm}}$, whose adjoint is $P_z:E_{z,\text{\rm sm}}\to E_z$. Then, since $J_z:\bfC^\bullet\to E_z$ is an isometry,
\[
\Psi_z^*\Psi_z=(P_{z,\text{\rm sm}}J_z)^*P_{z,\text{\rm sm}}J_z=J_z^{-1}P_zP_{z,\text{\rm sm}}J_z\;.
\]
It follows that, for every $p\in\YY_k$ ($k=0,\dots,n$), 
\[
\Psi_z^*\Psi_z\bfe_p=\sum_{q\in\YY_k}\langle P_{z,\text{\rm sm}}e_{p,z},e_{q,z}\rangle\bfe_q\;.
\]
Therefore
\begin{multline*}
\partial_z(\Psi_z^*\Psi_z)\bfe_p=\\
\sum_{q\in\YY_k}\big(\langle\partial_z(P_{z,\text{\rm sm}})e_{p,z},e_{q,z}\rangle
+\langle P_{z,\text{\rm sm}}\partial_z(e_{p,z}),e_{q,z}\rangle
+\langle P_{z,\text{\rm sm}}e_{p,z},\partial_{\bar z}(e_{q,z})\rangle\big)\bfe_q\;.
\end{multline*}
Then, by \Cref{p: P_z sm P_z+tau sm,p: partial_z P_z sm = O(mu^-1)}, \Cref{l: partial_z e_p z} and its analogue for $\partial_{\bar z}$,
\begin{align*}
\partial_z(\Psi_z^*\Psi_z)\bfe_p
&=O\big(\mu^{-1}\big)+\Big(\frac n{8\mu}-\frac12\big\langle|x_p^+|^2e_{p,z},e_{p,z}\big\rangle\Big)\bfe_p
+O\big(e^{-c\mu}\big)\\
&=\Big(\frac n{8\mu}-\frac12\big\langle|x_p^+|^2e_{p,z},e_{p,z}\big\rangle\Big)\bfe_p+O\big(\mu^{-1}\big)\;.
\end{align*}
But, by~\eqref{a_mu} and~\eqref{int_-2r^2r rho(x)^2 x^2 e^-mu x^2 dx},
\begin{align*}
\big\langle|x_p|^2e_{p,z},e_{p,z}\big\rangle
&=\Big(\int_{-2r}^{2r}\rho(x)^2e^{-\mu x^2}\,dx\Big)^{n-1}(n-k)\int_{-2r}^{2r}y^2\rho(y)^2e^{-\mu y^2}\,dy\\
&=\frac{n-k}{2\mu}\Big(\frac\pi\mu\Big)^{\frac n2}+O\big(e^{-c\mu}\big)\;.
\end{align*}
Hence
\[
\partial_z(\Psi_z^*\Psi_z)\bfe_p
=\Big(\frac n{8\mu}-\frac{n-k}{4\mu}\Big(\frac\pi\mu\Big)^{\frac n2}\Big)\bfe_p+O\big(\mu^{-1}\big)
=O\big(\mu^{-1}\big)\;,
\]
yielding the stated expression for $\partial_z\big(\Psi_z^*\Psi_z)$.

Now, arguing like in the proof of~\eqref{partial_z((w-D_z)^-1)} and using~\eqref{Psi_z^* Psi_z = 1 + O(e^-c mu)}, we get
\begin{align*}
\partial_z\big((\Psi_z^*\Psi_z)^{-1}\big)
&=-(\Psi_z^*\Psi_z)^{-1}\partial_z(\Psi_z^*\Psi_z)(\Psi_z^*\Psi_z)^{-1}\\
&=-\big(1+O\big(e^{-c\mu}\big)\big)O\big(\mu^{-1}\big)\big(1+O\big(e^{-c\mu}\big)\big)
=O\big(\mu^{-1}\big)\;.\;\qedhere
\end{align*}
\end{proof}

\section{Asymptotics of the large zeta invariant}

\subsection{Preliminaries on Quillen metrics}\label{ss: prelim Quillen}

\subsubsection{Case of a finite dimensional complex}\label{sss: Quillen metrics finite dim}

All vector spaces considered here are over $\C$. For a line $\lambda$, its dual $\lambda^*$ is also denoted by $\lambda^{-1}$. For a vector space $V$ of finite dimension, recall that $\det V=\bigwedge^{\dim V}V$. For a graded vector space $V^\bullet$ of finite dimension, let $\det V^\bullet=\bigotimes_k(\det V^k)^{(-1)^k}$.

Now consider a finite dimensional cochain complex $(V^\bullet,\partial)$, whose cohomology is denoted by $H^\bullet(V)$. Then there is a canonical isomorphism \cite{KnudsenMumford1976}, \cite[Section~1~a)]{BismutGilletSoule1988}
\begin{equation}\label{det V^bullet cong det H^bullet(V)}
\det V^\bullet\cong\det H^\bullet(V)\;.
\end{equation}
Given a Hermitian metric on $V^\bullet$ so that the homogeneous components $V^k$ are orthogonal one another, the corresponding norm $\|\ \|_{V^\bullet}$ on $V^\bullet$ induces a metric $\|\ \|_{\det V^\bullet}$ on $\det V^\bullet$, which corresponds to a metric $\|\ \|_{\det H^\bullet(V)}$ on $\det H^\bullet(V)$ via~\eqref{det V^bullet cong det H^bullet(V)}. 

On the other hand, consider the induced Laplacian, $\square=(\partial+\partial^*)^2=\partial\partial^*+\partial^*\partial$, whose kernel is a graded vector subspace $\HH^\bullet$. Then finite dimensional Hodge theory gives an isomorphism $H^\bullet(V)\cong\HH^\bullet$, which induces an isomorphism
\begin{equation}\label{det H^bullet(V) cong det HH^bullet}
\det H^\bullet(V)\cong\det\HH^\bullet\;.
\end{equation}
The restriction of $\|\ \|_{V^\bullet}$ to $\HH^\bullet$ induces a metric $\|\ \|_{\det\HH^\bullet}$ on $\det\HH^\bullet$, which corresponds to another metric $|\ |_{\det H^\bullet(V)}$ on $\det H^\bullet(V)$ via~\eqref{det H^bullet(V) cong det HH^bullet}. 

Let $\square'$ denote the restriction $\square:\im\square\to\im\square$. For $s\in\C$, let
\begin{equation}\label{theta(s)}
\theta(s)=\theta(s,\square)=-\Str(\sN(\square')^{-s})\;.
\end{equation}
This defines a holomorphic function on $\C$. Then the above metrics on $\det H^\bullet(V)$ satisfy \cite[Proposition~1.5]{BismutGilletSoule1988}, \cite[Theorem~1.1]{BismutZhang1992}, \cite[Theorem~1.4]{BismutZhang1994}
\begin{equation}\label{| |_det H^bullet(V) = | |_det H^bullet(V) exp(frac12 theta'(0))}
\|\ \|_{\det H^\bullet(V)}=|\ |_{\det H^\bullet(V)}e^{\theta'(0)/2}\;.
\end{equation}
If $ H^\bullet(V)=0$, then $\det H^\bullet(V)\equiv\C$ is canonically generated by 1, and we have $\|1\|_{\det H^\bullet(V)}=e^{\theta'(0)/2}$. Using the orthogonal projection $\Pi^1:V\to\im\partial$, we can write~\eqref{theta(s)} as
\begin{equation}\label{Trs(check square^-s Pi^1)}
\theta(s)=-\Str\big((\square')^{-s}\Pi^1\big)\;.
\end{equation}

Let $(\widetilde V^\bullet,\tilde\partial)$ be another finite dimensional cochain complex, endowed with a Hermitian metric so that the homogeneous components are orthogonal to each other, and let $\phi:(V,\partial)\to(\widetilde V^\bullet,\tilde\partial)$ be an isomorphism of cochain complexes, which may not be unitary. Then (see the proof of \cite[Theorem~6.17]{BismutZhang1994})
\begin{equation}\label{... = Trs(log(phi^* phi)}
\log\bigg(\frac{\|\ \|_{\det H^\bullet(\widetilde V)}}{\|\ \|_{\det H^\bullet(V)}}\bigg)^2=\Str(\log(\phi^*\phi))\;.
\end{equation}

\subsubsection{Case of an elliptic complex}\label{sss: Quillen metrics elliptic complex}

Some of the concepts of \Cref{sss: Quillen metrics finite dim} extend to the case where $V^\bullet=C^\infty(M;E^\bullet)$, for some graded Hermitian vector bundle $E^\bullet$ over $M$, and $\partial$ is an elliptic differential complex of order one. Then $\det H^\bullet(V)$ is defined because $\dim H^\bullet(V)<\infty$. Moreover Hodge theory for the Laplacian $\square$ gives the isomorphism~\eqref{det H^bullet(V) cong det HH^bullet}. Thus at least the norm $|\ |_{\det H^\bullet(V)}$ is defined in this setting. Now the expression~\eqref{theta(s)} only defines $\theta(s)=\theta(s,\square)$ when $\Re s>n/2$, but it has a meromorphic extension to $\C$, denoted in the same way; indeed,~\eqref{theta(s)} becomes
\[
\theta(s)=\theta(s,\square)=-\zeta(s,\square,\sN\sw)\;,
\]
for $\Re s>n/2$, and therefore this equality also holds for the meromorphic extensions. Furthermore $\theta(s)$ is smooth at $s=0$ \cite{Seeley1967}, and $\theta'(0)$ can be considered as a renormalized version of the super-trace of the operator $\sN\log(\square')$, which is not of trace class. Thus the right-hand side of~\eqref{| |_det H^bullet(V) = | |_det H^bullet(V) exp(frac12 theta'(0))} is defined in this way and plays the role of an analytic version of the metric $\|\ \|_{\det H^\bullet(V)}$, which is not directly defined. This kind of metrics were introduced by D.~Quillen \cite{Quillen1985} for the case of the Dolbeault complex. The expression~\eqref{Trs(check square^-s Pi^1)} also holds in this case for $\Re s\gg0$; in fact, it becomes
\[
\theta(s)=-\zeta\big(s,\square,\Pi^1\sw\big)\;,
\]
where this zeta function can be shown to define a meromorphic function on $\C$, even though $\Pi^1$ is not a differential operator, and this equality holds as meromorphic functions.

\subsubsection{Reidemeister, Milnor and Ray-Singer metrics}\label{sss: Ray-Singer metric}

Let $F$ be a flat vector bundle over $M$, defined by a representation $\rho$ of $\pi_1M$, and let $\nabla^F$ denote its covariant derivative. Consider a smooth triangulation $K$ of $M$ and the corresponding cochain complex $C^\bullet(K,F)$ with coefficients in $F$, whose cohomology is isomorphic to $H^\bullet(M,F)$ via the quasi-isomorphism
\[
\Omega(M;F)\to C^\bullet(K,F)=C_\bullet(K,F^*)^*
\]
defined by integration of differential forms on smooth simplices. Given a Hermitian structure $g^F$ on $F$, its restriction to the fibers over the barycenters of the simplices induces a metric on $C^\bullet(K,F)$, and the concepts of \Cref{sss: Quillen metrics finite dim} can be applied. In this case, the left-hand side of~\eqref{| |_det H^bullet(V) = | |_det H^bullet(V) exp(frac12 theta'(0))} is called the \emph{Reidemeister metric}, denoted by $\|\ \|^{\text{\rm R}}_{\det H^\bullet(M,F)}$. 

If $\nabla^Fg^F=0$ ($\rho$ is unitary) and $H^\bullet(M,F)=0$, then the Reidemeister torsion $\tau_M(\rho)$ is defined using $K$, and it is a topological invariant of $M$ \cite{Franz1935, Reidemeister1935,deRham1950}. Moreover $\tau_M(\rho)=\|1\|^{\text{\rm R}}_{\det H^\bullet(M,F)}$ is the exponential factor of the right-hand side of~\eqref{| |_det H^bullet(V) = | |_det H^bullet(V) exp(frac12 theta'(0))} \cite[Proposition~1.7]{RaySinger1971}. If we only assume $\nabla^Fg^F=0$, then $\|\ \|^{\text{\rm R}}_{\det H^\bullet(M,F)}$ is still a topological invariant of $M$.

Next, given a vector field $X$ on $M$ satisfying~\ref{i-a:  X ...},  $H^\bullet(M,F)$ is also isomorphic to the cohomology of $(C^\bullet(-X,W^-,F),\bfd^F)$ via the quasi-isomorphism
\[
\Phi^{-X,F}:\Omega(M,F)\to C^\bullet(-X,W^-,F)=C_\bullet(-X,W^-,F^*)^*\;.
\]
This complex has a metric induced by $g^F$, like in \Cref{sss: perturbed Morse complex}, and the concepts of \Cref{sss: Quillen metrics finite dim} can be also applied. In this case, the left-hand side of~\eqref{| |_det H^bullet(V) = | |_det H^bullet(V) exp(frac12 theta'(0))} is called the \emph{Milnor metric}, denoted by $\|\ \|^{\text{\rm M},-X}_{\det H^\bullet(M,F)}$, and the metric factor of the right-hand side of~\eqref{| |_det H^bullet(V) = | |_det H^bullet(V) exp(frac12 theta'(0))} is denoted by $|\ |^{\text{\rm M},-X}_{\det H^\bullet(M,F)}$. If $\nabla^Fg^F=0$, then $\|\ \|^{\text{\rm M},-X}_{\det H^\bullet(M,F)}=\|\ \|^{\text{\rm R}}_{\det H^\bullet(M,F)}$ \cite[Theorem~9.3]{Milnor1966}.

Finally, the concepts of \Cref{sss: Quillen metrics elliptic complex} can be applied to $(\Omega(M,F),d^F)$, whose cohomology is again $H^\bullet(M,F)$. In this case, the right-hand side of~\eqref{| |_det H^bullet(V) = | |_det H^bullet(V) exp(frac12 theta'(0))} is called the \emph{Ray-Singer metric}, denoted by $\|\ \|^{\text{\rm RS}}_{\det H^\bullet(M,F)}$, and the metric factor of the right-hand side of~\eqref{| |_det H^bullet(V) = | |_det H^bullet(V) exp(frac12 theta'(0))} is denoted by $|\ |^{\text{\rm RS}}_{\det H^\bullet(M,F)}$. If $H^\bullet(M,F)=0$, then the exponential factor of the right-hand side of~\eqref{| |_det H^bullet(V) = | |_det H^bullet(V) exp(frac12 theta'(0))} is called the \emph{analytic torsion} or \emph{Ray-Singer torsion}, denoted by $T_M(\rho)$. These concepts were introduced by Ray and Singer \cite{RaySinger1971}, who conjectured that $T_M(\rho)=\tau_M(\rho)$ if $\nabla^Fg^F=0$ and $H^\bullet(M,F)=0$. Independent proofs of this conjecture were given by Cheeger \cite{Cheeger1979} and M\"{u}ller \cite{Muller1978}. This conjecture still holds true if the induced Hermitian structure $g^{\det F}$ on $\det F$ is flat, as shown at the same time by Bismut and Zhang \cite{BismutZhang1992} and M\"{u}ller \cite{Muller1978}. Actually, in [10], Bismut and Zhang reformulated the conjecture in the form $\|\ \|^{\text{\rm RS}}_{\det H^\bullet(M,F)}=\|\ \|^{\text{\rm R}}_{\det H^\bullet(M,F)}$. Moreover, they also considered the case where $g^{\det F}$ is not assumed to be flat \cite{BismutZhang1992,BismutZhang1994}, extending the above results by introducing an additional term. The first ingredient of this extra term  is the 1-form
\begin{equation}\label{theta(F g^F)}
\theta(F,g^F)=\tr\big((g^F)^{-1}\nabla^Fg^F\big)\;,
\end{equation}
which vanishes if and only if $g^{\det F}$ is flat. Moreover $\theta(F,g^F)$ is closed and its cohomology class of $\theta(F,g^F)$ is independent of the choice of $g^F$ \cite[Proposition~4.6]{BismutZhang1992}; this class measures the obstruction to the existence of a flat Hermitian structure on $\det F$.

Let $e(M,\nabla^M)$ be the representative of the Euler class of $M$ given by the Chern-Weil theory using $g^M$; it belongs to $\Omega^n(M,o(M))$ because $M$ may not be oriented. Let $\psi(M,\nabla^M)$ be the current of degree $n-1$ on $TM$ constructed in \cite{MathaiQuillen1986} (see also \cite[Section~3]{BismutGilletSoule1990}, \cite[Section~3]{BismutZhang1992}, \cite[Section~2]{BurgheleaHaller2006}, \cite[Section~4]{BurgheleaHaller2008}). Identify the image of the zero section of $TM$ with $M$, and identify the conormal bundle of $M$ in $TM$ with $T^*M$. Let $\delta_M$ be the current on $TM$ defined by integration on $M$, and let $\pi:TM\to M$ be the vector bundle projection.

\begin{prop}[{Bismut-Zhang \cite[Theorem~3.7]{BismutZhang1992}}] \label{p: psi(M nabla^M)}
The following holds:
\begin{enumerate}[{\rm(i)}]
\item\label{i: psi -> (pm1)^n psi} For any smooth function $\lambda:TM\to\R^\pm$, under the mapping $v\mapsto\lambda v$, $\psi(M,\nabla^M)$ is changed into $(\pm1)^n\psi(M,\nabla^M)$.
\item\label{i: wave front set of psi} The current $\psi(M,\nabla^M)$ is locally integrable, and its wave front set is contained in $T^*M$. Thus $\psi(M,\nabla^M)$ is smooth on $TM\setminus M$.
\item\label{i: solid angle} The restriction of $-\psi(M,\nabla^M)$ to the fibers of $TM\setminus M$ coincides with the solid angle defined by $g^M$.
\item\label{i: d psi} We have
\[
d\psi(M,\nabla^M)=\pi^*e(M,\nabla^M)-\delta_M\;.
\]
\end{enumerate}
\end{prop}

\begin{rem}\label{r: psi(M nabla^M)}
In \Cref{p: psi(M nabla^M)}, observe that~\ref{i: psi -> (pm1)^n psi} and~\ref{i: d psi} are compatible because $e(M,\nabla^M)=0$ if $n$ is odd. By~\ref{i: wave front set of psi}--\ref{i: d psi}, the restriction of $\psi(M,\nabla^M)$ to $TM\setminus M$ is induced by a smooth differential form on the sphere bundle which transgresses $e(M,\nabla^M)$ (such a differential form was already defined and used in \cite{Chern1944}).
\end{rem}

\begin{thm}[{Bismut-Zhang \cite[Theorem~0.2]{BismutZhang1992}, \cite[Theorem~0.2]{BismutZhang1994}}]
\label{t: Bismut-Zhang}
We have
\[
\log\Bigg(\frac{\|\ \|^{\text{\rm RS}}_{\det H^\bullet(M,F)}}{\|\ \|^{\text{\rm M},-X}_{\det H^\bullet(M,F)}}\Bigg)^2
=-\int_M\theta(F,g^F)\wedge(-X)^*\psi(M,\nabla^M)\;.
\]
\end{thm}

\begin{rem}\label{r: X = -grad_g' h}
By~\ref{i-a:  X ...}, $X=-\grad_{g'}h$ for some Morse function $h$ and some Riemannian metric $g'$ on $M$, which may not be the given metric $g^M$. If we fix $h$, the right-hand side of the equality in~\Cref{t: Bismut-Zhang} is independent of the choice of $X$ satisfying $X=-\grad_{g'}h$ for some $g'$ \cite[Proposition~6.1]{BismutZhang1992}.
\end{rem}

\Cref{t: Bismut-Zhang} will be applied to the case of the flat complex line bundle $\LL^z$ with a Hermitian structure $g^{\LL^z}$ (\Cref{sss: perturbations}). By~\eqref{nabla^LL^z g^LL^z = -2 mu eta otimes g^LL^z} and~\eqref{theta(F g^F)},
\begin{equation}\label{theta(LL^z g^LL^z) = -2 eta}
\theta(\LL^z,g^{\LL^z})=-2\mu\eta\;.
\end{equation}

\subsection{Asymptotics of the large zeta invariant}\label{ss: large zeta invariant}

We prove \Cref{t: zeta_sm/la(1 z)}~\ref{i: zetala(1 z)} here. With the notation of \Cref{sss: Quillen metrics elliptic complex}, consider the meromorphic function $\theta(s,z)=\theta(s,\Delta_z)$, also defined in~\eqref{theta(s,z)}, as well as its components $\theta_{\text{\rm sm/la}}(s,z)$ defined in~\eqref{theta_sm/la(s,z)}. Consider also the current $\psi(M,\nabla^M)$ of degree $n-1$ on $TM$ (\Cref{sss: Ray-Singer metric}). By \Cref{p: psi(M nabla^M)}~\ref{i: psi -> (pm1)^n psi},
\begin{equation}\label{-bfz_la(-eta)}
-\bfz_{\text{\rm la}}(-\eta)=(-1)^n\bfz_{\text{\rm la}}(\eta)\;.
\end{equation}

\begin{notation}\label{n: asymp_2}
Let $\asymp_1$ be defined like $\asymp_0$ in \Cref{n: asymp_0}, using $O(|\mu|^{-1})$ instead of $O(e^{-c|\mu|})$.
\end{notation}

Take some Morse function $h$ on $M$ such that $Xh<0$ on $M\setminus\YY$, and $h$ is in standard form with respect to $X$. Then $X=-\grad_{g'}h$ for some Riemannian metric $g'$ (\Cref{sss: gradient-like}), which may not be the given metric $g$. Consider the flat complex line bundle $\LL_{z\eta-dh}$ with the Hermitian structure $g^{\LL_{z\eta-dh}}$ (\Cref{sss: perturbations}). Note that $\bfd^{\LL_{z\eta-dh}}_{dh}\equiv\bfd_{z\eta}$ on $C^\bullet(X,W^-,\LL_{z\eta-dh})\equiv\bfC^\bullet(X)$. So, by~\eqref{theta(LL^z g^LL^z) = -2 eta}, \Cref{t: Bismut-Zhang,r: X = -grad_g' h},
\begin{equation}\label{... = int_M (z eta - dh) wedge X^*psi(M nabla^M)}
\log\frac{\|\ \|^{\text{\rm RS}}_{\det H^\bullet_z(M)}}{\|\ \|^{\text{\rm M},-X}_{\det H^\bullet_z(M)}}
=\int_M(\mu\eta-dh)\wedge (-X)^*\psi(M,\nabla^M)\;,
\end{equation}
where $H^\bullet_z(M)=H^\bullet_{z\eta}(M)$. With the notation of \Cref{sss: Ray-Singer metric}, let
\[
\|\ \|^{\text{\rm RS,sm}}_{\det H^\bullet_z(M)}
=|\ |^{\text{\rm RS}}_{\det H^\bullet_z(M)}e^{\thetasm'(0,z)/2}\;.
\]
By~\eqref{| |_det H^bullet(V) = | |_det H^bullet(V) exp(frac12 theta'(0))},
\begin{equation}\label{... thetala'(0) ...}
\log\frac{\|\ \|^{\text{\rm RS}}_{\det H^\bullet_z(M)}}{\|\ \|^{\text{\rm M},-X}_{\det H^\bullet_z(M)}}
=\log\frac{\|\ \|^{\text{\rm RS,sm}}_{\det H^\bullet_z(M)}}{\|\ \|^{\text{\rm M},-X}_{\det H^\bullet_z(M)}}
+\frac{\thetala'(0,z)}{2}\;.
\end{equation}

By~\eqref{... = Trs(log(phi^* phi)} and \Cref{c: Phi_z: E_z sm^bullet to bfC^bullet iso}, for $\mu\gg0$,
\begin{multline}\label{... = Trs(log((Psi_z^* Psi_z)^-1 (Phi_z Psi_z)^* Phi_z Psi_z))}
\log\Bigg(\frac{\|\ \|^{\text{\rm RS,sm}}_{\det H^\bullet_z(M)}}{\|\ \|^{\text{\rm M},-X}_{\det H^\bullet_z(M)}}\Bigg)^2
=-\Str(\log(\Phi_z^*\Phi_z))
=-\Str\big(\log\big(\Psi_z^{-1}\Phi_z^*\Phi_z\Psi_z\big)\big)\\
=-\Str\big(\log\big((\Psi_z^*\Psi_z)^{-1}(\Phi_z\Psi_z)^*\Phi_z\Psi_z\big)\big)\;.
\end{multline}
From~\eqref{Psi_z^* Psi_z = 1 + O(e^-c mu)} and \Cref{t: Phi_z+tau Psi_z,t: partial_z(Phi_z Psi_z),t: partial_z((Psi_z^* Psi_z)^pm1)}, we obtain
\begin{align*}
\big((\Psi_z^*\Psi_z)^{-1}(\Phi_z\Psi_z)^*\Phi_z\Psi_z\big)^{-1}
&=\Big(\frac{\pi}{\mu}\Big)^{\frac n2-\sN}+O\big(e^{-c\mu}\big)\;,\\
\partial_z\big((\Psi_z^*\Psi_z)^{-1}(\Phi_z\Psi_z)^*\Phi_z\Psi_z\big)
&=\partial_z\big((\Psi_z^*\Psi_z)^{-1}\big)(\Phi_z\Psi_z)^*\Phi_z\Psi_z\\
&\phantom{={}}{}+(\Psi_z^*\Psi_z)^{-1}(\partial_{\bar z}(\Phi_z\Psi_z))^*\Phi_z\Psi_z\\
&\phantom{={}}{}+(\Psi_z^*\Psi_z)^{-1}(\Phi_z\Psi_z)^*\partial_z(\Phi_z\Psi_z)\\
&\asymp_0\bigg(O\big(\mu^{-1}\big)+\Big(\frac n{4\mu}-\frac \sN{2\mu}\Big)\bigg)\Big(\frac{\pi}{\mu}\Big)^{\sN-\frac n2}\;.
\end{align*}
So
\begin{multline*}
\partial_z\Str\big(\log\big((\Psi_z^*\Psi_z)^{-1}(\Phi_z\Psi_z)^*\Phi_z\Psi_z\big)\big)\\
\begin{aligned}
&=\Str\big((\Psi_z^*\Psi_z)^{-1}(\Phi_z\Psi_z)^*\Phi_z\Psi_z\big)^{-1}
\partial_z\big((\Psi_z^*\Psi_z)^{-1}(\Phi_z\Psi_z)^*\Phi_z\Psi_z\big)\\
&=O\big(\mu^{-1}\big)+\Str\Big(\frac n{4\mu}-\frac \sN{2\mu}\Big)+O\big(e^{-c\mu}\big)
=O\big(\mu^{-1}\big)\;.
\end{aligned}
\end{multline*}
Then, by~\eqref{... = Trs(log((Psi_z^* Psi_z)^-1 (Phi_z Psi_z)^* Phi_z Psi_z))},
\begin{equation}\label{partial_z log(RS sm / M X)}
\partial_z\log\frac{\|\ \|^{\text{\rm RS,sm}}_{\det H^\bullet_z(M)}}{\|\ \|^{\text{\rm M},-X}_{\det H^\bullet_z(M)}}
=O\big(\mu^{-1}\big)\;.
\end{equation}

By taking the derivative with respect to $z$ of both sides of~\eqref{... = int_M (z eta - dh) wedge X^*psi(M nabla^M)}, and using~\eqref{... thetala'(0) ...},~\eqref{partial_z log(RS sm / M X)} and \Cref{c: zetala(1 z) = partial_z thetala'(0 z)}, we get $\zeta_{\text{\rm la}}(1,z)\asymp_1\bfz_{\text{\rm la}}$, as stated in \Cref{t: zeta_sm/la(1 z)}~\ref{i: zetala(1 z)}.

\begin{rem}\label{r: zeta_la(1 z) approx ... exact form - 2}
In the case where $\eta=dh$, \Cref{t: zeta_sm/la(1 z)}~\ref{i: zetala(1 z)} agrees with \Cref{t: zeta_la(1 z) approx ... mu -> infty exact form}. In fact, by \Cref{p: psi(M nabla^M)}~\ref{i: d psi}, \Cref{t: zeta_sm/la(1 z)}~\ref{i: zetala(1 z)} and the Stokes formula,
\begin{align*}
\zeta_{\text{\rm la}}(1,z)&\asymp_1-\int_Mh\,(-X)^*d\psi(M,\nabla^M)
=-\int_Mh\,(-X)^*(\pi^*e(M,\nabla^M)-\delta_M)\\
&=-\int_Mh\,e(M,\nabla^M)+\sum_{p\in\YY}(-1)^{\ind(p)}h(p)\;.
\end{align*}
\end{rem}

\section{Asymptotics of the small zeta-invariant}

\subsection{Condition on the integrals along instantons}\label{ss: condition on the integrals along instantons}

Let
\begin{align*}
\MM_p&=\MM_p(\eta,X)=-\max\{\,\eta(\gamma)\mid\gamma\in\TT^1_p\,\}\quad\big(p\in\YY_+\big)\;,\\
\MM_k&=\MM_k(\eta,X)=\min_{p\in\YY_k}\MM_p\quad(k=1,\dots,n)\;.
\end{align*}
Thus~\ref{i-a:  tight} means that $\MM_p=\MM_k$ for all $k=1,\dots,n$ and $p\in\YY_k$. The following result will be proved in \Cref{s: integrals along instantons}.

\begin{thm}\label{t: extension of Smale}
For every $\xi\in H^1(M,\R)$ and numbers $a_n\ge\dots\ge a_1\gg0$ or $a_1\ge\dots\ge a_n\gg0$, there is some $\eta\in\xi$, satisfying~\ref{i-a:  eta Morse} and~\ref{i-a:  eta Lyapunov} with the given $X$ and some metric $g$, such that $\MM_p(\eta,X)=a_k$ for all $k=1,\dots,n$ and $p\in\YY_k$.
\end{thm}

\begin{rem}\label{r: extension of Smale-1}
If $\xi\ne0$, for $p\in\YY_k$, $q\in\YY_{k-1}$ and $\gamma,\delta\in\TT(p,q)\subset\TT^1_p$, the period $\langle\xi,\bar\gamma\bar\delta^{-1}\rangle=\eta(\gamma)-\eta(\delta)$ may not be zero. Hence it may not be possible to get $\eta(\gamma)=-a_k$ for all $\gamma\in\TT^1_p$, contrary to the case where $\xi=0$.
\end{rem}

From now on, we assume $\eta$ satisfies~\ref{i-a:  tight}, besides~\ref{i-a:  eta Morse} and~\ref{i-a:  eta Lyapunov}. By \Cref{t: extension of Smale}, this is possible for any prescription of the class $\xi=[\eta]\in H^1(M,\R)$. Let $a_k=\MM_k(\eta,X)$ ($k=1,\dots,n$).  Then $-\eta$ also satisfies~\ref{i-a:  eta Morse},~\ref{i-a:  eta Lyapunov} and~\ref{i-a:  tight} with $-X$ and $g$, and $\MM_k(-\eta,-X)=a_{n-k+1}$. So, if $M$ is oriented, by \Cref{c: m^j_z k only depends on | XX_k | and xi,c: m_k = m_n-k}, 
\begin{equation}\label{-bfz(M g -eta)}
-\bfz_{\text{\rm sm}}(-\eta)=-\sum_{k=1}^n(-1)^k\big(1-e^{a_{n-k+1}}\big)m^1_{n-k+1}\;.
\end{equation}

\subsection{Asymptotics of the perturbed Morse operators}\label{ss: bfd'}

Consider the notation of \Cref{sss: perturbed Morse complex}. By~\eqref{bfd_z bfe_q}, 
\begin{equation}\label{bfd_z k-1}
\bfd_{z,k-1}=e^{-a_kz}(\bfd'_{k-1}+\bfd''_{z,k-1})\;,
\end{equation}
for $k=1,\dots,n$, where
\begin{align}
\bfd'_{k-1}\bfe_q&=\sum_{p\in\YY_k,\ \gamma\in\TT(p,q),\ \eta(\gamma)=-a_k}\epsilon(\gamma)\bfe_p\;,
\label{bfd'_k-1 bfe_q}\\
\bfd''_{z,k-1}\bfe_q&=\sum_{p\in\YY_k,\ \gamma\in\TT(p,q),\ \eta(\gamma)<-a_k}e^{z(a_k+\eta(\gamma))}\epsilon(\gamma)\bfe_p\;,\label{bfd''_k-1 bfe_q}
\end{align}
for $q\in\YY_{k-1}$. Observe that
\begin{equation}\label{e^a_k z bfd_z k-1 -> bfd'_k-1}
e^{a_kz}\bfd_{z,k-1}=\bfd'_{k-1}+O(e^{-c\mu})\quad(\mu\to+\infty)\;.
\end{equation}
So
\[
\bfd'_k\bfd'_{k-1}=\lim_{\mu\to+\infty}e^{(a_{k+1}+a_k)z}\bfd_{z,k}\bfd_{z,k-1}=0\;.
\]
Hence the operator $\bfd'=\sum_k\bfd'_k$ on $\bfC^\bullet$ satisfies  $(\bfd')^2=0$. Taking adjoints in~\eqref{bfd_z k-1}--\eqref{bfd''_k-1 bfe_q}, or using~\eqref{bfdelta_z bfe_p}, we also get
\begin{equation}\label{bfdelta_z k}
\bfdelta_{z,k}=e^{-a_k\bar z}(\bfdelta'_k+\bfdelta''_{z,k})\;,
\end{equation}
for $k=1,\dots,n$, where
\begin{align}
\bfdelta'_k\bfe_p&=\sum_{q\in\YY_{k-1},\ \gamma\in\TT(p,q),\ \eta(\gamma)=-a_k}
\epsilon(\gamma)\,\bfe_q\;,\label{bfdelta'_k-1 bff_p}\\
\bfdelta''_{z,k}\bfe_p&=\sum_{q\in\YY_{k-1},\ \gamma\in\TT(p,q),\ \eta(\gamma)=-a_k}e
^{\bar z(a_k+\eta(\gamma))}\epsilon(\gamma)\bfe_q\;,\label{bfddelta''_k-1 bfe_q}
\end{align}
for $p\in\YY_k$. Moreover~\eqref{e^a_k z bfd_z k-1 -> bfd'_k-1} yields
\begin{equation}\label{e^a_k bar z bfdelta_z k -> bfdelta'_k}
e^{a_k\bar z}\bfdelta_{z,k}=\bfdelta'_k+O(e^{-c\mu})\quad(\mu\to+\infty)\;.
\end{equation}
Let $\bfdelta'=\sum_k\bfdelta'_k=(\bfd')^*$, which satisfies $(\bfdelta')^2 = 0$, and let
\begin{gather*}
\bfD'=\bfd'+\bfdelta'\;,\quad
\bfDelta'=(\bfD')^2=\bfd'\bfdelta'+\bfdelta'\bfd'\;.
\end{gather*}
We have
\begin{gather*}
\bfC^\bullet=\ker\bfDelta'\oplus\im\bfd'\oplus\im\bfdelta'\;,\\
\im\bfDelta'=\im\bfD'=\im\bfd'\oplus\im\bfdelta'\;,\quad
\ker\bfDelta'=\ker\bfD'=\ker\bfd'\cap\ker\bfdelta'\;.
\end{gather*}
The orthogonal projections of $\bfC^\bullet$ to $\ker\bfDelta'$, $\im\bfd'$ and $\im\bfdelta'$ are denoted by $\bfPi'=\bfPi^{\prime\,0}$, $\bfPi^{\prime\,1}$ and $\bfPi^{\prime\,2}$, respectively. Like in \Cref{sss: perturbations,sss: bfDelta_z}, the composition $(\bfd')^{-1}\bfPi^{\prime\,1}$ is defined on $\bfC^\bullet$. From~\eqref{e^a_k z bfd_z k-1 -> bfd'_k-1} and~\eqref{e^a_k bar z bfdelta_z k -> bfdelta'_k}, we easily get that, as $\mu\to+\infty$,
\begin{gather}
\bfPi^j_{z,k}=\bfPi^{\prime\,j}_k+O(e^{-c\mu})\quad(j=0,1,2)\;,\label{bfPi^j_z k = ...}\\
e^{-a_kz}(\bfd_{z,k-1})^{-1}\bfPi^1_{z,k}=(\bfd'_{k-1})^{-1}\bfPi^{\prime\,1}_k+O(e^{-c\mu})\;.
\label{e^-a_kz bfd_z_k-1^-1 bfPi^1_z_k = ...}
\end{gather}
By~\eqref{e^a_k z bfd_z k-1 -> bfd'_k-1} and~\eqref{e^a_k bar z bfdelta_z k -> bfdelta'_k}, on $\im\bfdelta_{z,k}\oplus\im\bfd_{z,k-1}$,
\begin{equation}\label{bfDelta_z}
\bfDelta_z=e^{-2a_k\mu}\bfDelta'+O(e^{-(2a_k+c)\mu})\quad(\mu\to+\infty)\;.
\end{equation}

\begin{prop}\label{p: eigenvalues of bfDelta_z}
For $k=0,\dots,n$ and $\mu\gg0$, the spectrum of $\bfDelta_z$ on $\im\bfdelta_{z,k}\oplus\im\bfd_{z,k-1}$ is contained in an interval of the form
\[
\big[Ce^{-2a_k\mu},C'e^{-2a_k\mu}\big]\quad(C'\ge C)\;.
\]
\end{prop}

\begin{proof}
The positive eigenvalues of $\bfDelta'$ are contained in an interval $[C_0,C'_0]$ ($C'_0\ge C_0>0$). By~\eqref{bfDelta_z}, for $\mu\gg0$ and $\bfe\in\im\bfdelta_{z,k}\oplus\im\bfd_{z,k-1}$,
\begin{align*}
\langle\bfDelta_z\bfe,\bfe\rangle
&\ge e^{2a_k\mu}\langle\bfDelta'\bfe,\bfe\rangle-C_1e^{-(2a_k+c)\mu}\|\bfe\|^2
\ge\big(C_0e^{-2a_k\mu}-C_1e^{-(2a_k+c)\mu}\big)\|\bfe\|^2\;,\\
\langle\bfDelta_z\bfe,\bfe\rangle
&\le e^{2a_k\mu}\langle\bfDelta'\bfe,\bfe\rangle+C_1e^{-(2a_k+c)\mu}\|\bfe\|^2
\le\big(C'_0e^{-2a_k\mu}+C_1e^{-(2a_k+c)\mu}\big)\|\bfe\|^2\;.
\end{align*}
Then result follows taking $0<C<C_0$ and $C'>C'_0$.
\end{proof}

\subsection{Estimates of the nonzero small spectrum}\label{ss: nonzero small spectrum}

\begin{thm}\label{t: estimates of the small eigenvalues}
If $\mu\gg0$, the spectrum of $\Delta_{z,\text{\rm sm}}$ on $\im\delta_{z,\text{\rm sm},k}\oplus\im d_{z,\text{\rm sm},k-1}$ is contained in an interval of the form 
\[
[C\mu e^{-2a_k\mu},C'\mu e^{-2a_k\mu}]\quad(C'\ge C)\;.
\]
\end{thm}

\begin{proof}
By the commutativity of~\eqref{CD im delta_z k+1 ...}, for every eigenvalue $\lambda$ of $\Delta_{z,\text{\rm sm}}$ on $\im\delta_{z,\text{\rm sm},k}\oplus\im d_{z,\text{\rm sm},k-1}$, there are normalized $\lambda$-eigenforms, $e\in\im\delta_{z,\text{\rm sm},k}$ and $e'\in\im d_{z,\text{\rm sm},k-1}$, so that $d_ze=\lambda^{1/2}e'$ and $\delta_ze'=\lambda^{1/2}e$. So the maximum and minimum of the spectrum of $\Delta_{z,\text{\rm sm}}$ on $\im\delta_{z,\text{\rm sm},k}\oplus\im d_{z,\text{\rm sm},k-1}$ is $\|d_{z,\text{\rm sm},k-1}\|^2$ and $\|d_{z,\text{\rm sm},k-1}^{-1}\Pi^1_{z,\text{\rm sm},k}\|^{-2}$, respectively. Similarly, the maximum and minimum of the spectrum of $\bfDelta_z$ on $\im\bfdelta_{z,k}\oplus\im\bfd_{z,k-1}$ is $\|\bfd_{z,k-1}\|^2$ and $\|\bfd_{z,k-1}^{-1}\bfPi^1_{z,k}\|^{-2}$, respectively. Then the result follows from \Cref{c: Phi_z^* Phi_z,c: (Phi_z^-1)^* Phi_z^-1,c: d_z sm = ...,p: eigenvalues of bfDelta_z}:
\begin{gather*}
\begin{aligned}
\|d_{z,\text{\rm sm},k-1}\|^2
&\le\|\Phi_{z,k}^{-1}\|^2\|\bfd_{z,k-1}\|^2\|\Phi_{z,k-1}\|^2\\
&\le\Big(\Big(\frac\mu\pi\Big)^{k-\frac n2}+O\big(e^{-c\mu}\big)\Big)C'_0e^{-2a_k\mu}
\Big(\Big(\frac\pi\mu\Big)^{k-1-\frac n2}+O\big(e^{-c\mu}\big)\Big)\\
&\le C'\mu e^{-2a_k\mu}\;,
\end{aligned}\\
\begin{aligned}
\|d_{z,\text{\rm sm},k-1}^{-1}\Pi^1_{z,\text{\rm sm},k-1}\|^{-2}
&\ge\|\Phi_{z,k-1}^{-1}\|^{-2}\|\bfd_{z,k-1}^{-1}\bfPi^1_{z,k}\|^{-2}\|\Phi_{z,k}\|^{-2}\\
&\ge\Big(\Big(\frac\pi\mu\Big)^{k-1-\frac n2}+O\big(e^{-c\mu}\big)\Big)C_0e^{-2a_k\mu}
\Big(\Big(\frac\mu\pi\Big)^{k-\frac n2}+O\big(e^{-c\mu}\big)\Big)\\
&\ge C\mu e^{-2a_k\mu}\;.\;\qedhere
\end{aligned}
\end{gather*}
\end{proof}

\subsection{Asymptotics of the small zeta invariant}\label{ss: sm zeta inv}

\Cref{t: zeta_sm/la(1 z)}~\ref{i: zetasm(1 z)} is proved here.

\begin{thm}\label{t: eta wedge d_z^-1 Pi^1_z sm k}
As $\mu\to+\infty$,
\[
{\eta\wedge}\,d_z^{-1}\Pi^1_{z,\text{\rm sm},k}\asymp_1\big(1-e^{a_k}\big)\Pi^1_{z,\text{\rm sm},k}\;.
\]
\end{thm}

\begin{proof}
Consider the notation of \Cref{ss: small complex vs Morse complex,ss: bfd'}. By \Cref{c: Pi^1_z sm,c: widetilde bfPi^j_z}, for $\mu\gg0$,
\begin{equation}\label{Pi^1_z sm - 2}
\Pi^1_{z,\text{\rm sm}}\asymp_0\widetilde\Psi_z\widetilde\bfPi^1_z\Phi_{z,\text{\rm sm}}
=\widetilde\Psi_z\bfPi^1_z\Phi_{z,\text{\rm sm}}\;.
\end{equation}

For brevity, let $S_z=\Phi_z\widetilde\Psi_{z-1}$ and $T_z=\Phi_{z-1}P_{z-1,\text{\rm sm}}\widetilde\Psi_z$ on $\bfC^\bullet$. By \Cref{c: Phi_z^* Phi_z,c: Pi^1_z sm}, 
\begin{equation}\label{S_z T_z asymp_2 1}
S_z,T_z\asymp_11\;.
\end{equation}
Moreover, by \Cref{p: P_z sm P_z+tau sm,c: Phi_z+tau widetilde Psi_z}, and the definitions of $\Psi_z$ and $\widetilde\Psi_z$, considered as maps $\bfC^\bullet\to L^2(M;\Lambda)$, we get 
\begin{multline}\label{widetilde Psi_z S_z}
\widetilde\Psi_zS_z=\widetilde\Psi_z\Phi_zP_{z-1,\text{\rm sm}}\widetilde\Psi_{z-1}
\asymp_1\widetilde\Psi_z\Phi_zP_{z,\text{\rm sm}}\widetilde\Psi_{z-1}\\
\asymp_1P_{z,\text{\rm sm}}\widetilde\Psi_{z-1}\asymp_1P_{z-1,\text{\rm sm}}\widetilde\Psi_{z-1}=\widetilde\Psi_{z-1}\;.
\end{multline}

By~\eqref{e^a_k z bfd_z k-1 -> bfd'_k-1},~\eqref{bfPi^j_z k = ...}, \eqref{e^-a_kz bfd_z_k-1^-1 bfPi^1_z_k = ...},~\eqref{Pi^1_z sm - 2}--\eqref{widetilde Psi_z S_z}, \Cref{p: eigenvalues of bfDelta_z,c: Phi_z+tau widetilde Psi_z,c: Phi_z+tau P_z+tau sm widetilde Psi_z,c: widetilde Psi_z^* widetilde Psi_z,c: Phi_z^* Phi_z,c: d_z+tau sm - d_z+tau P_z sm,c: Pi^1_z sm,c: Phi_z^-1,c: widetilde bfPi^j_z,c: d_z sm = ...,t: estimates of the small eigenvalues}, 
\begin{align*}
e^{a_k}\Pi^1_{z,\text{\rm sm},k}&\asymp_0 e^{a_k}\widetilde\Psi_z\bfPi^1_{z,k}\Phi_{z,\text{\rm sm}}
\asymp_1 e^{a_k}\widetilde\Psi_z\bfPi^{\prime\,1}_k\Phi_{z,\text{\rm sm}}\\
&=e^{a_k}\widetilde\Psi_z\bfd'_{k-1}(\bfd'_{k-1})^{-1}\bfPi^{\prime\,1}_k\Phi_{z,\text{\rm sm}}\\
&\asymp_1 e^{a_k}\widetilde\Psi_zS_z\bfd'_{k-1}T_z(\bfd'_{k-1})^{-1}\bfPi^{\prime\,1}_k\Phi_{z,\text{\rm sm}}\\
&\asymp_1 e^{a_k}\widetilde\Psi_{z-1}\bfd'_{k-1}T_z(\bfd'_{k-1})^{-1}\bfPi^{\prime\,1}_k\Phi_{z,\text{\rm sm}}\\
&\asymp_1 e^{a_k}\widetilde\Psi_{z-1}e^{a_k(z-1)}\bfd_{z-1,k-1}T_z
e^{-a_kz}\bfd_{z,k-1}^{-1}\bfPi^1_{z,k}\Phi_{z,\text{\rm sm}}\\
&=\widetilde\Psi_{z-1}\bfd_{z-1,k-1}T_z\bfd_{z,k-1}^{-1}\bfPi^1_{z,k}\Phi_{z,\text{\rm sm}}\\
&\asymp_0\widetilde\Psi_{z-1}\bfd_{z-1,k-1}T_z\widetilde\bfPi^2_{z,k-1}
\bfd_{z,k-1}^{-1}\widetilde\bfPi^1_{z,k}\Phi_{z,\text{\rm sm}}\\
&\asymp_0\widetilde\Psi_{z-1}\bfd_{z-1,k-1}\Phi_{z-1}P_{z-1,\text{\rm sm}}\Pi^2_{z,k-1}\widetilde\Psi_z\bfd_{z,k-1}^{-1}\Phi_z\Pi^1_{z,\text{\rm sm}}\\
&\asymp_0\widetilde\Psi_{z-1}\bfd_{z-1,k-1}\Phi_{z-1}P_{z-1,\text{\rm sm}}\Pi^2_{z,k-1}\Phi_z^{-1}\bfd_{z,k-1}^{-1}\Phi_z\Pi^1_{z,\text{\rm sm}}\\
&=\widetilde\Psi_{z-1}\Phi_{z-1}d_{z-1,\text{\rm sm},k-1}d_{z,\text{\rm sm},k-1}^{-1}\Pi^1_{z,\text{\rm sm},k}
\asymp_0 d_{z-1}d_z^{-1}\Pi^1_{z,\text{\rm sm},k}\;.
\end{align*}
Therefore
\[
{\eta\wedge}\,d_z^{-1}\Pi^1_{z,\text{\rm sm},k}=(d_z-d_{z-1})d_z^{-1}\Pi^1_{z,\text{\rm sm},k}
\asymp_1(1-e^{a_k})\Pi^1_{z,\text{\rm sm},k}\;.\qedhere
\]
\end{proof}

\Cref{t: zeta_sm/la(1 z)}~\ref{i: zetasm(1 z)} follows from \Cref{c: P_z sm : E_z -> E_z sm iso,c: zetala(1 z) = ...,t: eta wedge d_z^-1 Pi^1_z sm k}.

\begin{rem}
\Cref{t: zeta_sm/la(1 z)}~\ref{i: zetasm(1 z)} agrees with \Crefrange{c: zeta_sm/la(1 Delta_z theta wedge D_z bfw) in R}{c: zetasm(1 Delta_z theta wedge D_z bfw) is uniformly bd} by~\eqref{-bfz(M g -eta)}.
\end{rem}

\section{Prescription of the asymptotics of the zeta invariant}\label{s: prescription}

We prove \Cref{t: prescription of lim zeta(1 z)} here. By \Cref{t: extension of Smale}, given $a\gg0$, there is some $\eta_0\in\xi$ and some metric $g$ satisfying~\ref{i-a:  eta Morse} and~\ref{i-a:  tight} with the given $X$, and so that $\MM_k(\eta_0,X)=a$ for all $k=1,\dots,n$. Using the notation of \Cref{ss: Morse forms}, we are going to modify $\eta_0$ only in every $U_p$ for $p\in\YY_0\cup\YY_n$.

Fix any $\epsilon>0$ such that, for every $p\in\YY_0\cup\YY_n$, the open ball $B(p,3\epsilon)$ is contained in $U_p$. Let
\[
V=\bigcup_{p\in\YY_0\cup\YY_n}B(p,\epsilon)\;,\quad V'=\bigcup_{p\in\YY_0\cup\YY_n}B(p,2\epsilon)\;.
\]
Take a smooth function $\sigma:[0,3\epsilon]\to[0,1]$ so that
\[
\sigma'\le0\;,\quad\sigma([0,\epsilon])=1\;,\quad\sigma([2\epsilon,3\epsilon])=0\;.
\]
Let $f_j\in C^\infty(M,\R)$ ($j=0,n$) be the extension by zero of the combination of the functions $\sigma(|x_p|)\in\Cinftyc(B(p,3\epsilon),\R)$ ($p\in\YY_j$). We have 
\[
\supp df_j\subset V'_j\setminus V_j\;,\quad f_j(V_j)=1\;,\quad f_j(M\setminus V'_j)=0\;,\quad
Xf_0\ge0\;,\quad Xf_n\le0\;.
\]

For any $c_0,c_n\ge0$, let $\eta=\eta(c_0,c_n)=\eta_0-c_0\,df_0+c_n\,df_n$. This closed 1-form satisfies~\ref{i-a:  eta Morse} and~\ref{i-a:  tight} with $X$ and $g$, and we have
\[
\MM_1(\eta,X)=a+c_0\;,\quad\MM_n(\eta,X)=a+c_n\;,\quad
\MM_k(\eta_1,X)=a\quad(1<k<n)\;. 
\]
Hence, by \Cref{c: m^j_z k only depends on | XX_k | and xi},
\begin{equation}\label{Z_ sm(theta)}
\bfz_{\text{\rm sm}}(\eta)-\bfz_{\text{\rm sm}}(\eta_0)
=e^a(e^{c_0}-1)m^1_1+(-1)^ne^a(1-e^{c_n})m^1_n\;.
\end{equation}
By~\ref{i-a:  eta Morse}, $e(M,\nabla^M)=0$ on every $U_p$ ($p\in\YY$). So, using the Stokes formula,
\begin{align}
\bfz_{\text{\rm la}}(\eta)-\bfz_{\text{\rm la}}(\eta_0)
&=\int_M(c_n\,df_n-c_0\,df_0)\wedge (-X)^*\psi(M,\nabla^M)\notag\\
&=\int_M(c_0f_0-c_nf_n)\,(-X)^*d\psi(M,\nabla^M)\notag\\
&=\int_M(c_0f_0-c_nf_n)\,e(M,\nabla^M)-\sum_{p\in\YY}(-1)^{\ind(p)}(c_0f_0-c_nf_n)(p)\notag\\
&=(-1)^nc_n|\YY_n|-c_0|\YY_0|\;,\label{Z_ la(theta)}
\end{align}
Combining~\eqref{Z_ sm(theta)} and~\eqref{Z_ la(theta)}, we obtain
\[
\bfz(\eta)-\bfz(\eta_0)
=e^a(e^{c_0}-1)m^1_1+(-1)^ne^a(1-e^{c_n})m^1_n+(-1)^nc_n|\YY_n|-c_0|\YY_0|\;.
\]
Using local changes of $X$ and applying \cite[Lemmas~1.1 and~1.2]{Smale1961}, we can increase $|\YY_0|$ or $|\YY_n|$ as much as desired. By \Cref{l: m_z k = |XX_k| - beta_z^k} and~\eqref{m_z k^j}, we have
\begin{equation}\label{m^1_1 = |XX_0| - betaNo^0}
m^1_1=|\YY_0|-\betaNo^0\;,\quad m^1_n=|\YY_n|-\betaNo^n\;,
\end{equation}
which can be increased as much as desired. So, if $n$ is even (resp., odd), given any $\tau\in\R$ (resp., $\tau\gg0$), we get $\bfz(\eta(c_0,c_n))=\tau$ for some $c_0,c_n\ge0$.

Now assume $n$ is even. To prove that $\pm\bfz(\pm\eta)=\tau$, by~\eqref{-bfz_la(-eta)},~\eqref{Z_ sm(theta)} and~\eqref{Z_ la(theta)}, it is enough to prove that we can choose $|\YY_0|$, $|\YY_n|$, $c_0$ and $c_n$ so that
\begin{gather*}
\bfz_{\text{\rm sm}}(\eta)=\bfz_{\text{\rm sm}}(\eta_0)
+e^a(e^{c_0}-1)m^1_1+e^a(1-e^{c_n})m^1_n=0\;,\\
\bfz_{\text{\rm la}}(\eta)=\bfz_{\text{\rm la}}(\eta_0)
+c_n|\YY_n|-c_0|\YY_0|=\tau\;.
\end{gather*}
Using~\eqref{m^1_1 = |XX_0| - betaNo^0}, and writing $u=-e^{-a}\bfz_{\text{\rm sm}}(\eta_0)$ and $v=\tau-\bfz_{\text{\rm la}}(\eta_0)$, the above system becomes 
\begin{gather*}
(e^{c_0}-1)(|\YY_0|-\betaNo^0)+(1-e^{c_n})(|\YY_n|-\betaNo^n)=u\;,\\
c_n|\YY_n|-c_0|\YY_0|=v\;.
\end{gather*}
The following result states that these equalities are satisfied by some $c_0,c_n\ge0$ and $|\YY_0|,|\YY_n|\gg0$.

\begin{lem}\label{l: sols - system}
Given $u,v\in\R$ and $\beta,\gamma\ge0$, there are $c,d\ge0$ and integers $p,q\gg0$ such that
\begin{gather*}
(e^c-1)(p-\beta)+(1-e^d)(q-\gamma)=u\;,\\
dq-cp=v\;.
\end{gather*}
\end{lem}

\begin{proof}
Taking $q>0$, we get
\[
d=(cp+v)/q\;.
\]
Thus $cp+v\ge0$; i.e., $c\ge-v/p$. Let
\[
F_{p,q}(c)=(e^c-1)(p-\beta)+\big(1-e^{(cp+v)/q}\big)(q-\gamma)\;.
\]
We have to find integers $p,q\gg0$ and $c\ge0,-v/p$ such that $F_{p,q}(c)=u$.

Observe that
\begin{align}
\beta<p<q&\Rightarrow \lim_{c\to+\infty}F_{p,q}(c)=+\infty\;,\label{F_p,q(c) -> +infty}\\
\gamma<q<p&\Rightarrow \lim_{c\to+\infty}F_{p,q}(c)=-\infty\;.\label{F_p,q(c) -> -infty}
\end{align}
Note also that, if $(c,d,p,q)$ is a solution for some $(u,v, \beta, \gamma)$, then $(d,c,q,p)$ is a solution for $(-u,-v, \gamma,\beta)$. So it is sufficient to consider the case $v\geq 0$. In this case, $c$ can reach $0$ and
\[
F_{p,q}(0)=\big(1-e^{v/q}\big)(q-\gamma)\;,
\]
which is independent of $p$. Choose $q\gg\beta,\gamma$; thus $F_{p,q}(0)\le0$. If $u\ge F_{p,q}(0)$, take $p$ so that $\beta\ll p<q$, yielding $u\in\im F_{p,q}$ by~\eqref{F_p,q(c) -> +infty}. If $u<F_{p,q}(0)$, take $p>q$, yielding $u\in\im F_{p,q}$ by~\eqref{F_p,q(c) -> -infty}.
\end{proof}

\section{The switch of the order of integration}\label{s: switch}

The proof of \Cref{t: Z_pm} is given in this section. Let $\SS$ be the Schwartz space on $\R$. Recall that the space of tempered distributions is the continuous dual space $\SS'$, with the strong topology. Suppose first that~\eqref{Z_mu switching integrals} is used as definition of $Z_\mu$. By \Cref{t: zeta(1 z),t: zeta_sm/la(1 z)}, the expression~\eqref{Z_mu switching integrals} defines a tempered distribution $Z_\mu$ for $\mu\gg0$. Moreover, using also the formula of the inverse Fourier transform, we get, for $f\in\SS$, 
\[
\langle Z_\mu,f\rangle=\frac{1}{2\pi}\int_{-\infty}^\infty\zeta(1,z)\,\hat f(\nu)\,d\nu
\to\frac{\bfz}{2\pi}\int_{-\infty}^\infty\hat f(\nu)\,d\nu
=\bfz f(0)\;,
\]
as $\mu\to+\infty$, uniformly on $\nu$. For every $C>0$, this convergence is also uniform on $f\in\SS$ with $|\hat f(\nu)|,|\nu^2\hat f(\nu)|\le C$. So $Z_\mu\to\bfz\delta_0$ in $\SS'$ as $\mu\to+\infty$. To get \Cref{t: Z_pm}, it only remains to prove the following.

\begin{thm}\label{p: switch}
Both~\eqref{Z_mu} and~\eqref{Z_mu switching integrals} define the same tempered distribution $Z_\mu$ for $\mu\gg0$. 
\end{thm}

\begin{prop}\label{p: finite integral of the absolute value}
For $\mu\gg0$, $t>0$ and $f\in\SS$, 
\[
\int_{-\infty}^\infty\int_t^\infty\big|\Str\big({\eta\wedge}\,\delta_ze^{-u\Delta_z}\big)\big|\,|\hat f(\nu)|\,du\,d\nu<\infty\;.
\]
\end{prop}

\begin{proof}
By \cite[Corollary~XI.9.8 and Lemma~XI.9.9~(d)]{DunfordSchwartz1988-II},
\begin{align*}
\big|\Str\big({\eta\wedge}\,\delta_{z}e^{-u\Delta_z}\big)\big|
&\le\big|{\eta\wedge}\,\delta_{z}e^{-u\Delta_z}\big|_1
\le\|{\eta\wedge}\|\,\big|\delta_{z}e^{-u\Delta_z}\big|_1\\
&=\|\eta\|_{L^\infty}\,\Tr\big((d_{z}\delta_{z})^{1/2}e^{-u\Delta_z}\big)
\le\|\eta\|_{L^\infty}\,\Tr\big(\Delta_{z}^{1/2}e^{-u\Delta_z}\big)\;,
\end{align*}
where $|\ |_1$ denotes the trace norm. Hence 
\begin{align*}
\int_t^\infty\big|\Str\big({\eta\wedge}\,\delta_ze^{-u\Delta_z}\big)\big|\,du
&\le\|\eta\|_{L^\infty}\int_t^\infty \Tr\big(\Delta_{z}^{1/2}e^{-u\Delta_z}\big)\,du\\
&=\|\eta\|_{L^\infty} \Tr\big(\Delta_{z}^{-1/2}e^{-t\Delta_z}\Pi_z^\perp\big)\;.
\end{align*}

The operator $(1+\Delta)^{-N}$ is of trace class for any $N>n$. Therefore
\[
\Tr\big(\Delta_{z}^{-1/2}e^{-t\Delta_z}\Pi_z^\perp\big)
\le \big|(1+\Delta)^{-N}\big|_1 \big\|(1+\Delta)^{N}\Delta_{z}^{-1/2}e^{-t\Delta_z}\Pi_z^\perp\big\|\;.
\]
By \Cref{c: | alpha |_m z,t: estimates of the small eigenvalues}, for $\mu\gg0$ and $\alpha\in L^2(M;\Lambda)$, 
\begin{multline*}
\big\|(1+\Delta)^{N}\Delta_{z}^{-1/2}e^{-t\Delta_z}\Pi_z^\perp\alpha \big\|\\
\begin{aligned}
&\le C_0 \big\|\Delta_{z}^{-1/2}e^{-t\Delta_z}\Pi_z^\perp\alpha \big\|_{2N}
\le C_1|z|^{2N} \big\|\Delta_{z}^{-1/2}e^{-t\Delta_z}\Pi_z^\perp\alpha \big\|_{2N,z}\\ 
&= C_2 |z|^{2N} \sum_{k=0}^{2N}\big\|D_z^k\Delta_{z}^{-1/2}e^{-t\Delta_z}\Pi_z^\perp\alpha \big\|
\le C_3|z|^{2N}\sum_{k=0}^{2N}\frac1{t^{k/2}}\big\|\Delta_{z}^{-1/2}\Pi_z^\perp\alpha \big\|\\
&\le C|z|^{2N}\big(1+t^{-N}\big)e^{c \mu}\|\alpha\|\;.
\end{aligned}
\end{multline*}
Thus, since $f\in\SS$,
\begin{multline*}
\int_{-\infty}^\infty\int_t^\infty\big|\Str\big({\eta\wedge}\,\delta_ze^{-u\Delta_z}\big)\big|\,|\hat f(\nu)|\,du\,d\nu\\
\le C\|\eta\|_{L^\infty}\big|(1+\Delta)^{-N}\big|_1\big(1+t^{-N}\big)e^{c \mu}
\int_{-\infty}^\infty|z|^{2N}|\hat f(\nu)|\,d\nu<\infty\;.\qedhere
\end{multline*}
\end{proof} 

\begin{proof}[Proof of \Cref{p: switch}]
We compute
\begin{multline*}
-\frac{1}{2\pi}\int_{-\infty}^\infty\lim_{t\downarrow0}\Str\left({\eta\wedge}\,d_z^{-1}e^{-t\Delta_z}\Pi_z^1\right)\,\hat f(\nu)\,d\nu\\
\begin{aligned}
&=-\frac{1}{2\pi}\lim_{t\downarrow0}\int_{-\infty}^\infty\Str\big({\eta\wedge}\,d_z^{-1}e^{-t\Delta_z}\Pi_z^1\big)\,\hat f(\nu)\,d\nu\\
&=\frac{1}{2\pi}\lim_{t\downarrow0}\int_{-\infty}^\infty\int_t^\infty\Str\big({\eta\wedge}\,\delta_ze^{-u\Delta_z}\big)\,\hat f(\nu)\,du\,d\nu\\
&=\frac{1}{2\pi}\lim_{t\downarrow0}\int_t^\infty\int_{-\infty}^\infty\Str\big({\eta\wedge}\,\delta_ze^{-u\Delta_z}\big)\,\hat f(\nu)\,d\nu\,du\\
&=\frac{1}{2\pi}\int_0^\infty\int_{-\infty}^\infty\Str\big({\eta\wedge}\,\delta_ze^{-u\Delta_z}\big)\,\hat f(\nu)\,d\nu\,du\;.\;\end{aligned}
\end{multline*}
Here, the first equality is given by the Lebesgue's dominated convergence theorem, whose hypothesis is satisfied because $\hat f\in\SS$ and $|\Str({\eta\wedge}\,d_z^{-1}e^{-t\Delta_z}\Pi_z^1)|\le C$ for all $t>0$, $|\mu|\gg0$ and $\nu\in\R$ by \Cref{t: zeta(1 z),t: zeta_sm/la(1 z)}. The third equality is given by Fubini's theorem, whose hypothesis is satisfied by \Cref{p: finite integral of the absolute value}.
\end{proof}

\appendix

\section{Integrals along instantons}\label{s: integrals along instantons}

\Cref{t: extension of Smale} is proved here. We show the case where $a_n\ge\dots\ge a_1\gg0$. Then the case where $a_1\ge\dots\ge a_n\gg0$ follows by using $-X$ and $-\xi$.

By \cite[Theorem~B]{Smale1961}, there is some Morse function $h$ on $M$ such that $h(\YY_k)=\{k\}$ ($k=0,\dots,n$), $Xh<0$ on $M\setminus\YY$, and $h$ is in standard form with respect to $X$; in particular, $\Crit_k(h)=\YY_k$. Now we proceed like in the proof of \cite[Proposition~16~(i)]{BurgheleaHaller2004}. Since $\YY$ is finite, there is some $\eta'\in\xi$ such that $\eta'=0$ on some open neighborhood $U_p$ of every $p\in\YY$. Let $U_k=\bigcup_{p\in\YY_k}U_p$ and $U=\bigcup_kU_k$. We can assume $h(U_k)\subset(k-1/4,k+1/4)$ for all $k=0,\dots,n$. If $C\gg0$, then the representative $\eta'':=\eta'+C\,dh$ of $\xi$ satisfies $\eta''(X)<0$ on $M\setminus\YY$.

For $k=0,\dots,n$, let $I_k^\pm\subset\R$ be the closed interval with boundary points $k\pm1/4$ and $k\pm1/2$. Since there are no critical values of $h$ in $I_k^\pm$, every $T_k^\pm:=h^{-1}(I_k^\pm)$ is compact submanifold with boundary of dimension $n$, every $\Sigma_k^\pm:=h^{-1}(k\pm1/2)$ is a closed submanifold of codimension $1$, and there are identities $T_k^\pm\equiv\Sigma_k^\pm\times I_k^\pm$ given by $x\equiv(\pi_k^\pm(x),h(x))$ ($x\in T_k^\pm$), where $\pi_k^\pm(x)$ is the unique point of $\Sigma_k^\pm$ that meets the $\phi$-orbit of $x$. Of course, $\Sigma_k^-=\Sigma_{k-1}^+$ ($k=1,\dots,n$) and $T_0^-=\Sigma_0^-=T_n^+=\Sigma_n^+=\emptyset$. (See Figure~\ref{f: pants}.)

\begin{figure}[tb]
	\centering
	\includegraphics[scale=0.4]{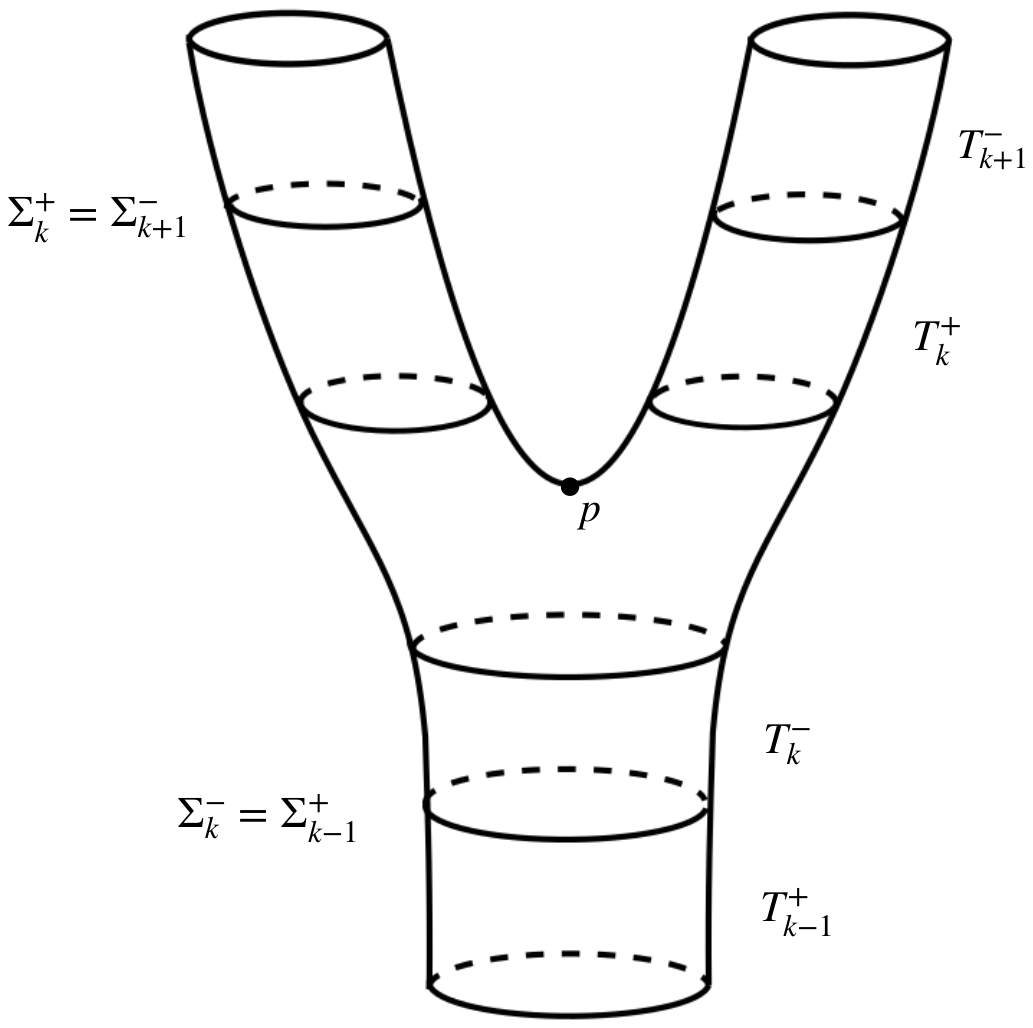}
	\caption{A representation of the sets $T_k^\pm$, $\Sigma_k^\pm$, $T_{k-1}^+$ and $T_{k+1}^-$, taking $\YY_k=\{p\}$.}
	\label{f: pants}
\end{figure}

We have $\Sigma_k^\pm\pitchfork\iota_p^\pm(W^\pm_p)$ for $p\in\YY_k$. Let $K_p^\pm=\Sigma_k^\pm\cap \iota_p^\pm(W^\pm_p)$ and $K_k^\pm=\bigcup_{p\in\YY_k}K_p^\pm$, which are closed submanifolds of $\Sigma_k^\pm$;  $K_k^-$ is of codimension $k$ in $\Sigma_k^-$, and $K_k^+$ of codimension $n-k$ in $\Sigma_k^+$. Since the $\alpha$- and $\omega$-limits of the orbits of $X$ are zero points, the orbit of $\phi$ through every point $x\in\Sigma_k^+\setminus K_k^+$ meets $\Sigma_k^-\setminus K_k^-$ at a unique point $\psi_k(x):=\phi^{\tau_k(x)}(x)$ ($\tau_k(x)>0$). This defines a diffeomorphism $\psi_k:\Sigma_k^+\setminus K_k^+\to\Sigma_k^-\setminus K_k^-$ and a smooth function $\tau_k:\Sigma_k^+\setminus K_k^+\to\R^+$. Moreover the sets $K_p^\pm$ ($p\in\YY_k$) have corresponding open neighborhoods $V_p^\pm$ in $\Sigma_k^\pm$, with disjoint closures, such that $\psi_k(V_p^+\setminus K_p^+)=V_p^-\setminus K_p^-$. Take smooth functions $\lambda_p^\pm$ ($p\in\YY_k)$ on $\Sigma_k^\pm$ so that $0\le\lambda_p^\pm\le1$, $\supp\lambda_p^\pm\subset V_p^\pm$, $\lambda_p^\pm=1$ on $K_p^\pm$, and $\lambda_p^+=\psi_k^*\lambda_p^-$ on $\Sigma_k^+\setminus K_k^+$. Moreover let
\begin{align*}
\widetilde T_k&=h^{-1}([k-1/2,k+1/2])\;,\quad
\widetilde K_p=\widetilde T_k\cap\big(\iota_p^+(W^+_p)\cup\iota_p^-(W^-_p)\big)\;,\\
\widetilde V_p&=\{\,\phi^t(x)\mid x\in V_p^+\setminus K_p^+,\ 0\le t\le\tau_k(x)\,\}\cup\widetilde K_p\;,\\
\widetilde K_k&=\bigcup_{p\in\YY_k}\widetilde K_p\;,\quad
\widetilde V_k=\bigcup_{p\in\YY_k}\widetilde V_p\;,\quad
M_k=h^{-1}((-\infty,k+1/2])\;.
\end{align*}
Thus $M_k=\widetilde T_0\cup\dots\cup\widetilde T_k$. Note that $\widetilde T_k$ and $M_k$ are compact submanifolds with boundary of dimension $n$, and every $\widetilde V_p$ (resp., $\widetilde K_p$) is open (resp., closed) in $\widetilde T_k$. We also get smooth functions $\tilde\lambda_p$ ($p\in\YY_k$) on $\widetilde T_k$ determined by the condition $\tilde\lambda_p(\phi^t(x))=\lambda_p^+(x)$ for all $x\in\Sigma_k^+\setminus K_k^+$ and $0\le t\le\tau_k(x)$. They satisfy $0\le\tilde\lambda_p\le1$, $\supp\tilde\lambda_p\subset\widetilde V_p$, and $\tilde\lambda_p=1$ on $\widetilde K_p$. 

Let
\begin{gather*}
A_p=\max\{\,|\eta'(\gamma)|\mid\gamma\in\TT^1_p\,\}\quad\big(p\in\YY_+\big)\;,\\
A_k=\max_{p\in\YY_k}A_p\quad(k=1,\dots,n)\;,\quad A=\max\{A_1,\dots,A_n\}\;.
\end{gather*}
We can suppose $C>A$ and $a_1>C+A>0$. For $p\in\YY_k$, $q\in\YY_{k-1}$ and $\gamma\in\TT(p,q)$,
\[
dh(\gamma)=h(q)-h(p)=-1\;.
\]
Therefore
\begin{equation}\label{gamma(theta'')}
0>\eta''(\gamma)=\eta'(\gamma)+C\,dh(\gamma)\ge-A-C>-a_1\quad(\gamma\in\TT^1)\;.
\end{equation}

\begin{claim}\label{cl: f_k}
For $k=0,\dots,n$, there is a smooth function $f_k$ on $M$ such that
\begin{align}
df_k(X)&\le0\;,\label{df_k(X) le 0}\\
\supp df_k&\subset\mathring M_k\;,\label{supp df_k}\\
\max\{\,(\eta''+df_k)(\gamma)\mid\gamma\in\TT^1_p\,\}&=-a_l\quad(p\in\YY_l,\ 1\le l\le k)\;,
\label{min gamma(theta'' + df_k) | gamma in TT^1_p}\\
(\eta''+df_k)(\delta)&>-a_k\quad(\delta\in\TT^1_{k+1})\;.\label{(theta'' + df_k)(delta)}
\end{align}
\end{claim}

The statement follows directly from \Cref{cl: f_k} taking $\eta=\eta''+df_n$. So we only have to prove this assertion.

We proceed by induction on $k$. For $k=0$, we choose $f_0=0$. Then~\eqref{min gamma(theta'' + df_k) | gamma in TT^1_p} is vacuous,~\eqref{df_k(X) le 0} and~\eqref{supp df_k} are trivial, and~\eqref{(theta'' + df_k)(delta)} is given by~\eqref{gamma(theta'')}. 

Now take any $k\ge1$ and assume $f_{k-1}$ is defined and satisfies~\eqref{df_k(X) le 0}--\eqref{(theta'' + df_k)(delta)}. Let
\begin{align}
b_p&=-\max\{\,(\eta''+df_{k-1})(\gamma)\mid\gamma\in\TT^1_p\,\}\quad(p\in\YY_k)\;,\label{b_p}\\
b_k&=\min\{\,b_p\mid p\in\YY_k\,\}\;.\notag
\end{align}
For every $p\in\YY_k$, we have $b_p<a_{k-1}\le a_k$ because $f_{k-1}$ satisfies~\eqref{(theta'' + df_k)(delta)}. So there is a smooth function $h_p^-$ on $I_k^-$ such that $(h_p^-)'\ge0$, $h_p^-=0$ around $k-1/2$, and $h_p^-=a_k-b_p$ around $k-1/4$. Let $\tilde h_p^-$ be the function on $V_p^-\times I_k^-\subset\Sigma_k^-\times I_k^-\equiv T_k^-$ given by $\tilde h_p^-(x,s)=h_p^-(s)$. We have $\tilde h_p^-=0$ around $V_p^-\times\{k-1/2\}$ and $\tilde h_p^-=a_k-b_p$ around $V_p^-\times\{k-1/4\}$. Thus $\tilde h_p^-$ has a smooth extension to $\widetilde V_p$, also denoted by $\tilde h_p^-$, which is equal to $a_k-b_p$ on $\widetilde V_p\setminus T_k^-$. The function $\tilde\lambda_p\tilde h_p^-$ on $\widetilde V_p$ can be extended by zero to get a smooth function on $\widetilde T_k$, also denoted by $\tilde\lambda_p\tilde h_p^-$. Let $\tilde h_k^-=\sum_{p\in\YY_k}\tilde\lambda_p\tilde h_p^-$ on $\widetilde T_k$.

On the other hand, let $\rho_k$ be a smooth function on $I_k^+$ such that $\rho_k'\ge0$, $\rho_k=0$ around $k+1/4$, and $\rho_k=1$ around $k+1/2$. Let $\tilde\rho_k$ be the smooth function on $T_k^+\equiv\Sigma_k^+\times I_k^+$ given by $\tilde\rho_k(x,s)=\rho_k(s)$, and let
\[
\tilde h_k^+=\tilde h_k^-(1-\tilde\rho_k)+(a_k-b_k)\tilde\rho_k
\]
on $T_k^+$. This smooth function is equal to $\tilde h_k^-$ around $\Sigma_k^+\times\{k+1/4\}$, and is equal to $a_k-b_k$ around $\Sigma_k^+\times\{k+1/2\}\equiv\Sigma_k^+$. So the functions, $\tilde h_k^-$ on $\widetilde T_k\setminus T_k^+$ and $\tilde h_k^+$ on $T_k^+$, can be combined to produce a smooth function $\tilde h_k$ on $\widetilde T_k$. Since $\tilde h_k=0$ around $\Sigma_k^-$ and $\tilde h_k=a_k-b_k$ around $\Sigma_k^+$, there is a smooth extension of $\tilde h_k$ to $M$, also denoted by $\tilde h_k$, which is constant on $M\setminus\widetilde T_k$.

Let $f_k=f_{k-1}+\tilde h_k$ on $M$. This smooth function satisfies~\eqref{df_k(X) le 0} because $f_{k-1}$ satisfies~\eqref{df_k(X) le 0}, and $X$ induces the opposite of the standard orientation on every fiber $\{x\}\times I_k^\pm\equiv I_k^\pm$ of $T_k^\pm$ ($x\in\Sigma_k^\pm$). It also satisfies~\eqref{supp df_k} and~\eqref{min gamma(theta'' + df_k) | gamma in TT^1_p} for $p\in\YY_l$ with $1\le l<k$ because $f_{k-1}$ satisfies these properties and $d\tilde h_k$ is supported in the interior of $\widetilde T_k$.

Next, take any $p\in\YY_k$, $q\in\YY_{k-1}$ and $\gamma\in\TT(p,q)\subset\TT^1_p$. We have $\gamma\cap T_k^-\equiv\{x\}\times I_k^-$ for some $x\in K_p^-\cap K_q^+\subset\Sigma_k^-=\Sigma_{k-1}^+$, and the orientation of $\gamma\cap T_k^-$ agrees with the opposite of the standard orientation of $\{x\}\times I_k^-\equiv I_k^-$. Then
\begin{align*}
(\eta''+df_k)(\gamma)&=(\eta''+df_{k-1}+d\tilde h_k)(\gamma)\le-b_p+\lambda_p^-(x)d\tilde h_p^-(\gamma)\\
&=-b_p-\int_{I_k^-}dh_p^-=-b_p-(a_k-b_p)=-a_k\;.
\end{align*}
Here, the equality holds when the maximum of~\eqref{b_p} is achieved at $\gamma$. Hence $f_k$ also satisfies~\eqref{min gamma(theta'' + df_k) | gamma in TT^1_p} for $p\in\YY_k$.

Finally, take any $p\in\YY_k$, $u\in\YY_{k+1}$ and $\delta\in\TT(u,p)\subset\TT^1_u\subset\TT^1_{k+1}$. Thus $\delta\cap T_k^+\equiv\{y\}\times I_k^+$ for some $y\in K_p^+\cap K_u^-\subset\Sigma_k^+=\Sigma_{k+1}^-$, and the orientation of $\delta\cap T_k^+$ agrees with the opposite of the standard orientation of $\{y\}\times I_k^+\equiv I_k^+$. Then
\begin{align*}
(\eta''+df_k)(\delta)&=(\eta''+df_{k-1}+d\tilde h_k)(\delta)
=\eta''(\delta)+d\tilde h_k^+(\delta)\\
&=\eta''(\delta)+\big(\tilde h_k^-(y)-(a_k-b_k)\big)\int_{I_k^-}d\rho_k\\
&=\eta''(\delta)+\tilde\lambda_p(y)\tilde h_p^-(y)+b_k-a_k\\
&=\eta''(\delta)+a_k-b_p+b_k-a_k=\eta''(\delta)+b_k-b_p
\ge\eta''(\delta)>-a_k\;,
\end{align*}
where the second equality is true because $f_{k-1}$ satisfies~\eqref{supp df_k}, and the last inequality holds by~\eqref{gamma(theta'')}. So $f_k$ satisfies~\eqref{(theta'' + df_k)(delta)}.


\bibliographystyle{amsplain}

\bibliography{../../../../Suso}

\end{document}